\documentclass[fleqn,10pt]{olplainarticle}
\usepackage[all]{xy}
\usepackage{tikz-cd}
\usepackage[mathcal]{euscript}
\newtheorem{thm}{Theorem}[section]
\newtheorem{lemma}[thm]{Lemma}
\newtheorem{theorem}{Theorem}

\theoremstyle{definition}
\newtheorem{definition}[thm]{Definition}
\newtheorem{example}[thm]{Example}

\newtheorem{prop}[thm]{Proposition}
\newtheorem{cor}[thm]{Corollary}

\theoremstyle{remark}
\newtheorem{remark}[thm]{Remark}
\title{D\'evissage for Algebraic K-theory of Small Stable $\infty$-categories}

\author[$\dagger$]{Chunhui Wei}
\affil[$\dagger$]{The University of Melbourne\newline
chunhuiw2@student.unimelb.edu.au\newline chunhuiwei@mail.ustc.edu.cn}

\begin{abstract}
In this article, we extend the theorem of heart\cite{Barwick_2015}, which implies Quillen's d\'evissage theorem by \cite{Efimov2025}, to generic small stable $\infty$-categories. To be precise, we establish a necessary and sufficient condition under when an exact functor between stable $\infty$-categories induces isomorphisms of non-negative $K$-groups when this exact functor satisfies the d\'evissage condition.
\end{abstract}

\begin{document}

\flushbottom
\maketitle
\thispagestyle{empty}
\tableofcontents
\section*{Introduction}
Algebraic $K$-theory provides fundamental tools to investigate the structure of stable $\infty$-categories, revealing deep connections between categorical decompositions and invariants arising from homotopy-theoretic contexts. In this article, we establish new results characterizing algebraic $K$-groups under conditions generalizing classical d\'evissage arguments.

Our main contributions focus on exact functors between (idempotent-complete) stable $\infty$-categories that satisfy suitable d'evissage conditions. Building upon this framework, we prove that an exact functor induces isomorphisms on all higher $K$-groups $K_n$ for $n \geq 1$ whenever it satisfies a weak d\'evissage condition and is fillable. Furthermore, if the stronger d\'evissage condition is met, the isomorphism extends to $K_0$. In addition, we establish vanishing results for higher $K$-groups of fillable stable $\infty$-categories themselves, uncovering new structural constraints.
\begin{definition}
    Let $\mathcal{F}:\mathcal{A}\to\mathcal{C}$ be an exact functor between small stable idempotent-complete $\infty$-categories.
    \begin{enumerate}[label=(\arabic*)]
        \item $\mathcal{F}$ satisfies \emph{d\'evissage condition} if $\mathcal{F}(\mathcal{A})$ generates $\mathcal{C}$ as a stable $\infty$-category.
        \item $\mathcal{F}$ satisfies \emph{weak d\'evissage condition} if $\mathcal{F}(\mathcal{A})$ generates $\mathcal{C}$ as a stable idempotent-complete $\infty$-category.
        \item $\mathcal{F}$ is \emph{$n$-fillable} if $(\mathrm{Gap}^\lor_{(-)})^k\mathcal{F}$ satisfies the d\'evissage condition for all $1\leq k\leq n$, where
        \[
        \mathrm{Gap}^\lor_{(-)}:\mathrm{Fun}(\Delta^1,\mathrm{Cat}^\mathrm{perf})\to\mathrm{Fun}(\Delta^1,\mathrm{Cat}^\mathrm{perf})
        \]
        is a functor sending exact functors to exact functors.
    \end{enumerate}
\end{definition}
\begin{theorem}[Theorem \ref{mainresult1}]\label{thmA}
     Let $\mathcal{F}:\mathcal{A}\to\mathcal{C}$ be an exact functor between small stable idempotent-complete $\infty$-categories. If $\mathcal{F}$ satisfies the weak d\'evissage condition(resp. d\'evissage condition), then the following statements are equivalent:
    \begin{enumerate}[label=(\arabic*)]
        \item $\mathcal{F}$ is $n$-fillable, and 
        \item $\mathcal{F}$ induces isomorphisms $K_i(\mathcal{A})\xrightarrow{\simeq}K_i(\mathcal{C})$ for all $1\leq i\leq n-1$, an epimorphism $K_n(\mathcal{A})\twoheadrightarrow K_n(\mathcal{C})$ and a monomorphism (resp. an isomorphism) $K_0(\mathcal{A})\hookrightarrow K_0(\mathcal{C})$.
    \end{enumerate}
\end{theorem}
\begin{definition}
    A small idempotent-complete $\infty$-category $\mathcal{C}$ is \emph{$n$-fillable} if the diagonal map $\delta_\mathcal{C}:\mathcal{C}\to\mathcal{C}^2,X\mapsto (X,X)$ is $n$-fillable.
\end{definition}
\begin{theorem}[Theorem \ref{mainresult2}]\label{thmB}
     Let $\mathcal{C}\in\mathrm{Cat}^{\mathrm{perf}}$ be a small idempotent-complete $\infty$-category. Then $K_i(\mathcal{C})=0$ for all $1\leq i\leq n$ if and only if $\mathcal{C}$ is $n$-fillable.
\end{theorem}

Furthermore, we have a D\'evissage-type Theorem for general localizing invariants.

\begin{theorem}[Theorem \ref{mainresult3}]\label{thmC}
    Suppose $\mathfrak{S}$ is a fillable class of some exact functors and $L:\mathrm{Cat}^\mathrm{perf}\to\mathrm{Sp}$ is a localizing invariant such that $L_m$ is $\mathfrak{S}$-epimorphic for some $m\in\mathbb{Z}$. Then
      \begin{enumerate}[label=(\arabic*)]
          \item $L_m$ is $\mathfrak{S}_\mathrm{devissage}^w\cap\mathfrak{S}_\mathrm{1-fillability}\cap(\mathrm{Gap}^\lor_{(-)})^{-1}\mathfrak{S}$-monomorphic.
          \item $L_m$ is $\mathfrak{S}_\mathrm{devissage}\cap\mathfrak{S}_\mathrm{1-fillability}\cap\mathfrak{S}$-invariant.
          \item $L_{m+n}$ is $\mathfrak{S}_\mathrm{devissage}^w\cap\mathfrak{S}_\mathrm{n-fillability}\cap(\mathrm{Gap}^\lor_{(-)})^{-1}\mathfrak{S}$-epimorphic for every $n\geq 1$.
          \item $L_{m+k}$ is $\mathfrak{S}_\mathrm{devissage}^w\cap\mathfrak{S}_\mathrm{n-fillability}\cap(\mathrm{Gap}^\lor_{(-)})^{-1}\mathfrak{S}$-invariant for every $0<k<n$.
      \end{enumerate}
\end{theorem}
\begin{remark}
    In Section 4, we explain the details for the notation in this theorem. In short, it means in some cases, epimorphisms between lower homotopy groups imply isomorphisms between higher homotopy groups. The following "the theorem of heart" is just an example.
\end{remark}
Using it, we make some remarks on the theorem of heart.
\begin{definition}
    A small \emph{$t$-category} is a small stable $\infty$-category $\mathcal{C}$ with a bounded $t$-structure $(\mathcal{C}_{\ge 0}, \mathcal{C}_{\le 0})$. An $t$-exact functor $F:\mathcal{A}\to\mathcal{C}$ between small $t$-categories is called \emph{coconnective} if
        \begin{itemize}
            \item $F(\mathcal{A})$ generates $\mathcal{C}$ as a stable idempotent-complete $\infty$-categories(it satisfies the weak d\'evissage condition), and
            \item $F|_{\mathcal{A}^\heartsuit}:\mathcal{A}^\heartsuit\to\mathcal{C}^\heartsuit$ is fully faithful.
        \end{itemize}
\end{definition}
\begin{cor}(Corollary \ref{cor:heart})
    Let $L:\mathrm{Cat}^\mathrm{perf}\to\mathrm{Sp}$ be a localizing invariant and $m\in\mathbb{Z}$. Then the following statements are equivalent:
    \begin{enumerate}[label=(\arabic*)]
        \item for every coconnective functor $\mathcal{F}$ between small $t$-categories, $\pi_mL(\mathcal{F})$ is an epimorphism.
        \item for every coconnective functor $\mathcal{F}$ between small $t$-categories, $\pi_{m+n}L(\mathcal{F})$ is an isomorphism for every $n\geq 0$.
    \end{enumerate}
\end{cor}

Overview of this article:
\begin{itemize}
    \item In Section 1, we introduce basic definitions and properties of stable \(\infty\)-categories.
    \item  In Section 2, we formalize the (weak) d\'evissage condition for an exact functor. It concerns whether the objects in the image are enough to generate the whole $\infty$-category.
    \item In Section 3, we develop the categorification of fibers and loops for universal localizing invariant. It is used to describe the higher $K$-groups of our target by $K_0$-groups of some other $\infty$-categories.
    \item In Section 4, we discuss the fillability of exact functors and present the main results concerning algebraic $K$-theory and generic localizing invariants. This condition determines whether the class of cofiber sequences in the image are enough to describe the class of cofiber sequences in the whole $\infty$-category.
    \item In Section 5, we show how to use our results to reprove the theorem of heart.
\end{itemize} 

\section{Preliminaries on Stable $\infty$-category and exact functors}
This section is a recollection and some remarks of \cite[Section 4,5]{Hoyois_2017}, which is a refinement and generalization of \cite{blumberg_gepner_tabuada_2013}.
\subsection{Stable $\infty$-category}
Beyond ordinary categories, the landscape of category theory expands with the introduction of higher categories, including the concept of $\infty$-categories, which play an important role in $K$-theory. Bergner \cite{bergner_2006} wrote a survey of different models of $\infty$-categories. The most widely used model, due to Joyal \cite{Joyal_2002} and Lurie \cite{lurie_2009}, defines $\infty$-categories as \textit{weak Kan complexes} (or \textit{quasi-categories}).
\begin{definition}\cite[Definition 1.1.1.4]{lurie_2017}\quad\\
A \textit{triangle} in a pointed $\infty$-category $\mathcal{C}$ is a functor $\Delta^1\times \Delta^1\to\mathcal{C}$
\[
\begin{tikzcd}
X \arrow[r, "f"] \arrow[d] & Y \arrow[d, "g"] \\
0 \arrow[r] & Z
\end{tikzcd}
\]
A \textit{fiber sequence} in $\mathcal{C}$ is a pullback triangle, and a \textit{cofiber sequence} in $\mathcal{C}$ is a pushout triangle. We will abuse notation by writing $X\to Y\to Z$ for a triangle.
\end{definition}

\begin{definition}\cite[Definition 1.1.1.6]{lurie_2017}\quad\\
Given a morphism $g: X \to Y$ in a pointed $\infty$-category $\mathcal{C}$. A \textit{fiber} of $g$ is a fiber sequence
\[
\begin{tikzcd}
W \arrow[r] \arrow[d] & X \arrow[d, "g"] \\
0 \arrow[r] & Y
\end{tikzcd}
\]
Dually, a \textit{cofiber} of $g$ is a cofiber sequence
\[
\begin{tikzcd}
X \arrow[r, "g"] \arrow[d] & Y \arrow[d] \\
0 \arrow[r] & Z
\end{tikzcd}
\]
We also refer to $W$ and $Z$ simply as the fiber and cofiber of $g$, respectively, and write $W = \mathrm{fib}(g)$ and $Z = \mathrm{cofib}(g)$.
\end{definition}

\begin{definition}\cite[Definition 1.1.1.9, Proposition 1.1.3.4]{lurie_2017}\quad\\
A \textit{stable} $\infty$-category is a pointed $\infty$-category satisfying:
\begin{itemize}
    \item It is finite complete and finite cocomplete.
    \item A square is a pushout if and only if it is a pullback.
\end{itemize}
Or, equivalently:
\begin{itemize}
    \item Every morphism admits both a fiber and a cofiber.
    \item A triangle is a fiber sequence if and only if it is a cofiber sequence.
\end{itemize}
\end{definition}
\begin{definition}
Let $\mathcal{C}$ be an $\infty$-category.
    \begin{enumerate}[label=(\arabic*)]
        \item Two morphisms $f,g:X\to Y$ are \emph{(homotopy) equivalent}, denote by $f\simeq g$, if $[f]=[g]$ in $\pi_0\mathrm{Map}(X,Y)$.
        \item Two objects $X,Y$ are \emph{(homotopy) equivalent}, denote by $X\simeq Y$, if there are two morphisms $f:X\to Y,g:Y\to X$ such that $f\circ g\simeq \mathrm{id}_Y$ and $g\circ f\simeq\mathrm{id}_X$. $f,g$ are called \emph{(homotopy) equivalences}.
    \end{enumerate}
\end{definition}

We display some properties of stable $\infty$-categories firstly.
\begin{lemma}\label{lem:diagram1}
Suppose \( X' \to X \to X'' \), \( Y' \to Y \to Y'' \), and \( Z' \to Z \to Z'' \) are three cofiber sequences in a stable \( \infty \)-category \( \mathcal{C} \).
\begin{enumerate}[label=(\arabic*)]
    \item Given a functor \( f: \Delta^1 \times \Delta^1 \to \mathcal{C} \) represented by the diagram
    \[
    \begin{tikzcd}
        X' \ar[r] \ar[d] & X \ar[d] \\
        Y' \ar[r] & Y
    \end{tikzcd},
    \]
    there is a unique morphism \( X'' \to Y'' \) extending \( f \) to a functor \( \Delta^2 \times \Delta^1 \to \mathcal{C} \) of the form
    \[
    \begin{tikzcd}
        X' \ar[r] \ar[d] & X \ar[d] \ar[r] & X'' \ar[d] \\
        Y' \ar[r] & Y \ar[r] & Y''
    \end{tikzcd}.
    \]

    \item Given a functor \( f: \Delta^1 \times \Delta^1 \to \mathcal{C} \) represented by the diagram
    \[
    \begin{tikzcd}
        X \ar[r] \ar[d] & X'' \ar[d] \\
        Y \ar[r] & Y''
    \end{tikzcd},
    \]
    there is a unique morphism \( X' \to Y' \) extending \( f \) to a functor \( \Delta^2 \times \Delta^1 \to \mathcal{C} \) of the form
    \[
    \begin{tikzcd}
        X' \ar[r] \ar[d] & X \ar[d] \ar[r] & X'' \ar[d] \\
        Y' \ar[r] & Y \ar[r] & Y''
    \end{tikzcd}.
    \]

    \item Given a functor \( f: \Delta^1 \times \Delta^2 \to \mathcal{C} \) represented by the diagram
    \[
    \begin{tikzcd}
        X' \ar[r] \ar[d] & X \ar[d] \\
        Y' \ar[r] \ar[d] & Y \ar[d] \\
        Z' \ar[r] & Z
    \end{tikzcd}
    \]
    where the two columns are cofiber sequences, there are unique morphisms \( X'' \to Y'' \) and \( Y'' \to Z'' \) extending \( f \) to a functor \( \Delta^2 \times \Delta^2 \to \mathcal{C} \) of the form
    \[
    \begin{tikzcd}
        X' \ar[r] \ar[d] & X \ar[d] \ar[r] & X'' \ar[d] \\
        Y' \ar[r] \ar[d] & Y \ar[d] \ar[r] & Y'' \ar[d] \\
        Z' \ar[r] & Z \ar[r] & Z''
    \end{tikzcd}
    \]
    such that \( X'' \to Y'' \to Z'' \) is a cofiber sequence.

    \item Given a functor \( f: \Delta^1 \times \Delta^2 \to \mathcal{C} \) represented by the diagram
    \[
    \begin{tikzcd}
        X \ar[r] \ar[d] & X'' \ar[d] \\
        Y \ar[r] \ar[d] & Y'' \ar[d] \\
        Z \ar[r] & Z''
    \end{tikzcd}
    \]
    where the two columns are cofiber sequences, there are unique morphisms \( X' \to Y' \) and \( Y' \to Z' \) extending \( f \) to a functor \( \Delta^2 \times \Delta^2 \to \mathcal{C} \) of the form
    \[
    \begin{tikzcd}
        X' \ar[r] \ar[d] & X \ar[d] \ar[r] & X'' \ar[d] \\
        Y' \ar[r] \ar[d] & Y \ar[d] \ar[r] & Y'' \ar[d] \\
        Z' \ar[r] & Z \ar[r] & Z''
    \end{tikzcd}
    \]
    such that \( X' \to Y' \to Z' \) is a cofiber sequence.
\end{enumerate}
\end{lemma}
\begin{proof}
    (1) and (2) are from the universal property of cofibers. Besides, (1) and (2) give the existence of functors $\Delta^2\times \Delta^2\to\mathcal{C}$ such that all rows are cofiber sequences. The last column is cofiber sequence is because $\mathrm{Ho}(\mathcal{C})$ is a triangulated category whose distinguished triangles are cofiber sequences \cite[Theorem 1.1.2.14]{lurie_2017}. 
\end{proof}

\begin{lemma}\label{lem:cofiberofpullback}\cite[Lemma 2.21]{2019InMat.216..241A}
Given a functor \( \Delta^1 \times \Delta^1 \to \mathcal{C} \) represented by the diagram
\[
\begin{tikzcd}
M \arrow[r, "f"] \arrow[d] & N \arrow[d] \\
P \arrow[r, "g"] & Q
\end{tikzcd},
\]
the induced map \( \operatorname{cofib}(f) \to \operatorname{cofib}(g) \) is an equivalence if and only if this square is a pushout square.
\end{lemma}

\begin{lemma}\cite[Lemma 4.4.2.1]{lurie_2009}\label{lem:diagram3}
Let \( \mathcal{C} \) be an \( \infty \)-category and suppose we are given a functor \( \Delta^2 \times \Delta^1 \to \mathcal{C} \) which we will depict as a diagram
\[
\begin{tikzcd}
X \arrow[r] \arrow[d] & Y \arrow[r] \arrow[d] & Z \arrow[d] \\
X' \arrow[r] & Y' \arrow[r] & Z'
\end{tikzcd}
\]
Suppose that the left square is a pushout in \( \mathcal{C} \). Then the right square is a pushout if and only if the outer square is a pushout.
\end{lemma}

\begin{definition}
An \textit{exact} functor between stable \(\infty\)-categories is a functor preserving zero objects and fiber sequences.
\end{definition}

\begin{prop}\cite[Proposition 1.1.4.1]{lurie_2017}\quad\\
A functor between stable \(\infty\)-categories is exact if and only if it preserves finite limits and colimits.
\end{prop}
The collection \(\mathrm{Cat}_{\infty}^{\mathrm{ex}} \subseteq \mathrm{Cat}_{\infty}\) of stable \(\infty\)-categories and exact functors is a subcategory.

\begin{thm}\cite[Theorem 1.1.4.4,Proposition 1.1.4.6]{lurie_2017}
    The $\infty$-category $\mathrm{Cat}^\mathrm{ex}$ admits all finite limits and filtered colimits.
\end{thm}
\begin{definition}
    An $\infty$-functor between $\infty$-categories is called \emph{left exact} if it preserves all finite limits.
\end{definition}
\begin{lemma}\cite[Corollary]{lurie_2009}\label{lem:leftexact1}
    An $\infty$-functor $F:\mathcal{C}\to \mathcal{D}$ out of an $\infty$-category $\mathcal{C}$ that has all finite colimits is left exact if and only if it preserves pullbacks and the terminal object.
\end{lemma}
An $\infty$-category is \textit{idempotent-complete} if the image of this category under the Yoneda embedding is closed with respect to retracts. We denote the category of small idempotent-complete stable $\infty$-categories as $\mathrm{Cat}^{\mathrm{perf}}$. The inclusion $\mathrm{Cat}^{\mathrm{perf}} \to \mathrm{Cat}^{\rm ex}$ has a left adjoint, referred to as $\mathrm{Idem}:\mathrm{Cat}^{\rm ex} \to \mathrm{Cat}^{\mathrm{perf}}$, which is known as \textit{idempotent completion} \cite{lurie_2009}. Hence, the inclusion $\mathrm{Cat}^{\mathrm{perf}} \to \mathrm{Cat}^{\rm ex}$ is left exact.

\subsection{$\mathcal{E}$-linear categories}
\begin{definition}
     Denote by
     \begin{itemize}
         \item $\mathcal{P}\mathrm{r}^\mathrm{L}_\mathrm{St}$, the $\infty$-category of stable presentable $\infty$-categories and left adjoint functors.
         \item $\mathrm{Fun}^\mathrm{ex}(\mathcal{A},\mathcal{B})$, the full subcategory of $\mathrm{Fun}(\mathcal{A},\mathcal{B})$ spanned by exact functors.
         \item $\mathrm{Fun}^\mathrm{L}(\mathcal{A},\mathcal{B})$, the full subcategory of $\mathrm{Fun}(\mathcal{A},\mathcal{B})$ spanned by left adjoint functors.
     \end{itemize}
\end{definition}
\begin{definition}
    Let $\mathcal{A}$ be a small, stable, and idempotent complete $\infty$-category. For $X,Y\in\mathcal{A}$, denote by $\mathrm{Map}^\mathrm{Sp}_\mathcal{A}(X,Y)$ the mapping spectrum of $X,Y$.
\end{definition}
From \cite[Section 3.1]{blumberg_gepner_tabuada_2013}, $\mathrm{Cat}^\mathrm{perf}$ and $\mathcal{P}\mathrm{r}^\mathrm{L}_\mathrm{St}$ admit (closed) symmetric monoidal structures such that $\mathrm{Ind}:\mathrm{Cat}^\mathrm{perf}\to\mathcal{P}\mathrm{r}^\mathrm{L}_\mathrm{St}$ is a symmetric monoidal functor. Their tensor products will be denoted by $\otimes$ and $\otimes_\mathrm{L}$, respectively. The universal properties are:
\begin{itemize}
    \item Given $\mathcal{A},\mathcal{B},\mathcal{C}\in \mathrm{Cat}^\mathrm{perf}$, we have equivalence of $\infty$-categories
    \[
    \mathrm{Fun}^\mathrm{ex}(\mathcal{A}\otimes\mathcal{B},\mathcal{C})\simeq\mathrm{Fun}^\mathrm{ex}(\mathcal{A},\mathrm{Fun}^\mathrm{ex}(\mathcal{B},\mathcal{C})).
    \]
    \item Given $\mathcal{A},\mathcal{B},\mathcal{C}\in \mathrm{Cat}_\mathrm{St}^\mathrm{L}$, we have equivalence of $\infty$-categories
    \[
    \mathrm{Fun}^\mathrm{L}(\mathcal{A}\otimes^\mathrm{L}\mathcal{B},\mathcal{C})\simeq\mathrm{Fun}^\mathrm{L}(\mathcal{A},\mathrm{Fun}^\mathrm{L}(\mathcal{B},\mathcal{C})).
    \]
\end{itemize}
\begin{definition}
     Let $\mathcal{E}\in\mathrm{CAlg}(\mathrm{Cat}^\mathrm{perf})$ be a commutative algebra in $\mathrm{Cat}^\mathrm{perf}$, i.e., a small, stable, and idempotent complete symmetric monoidal $\infty$-category whose tensor product $\mathcal{E} \times \mathcal{E} \to \mathcal{E}$ is exact in each variable. We set
\begin{itemize}
\item $\mathrm{Cat}^\mathrm{perf}(\mathcal{E}):=\operatorname{Mod}_{\mathcal{E}}(\mathrm{Cat}^\mathrm{perf})$.
\item $\mathcal{P}\mathrm{r}^\mathrm{L}(\mathcal{E}) := \operatorname{Mod}_{\operatorname{Ind}(\mathcal{E})}(\mathcal{P}\mathrm{r}^\mathrm{L}_{\mathrm{St}})$.
\end{itemize}
\end{definition}
\begin{remark}
    For the definition of commutative algebra and its module, refer to \cite[Section 4.5]{lurie_2017}
\end{remark}
\begin{remark}
    $\mathrm{Cat}^\mathrm{perf}(\mathrm{Sp}^\omega)=\mathrm{Cat}^\mathrm{perf},\mathcal{P}\mathrm{r}^\mathrm{L}_\mathrm{St}(\mathrm{Sp}^\omega)=\mathcal{P}\mathrm{r}^\mathrm{L}_\mathrm{St}$ 
\end{remark}
Recall to \cite[Theorem 4.5.2.1]{lurie_2017}, the tensor product $\otimes$ and the internal Hom object $\mathrm{Fun}^\mathrm{ex}(-,-)$ induces the tensor product $\otimes_\mathcal{E}$ and the internal Hom object $\mathrm{Fun}^\mathrm{ex}_\mathcal{E}(-,-)$ on $\mathrm{Cat}^\mathrm{perf}(\mathcal{E})$. Similarly, the tensor product $\otimes^\mathrm{L}$ and the internal Hom object $\mathrm{Fun}^\mathrm{L}(-,-)$ induces the tensor product $\otimes^\mathrm{L}_\mathcal{E}$ and the internal Hom object $\mathrm{Fun}^\mathrm{L}_\mathcal{E}(-,-)$ on $\mathrm{Cat}^\mathrm{perf}(\mathcal{E})$. And $\mathrm{Ind}$ lifts to a symmetric monoidal functor
\[
\mathrm{Ind}:\mathrm{Cat}^\mathrm{perf}(\mathcal{E})\to\mathcal{P}\mathrm{r}^\mathrm{L}(\mathcal{E}).
\]
Objects in $\mathrm{Cat}^\mathrm{perf}(\mathcal{E})$ are called \emph{$\mathcal{E}$-linear} $\infty$-categories or \emph{$\mathcal{E}$-module} $\infty$-categories. Objects in $\mathrm{Fun}^\mathrm{ex}_\mathcal{E}(-,-)$ are called \emph{$\mathcal{E}$-linear} functors or \emph{$\mathcal{E}$-module} functors.
\begin{prop}\cite[Proposition 4.7]{Hoyois_2017}
    For every $\mathcal{E}\in\mathrm{CAlg}(\mathrm{Cat}^\mathrm{perf})$, $\mathrm{Cat}^\mathrm{perf}(\mathcal{E})$ is compactly generated.
\end{prop}
If $\mathcal{E}$ is a commutative algebra in $\mathrm{Cat}^\mathrm{perf}$ and $\mathcal{A} \in \mathrm{Cat}^\mathrm{perf}(\mathcal{E})$, then $\mathcal{A}$ is naturally enriched over $\operatorname{Ind}(\mathcal{E})$. Indeed, given $X \in \mathcal{A}$, the functor $\mathcal{E} \to \mathcal{A}$ sending $E$ to $E \otimes X$ preserves finite colimits and hence admits an ind-right adjoint $\mathrm{Map}_\mathcal{A}^\mathcal{E}(X, -): \mathcal{A} \to \operatorname{Ind}(\mathcal{E})$. For $X,Y\in\mathcal{A},E\in\mathcal{E}$,
    \[
    \mathrm{Map}^{\mathrm{Sp}}_{\mathcal{E}}(E, \mathrm{Map}_{\mathcal{A}}^{\mathcal{E}}(X, Y)) \simeq \mathrm{Map}^{\mathrm{Sp}}_{\mathcal{A}}(E \otimes X, Y).
    \]
    As a consequence,
    \[
    \pi_n\mathrm{Map}_{\mathcal{A}}^{\mathcal{E}}(X, Y)\simeq \pi_n\mathrm{Map}_{\mathcal{A}}^\mathrm{Sp}(X, Y)
    \]
    for all $n\in\mathbb{Z}$.

One can hence define a functor
\[
\mathcal{A} \to \operatorname{Fun}^\mathrm{ex}(\mathcal{A}^\mathrm{op}, \operatorname{Ind}(\mathcal{E}))
\]
given informally by $a \mapsto \mathrm{Map}_\mathcal{A}^\mathcal{E}(-, a)$. When $\mathcal{E} = \mathrm{Sp}^\omega$, the functor is fully faithful, but this is not true for more general $\mathcal{E}$.

\begin{definition}
    Let $\mathcal{E}$ be a commutative algebra in $\mathrm{Cat}^\mathrm{perf}$. $\mathcal{E}$ is called \emph{rigid} if every object of $\mathcal{E}$ is dualizable.
\end{definition}
In that case, we will show that $\mathcal{A} \to \operatorname{Fun}^\mathrm{ex}(\mathcal{A}^\mathrm{op}, \operatorname{Ind}(\mathcal{E}))$ factors through a fully faithful $\mathcal{E}$-linear embedding $\mathcal{A} \hookrightarrow \operatorname{Fun}^\mathrm{ex}_\mathcal{E}(\mathcal{A}^\mathrm{op}, \operatorname{Ind}(\mathcal{E}))$.

We denote by $\mathrm{CAlg}^\mathrm{rig}(\mathrm{Cat}^\mathrm{perf}) \subset \mathrm{CAlg}(\mathrm{Cat}^\mathrm{perf})$ the full subcategory spanned by the rigid symmetric monoidal $\infty$-categories.
\begin{prop}\cite[Proposition 4.9]{Hoyois_2017}\quad\\
    Let $\mathcal{E}\in \mathrm{CAlg}^\mathrm{rig}(\mathrm{Cat}^\mathrm{perf})$. Then
    \begin{enumerate}[label=(\arabic*)]
\item There is a canonical equivalence of symmetric monoidal $\infty$-categories
\[
\mathcal{E} \simeq \mathcal{E}^\mathrm{op}, \quad e \mapsto e^\vee = \mathrm{Map}_{\mathcal{E}}^{\mathcal{E}}(e, \mathbf{1}).
\]
\item For any $\mathcal{A} \in \mathcal{P}\mathrm{r}^\mathrm{L}(\mathcal{E})$, the action of $\mathcal{E}$ on $\mathcal{A}$ restricts to the full subcategory $\mathcal{A}^\omega \subset \mathcal{A}$ of compact objects. In particular, if $\mathcal{A}$ is compactly generated, then it belongs to the essential image of the functor $\mathrm{Ind}:\mathrm{Cat}^\mathrm{perf}(\mathcal{E}) \to \mathcal{P}\mathrm{r}^\mathrm{L}(\mathcal{E})$.

\item For any $\mathcal{E}$-module $\infty$-categories $\mathcal{A}$ and $\mathcal{B}$ and any $F \in \mathrm{Fun}^\mathrm{L}_\mathcal{E}(\mathcal{A}, \mathcal{B})$, the right adjoint $G: \mathcal{B} \to \mathcal{A}$ of $F$ has a canonical structure of $\mathcal{E}$-module functor, and the unit $\mathrm{id}_\mathcal{A} \to G \circ F$ and counit $F \circ G \to \mathrm{id}_\mathcal{B}$ have canonical structures of $\mathcal{E}$-module natural transformations.

\item Let $\mathcal{A}$ and $\mathcal{B}$ be arbitrary $\mathcal{E}$-module $\infty$-categories. Then there is a canonical equivalence of $\mathcal{E}$-module $\infty$-categories
\[
\mathrm{Fun}^\mathrm{L}_\mathcal{E}(\mathcal{A}, \mathcal{B}) \simeq \mathrm{Fun}^\mathrm{R}_\mathcal{E}(\mathcal{B}, \mathcal{A})^\mathrm{op},
\]
sending a left adjoint $\mathcal{E}$-module functor to its right adjoint.
\end{enumerate}
\end{prop}
\subsection{Additive invariants and localizing invariants}
\begin{definition}\cite[Definition 5.12, Proposition 5.15]{blumberg_gepner_tabuada_2013}\quad\\
     An \textit{exact sequence} of small stable idempotent-complete $\infty$-categories $\mathcal{A}\to\mathcal{B}\to\mathcal{C}$ is a sequence where the composite morphism is zero, the functor $\mathcal{A} \to \mathcal{B}$ is fully faithful, and the morphism from $\mathcal{B}/\mathcal{A}$ to $\mathcal{C}$ constitutes an equivalence after idempotent-completion. The map $\mathcal{B}\to\mathcal{C}$ is called a \emph{Verdier projection}.
\end{definition}
\begin{definition}\cite[Definition 5.18]{blumberg_gepner_tabuada_2013}\quad\\
    A \textit{split-exact} sequence of small stable $\infty$-categories is an exact sequence $\mathcal{A}\xrightarrow{f}\mathcal{B}\xrightarrow{g}\mathcal{C}$ with exact functors $i:\mathcal{B}\to\mathcal{A}$ and $j:\mathcal{C}\to\mathcal{B}$, which serve as right adjoints to $f$ and $g$, respectively, such that $i\circ f\simeq\mathrm{Id}$ and $g\circ j\simeq\mathrm{Id}$.
\end{definition}
\begin{definition}\cite[Definition 5.3]{Hoyois_2017}\quad\\
Let $\mathcal{E}\in\mathrm{CAlg}^\mathrm{rig}(\mathrm{Cat}^\mathrm{perf})$. A sequence
\[
\mathcal{A} \xrightarrow{f} \mathcal{B} \xrightarrow{g} \mathcal{C}
\]
in $\mathrm{Cat}^\mathrm{perf}(\mathcal{E})$ is called \emph{exact} (resp. \emph{split exact}) if its image by the forgetful functor $\mathrm{Cat}^\mathrm{perf}(\mathcal{E}) \to \mathrm{Cat}^\mathrm{perf}$ is exact (resp. split exact).
\end{definition}
\begin{definition}\cite[Definition 5.11]{Hoyois_2017}
    Let $\mathcal{X}$ be a stable presentable $\infty$-category and let $L: \mathrm{Cat}^\mathrm{perf}(\mathcal{E}) \to \mathcal{X}$ be a functor. We say that $L$ is an \emph{additive invariant} if the following conditions are satisfied:
\begin{enumerate}
\item $L$ preserves filtered colimits.
\item $L$ preserves zero objects.
\item $L$ sends split exact sequences in $\mathrm{Cat}^\mathrm{perf}(\mathcal{E})$ to cofiber sequences in $\mathcal{X}$.
\end{enumerate}
We denote by $\mathrm{Fun}_\mathrm{add}(\mathrm{Cat}^\mathrm{perf}(\mathcal{E}), \mathcal{X})$ the $\infty$-category of additive invariants with values in $\mathcal{X}$.
\end{definition}
\begin{definition}\cite[Definition 5.16]{Hoyois_2017}
    Let $\mathcal{X}$ be a stable presentable $\infty$-category and let $L: \mathrm{Cat}^\mathrm{perf}(\mathcal{E}) \to \mathcal{X}$ be a functor. We say that $L$ is an \emph{localizing invariant} if the following conditions are satisfied:
\begin{enumerate}
\item $L$ preserves filtered colimits.
\item $L$ preserves zero objects.
\item $L$ sends exact sequences in $\mathrm{Cat}^\mathrm{perf}(\mathcal{E})$ to cofiber sequences in $\mathcal{X}$.
\end{enumerate}
We denote by $\mathrm{Fun}_\mathrm{loc}(\mathrm{Cat}^\mathrm{perf}(\mathcal{E}), \mathcal{X})$ the $\infty$-category of localizing invariants with values in $\mathcal{X}$.
\end{definition}
\begin{example}
    \cite[Section 5.2,5.3]{Hoyois_2017} constructs the universal additive invariant $\mathcal{U}_\mathrm{add}:\mathrm{Cat}^\mathrm{perf}(\mathcal{E})\to\mathrm{Mot}(\mathcal{E})$ and the universal localizing invariant $\mathcal{U}_\mathrm{loc}:\mathrm{Cat}^\mathrm{perf}(\mathcal{E})\to\mathbb{M}\mathrm{ot}(\mathcal{E})$.
\end{example}
\cite[Section 5]{blumberg_gepner_tabuada_2013} construct the connective K-theory $K^\mathrm{cn}:\mathrm{Cat}^{\mathrm{ex}}\to\mathrm{Sp}$ of small stable $\infty$-categories. Its restriction on $\mathrm{Cat}^\mathrm{perf}$ is an additive invariant. It is not a Morita invariant. The non-connective K-theory $K:\mathrm{Cat}^{\mathrm{ex}}\to\mathrm{Sp}$ of small stable $\infty$-categories, which is constructed by \cite[Section 7]{blumberg_gepner_tabuada_2013}, is a Morita invariant. Besides, its restriction on $\mathrm{Cat}^\mathrm{perf}$ is a localizing invariant.

There are many other different constructions of higher algebraic K-theory. Waldhausen\cite{Waldhausen_1985} constructed algebraic K-theory for Waldhausen categories. Also, algebraic K-theory are defined for model categories of simplicial categories and dg categories. However, they are very difficult to compute. One of the reason we choose the algebraic K-theory of $\infty$-categories is easy. Except for \cite{blumberg_gepner_tabuada_2013}, the K-theory of $\infty$-categories is also studied by \cite{barwick_2016}. Other reason is because it very extensive. \cite{Maxime_Vladimir_Christoph_2024} shows every spectrum is the k-theory of a stable $\infty$-category. So we expect that it covers any other constructions of algebraic K-theory.

\subsection{Exactness properties based on a class of functors}
\begin{definition}
    Let $\mathfrak{S}$ be a class of some $\mathcal{E}$-linear functors. A functor $H:\mathrm{Cat}^\mathrm{perf}(\mathcal{E})\to\mathrm{Ab}$ into abelian groups is called
    \begin{itemize}
        \item \emph{$\mathfrak{S}$-epimorphic}, if for every exact functor $\mathcal{F}\in\mathfrak{S}$, $H(\mathcal{F})$ is an epimorphism.
        \item \emph{$\mathfrak{S}$-monomorphic}, if for every exact functor $\mathcal{F}\in\mathfrak{S}$, $H(\mathcal{F})$ is a monomorphism.
        \item \emph{$\mathfrak{S}$-invariant}, if for every exact functor $\mathcal{F}\in\mathfrak{S}$, $H(\mathcal{F})$ is an isomorphism.
    \end{itemize}
\end{definition}
There are some obvious observations.
\begin{lemma}\label{lem:exactproperty1}
    Let $\mathfrak{S}_1\subseteq\mathfrak{S}_2$ be two classes of some $\mathcal{E}$-linear functors and $H:\mathrm{Cat}^\mathrm{perf}(\mathcal{E})\to\mathrm{Ab}$ be a functor. If $H$ is $\mathfrak{S}_2$-epimorphic(resp. monomorphic, invariant), it is $\mathfrak{S}_1$-epimorphic(resp. monomorphic, invariant).
\end{lemma}
\begin{proof}
    This is trivial.
\end{proof}
\begin{lemma}\label{lem:exactproperty2}
    Let $\{\mathfrak{S}_\alpha\}_{\alpha\in I}$ be some classes of some $\mathcal{E}$-linear functors and $H:\mathrm{Cat}^\mathrm{perf}(\mathcal{E})\to\mathrm{Ab}$ be a functor. Then
    \begin{enumerate}[label=(\arabic*)]
        \item If $H$ is $\mathfrak{S}_\alpha$-epimorphic(resp. monomorphic, invariant) for some $\alpha$, it is $\cap_{\alpha\in I}\mathfrak{S}_\alpha$-epimorphic(resp. monomorphic, invariant).
        \item If $H$ is $\mathfrak{S}_\alpha$-epimorphic(resp. monomorphic, invariant) for every $\alpha$, it is $\cup_{\alpha\in I}\mathfrak{S}_\alpha$-epimorphic(resp. monomorphic, invariant).
    \end{enumerate}
\end{lemma}
\begin{proof}\quad
    \begin{enumerate}[label=(\arabic*)]
        \item Since $\cap_{\alpha\in I}\mathfrak{S}_\alpha\subset\mathfrak{S}_\alpha$, this is from Lemma \ref{lem:exactproperty1}.
        \item For $\mathcal{F}\in \cup_{\alpha\in I}\mathfrak{S}_\alpha$, we can find some $\alpha$ such that $\mathcal{F}\in\mathfrak{S}_\alpha$. Since $H$ is $\mathfrak{S}_\alpha$-epimorphic(resp. monomorphic, invariant) for every $\alpha$, we have $H(\mathcal{F})$ is an epimorphism(resp. monomorphism, isomorphism).
    \end{enumerate}
\end{proof}
\begin{lemma}\label{lem:exactproperty3}
    Let $\mathfrak{S}$ be a class of some $\mathcal{E}$-linear functor and $H_\alpha:\mathrm{Cat}^\mathrm{perf}(\mathcal{E})\to\mathrm{Ab}$ be some functors for $\alpha\in I$. If $H_\alpha$ is $\mathfrak{S}$-epimorphic(resp. monomorphic, invariant) for every $\alpha$, $\Pi_{\alpha\in I}H_\alpha$ is $\mathfrak{S}$-epimorphic(resp. monomorphic, invariant).
\end{lemma}
\begin{proof}
    For $\mathcal{F}\in\mathfrak{S}$, Since $H_\alpha(\mathcal{F})$ is an epimorphism(resp. monomorphism, isomorphism), we have $\Pi_{\alpha\in I}H_\alpha(\mathcal{F})$ is an epimorphism(resp. monomorphism, isomorphism).
\end{proof}
Now we consider the $\infty$-category $\mathrm{Fun}(\mathrm{Cat}^\mathrm{perf}(\mathcal{E}),\mathrm{Ab})$.
\begin{definition}
    A sequence of natural transformations
    \[
    \cdots\to H_{i-1}\to H_i\to H_{i+1}\to\cdots
    \]
    in $\mathrm{Fun}(\mathrm{Cat}^\mathrm{perf}(\mathcal{E}),\mathrm{Ab})$ is called \emph{exact} if
    \[
    \cdots\to H_{i-1}(\mathcal{C})\to H_i(\mathcal{C})\to H_{i+1}(\mathcal{C})\to\cdot
    \]
    is an exact sequence of abelian groups for every $\mathcal{C}\in \mathrm{Cat}^\mathrm{perf}(\mathcal{E})$.
\end{definition}
Then we have a "Five Lemma":
\begin{lemma}[Five Lemma]\label{lem:exactproperty4}
    Let $H_1\to H_2\to H_3\to H_4\to H_5$ be an exact sequence in $\mathrm{Fun}(\mathrm{Cat}^\mathrm{perf}(\mathcal{E}),\mathrm{Ab})$ and $\mathfrak{S}$ be a class of some $\mathcal{E}$-linear functors.
    \begin{enumerate}[label=(\arabic*)]
        \item If $H_1$ is $\mathfrak{S}$-epimorphic, and $H_2,H_4$ are $\mathfrak{S}$-monomorphic, then $H_3$ is $\mathfrak{S}$-monomorphic.
        \item If $H_5$ is $\mathfrak{S}$-monomorphic, and $H_2,H_4$ are $\mathfrak{S}$-epimorphic, then $H_3$ is $\mathfrak{S}$-epimorphic.
        \item If $H_1,H_2,H_4,H_5$ are $\mathfrak{S}$-invariant, then $H_3$ is $\mathfrak{S}$-invariant.
    \end{enumerate}
\end{lemma}
\begin{proof}
    Just from the Five Lemma of abelian groups.
\end{proof}
\begin{definition}
    A natural transformation $H_1\to H_2$ in $\mathrm{Fun}(\mathrm{Cat}^\mathrm{perf}(\mathcal{E}),\mathrm{Ab})$ is an \emph{epimorphism}(resp. a \emph{monomorphism}) if $H_1(\mathcal{C})\to H_2(\mathcal{C})$ is an epimorphism(resp. a monomorphism) for every $\mathcal{C}\in \mathrm{Cat}^\mathrm{perf}(\mathcal{E})$.
\end{definition}
\begin{lemma}
    Let $H_1\to H_2$ be a natural transformation in $\mathrm{Fun}(\mathrm{Cat}^\mathrm{perf}(\mathcal{E}),\mathrm{Ab})$ and $\mathfrak{S}$ be a class of some $\mathcal{E}$-linear functors.
    \begin{enumerate}[label=(\arabic*)]
        \item If $H_1\to H_2$ is an epimorphism and $H_1$ is $\mathfrak{S}$-epimorphic, then $H_2$ is $\mathfrak{S}$-epimorphic.
        \item If $H_1\to H_2$ is a monomorphism and $H_2$ is $\mathfrak{S}$-monomorphic, then $H_1$ is $\mathfrak{S}$-monomorphic.
    \end{enumerate}
\end{lemma}
\begin{proof}\quad
    \begin{enumerate}[label=(\arabic*)]
        \item Apply Lemma \ref{lem:exactproperty4} to $0\to \operatorname{ker}(H_1\to H_2)\to H_1\to H_2\to 0$.
        \item Apply Lemma \ref{lem:exactproperty4} to $0\to H_1\to H_2\to H_2/H_1\to 0$.
    \end{enumerate}
\end{proof}
\begin{lemma}\label{lem:exactproperty6}
    Let $\mathfrak{S}$ be a class of some $\mathcal{E}$-linear functors. Small limit of $\mathfrak{S}$-monomorphic functors in $\mathrm{Fun}(\mathrm{Cat}^\mathrm{perf}(\mathcal{E}),\mathrm{Ab})$ is still $\mathfrak{S}$-monomorphic. Small colimit of $\mathfrak{S}$-epimorphic functors in $\mathrm{Fun}(\mathrm{Cat}^\mathrm{perf}(\mathcal{E}),\mathrm{Ab})$ is still $\mathfrak{S}$-epimorphic.
\end{lemma}
\begin{definition}
    Let $\mathfrak{S}$ be a class of some $\mathcal{E}$-linear functors and $S_1,S_2$ be two classes of $\mathcal{E}$-linear $\infty$-categories. Denote by
    \begin{itemize}
        \item $S_1/\mathfrak{S}$, the class of $\mathcal{E}$-linear functors $\mathcal{F}:\mathcal{A}\to\mathcal{B}$ with $\mathcal{A}\in S_1$ and $\mathcal{F}\in\mathfrak{S}$.
        \item $\mathfrak{S}/S_2$, the class of $\mathcal{E}$-linear functors $\mathcal{F}:\mathcal{A}\to\mathcal{B}$ with $\mathcal{B}\in S_2$ and $\mathcal{F}\in\mathfrak{S}$.
    \end{itemize}
\end{definition}
\begin{definition}
    Let $H:\mathrm{Cat}^\mathrm{perf}(\mathcal{E})\to\mathrm{Ab}$ be a functor. The \emph{kernel} of $H$, denoted by $\ker H$, is the class of objects $\mathcal{C}\in \mathrm{Cat}^\mathrm{perf}(\mathcal{E})$ with $H(\mathcal{C})$ trivial.
\end{definition}
\section{D\'evissage condition for exact functors}
\subsection{Category of sequences}
\begin{definition}\cite[Definition 7.1]{blumberg_gepner_tabuada_2013}
Let \( \operatorname{Gap}([n], \mathcal{C}) \) be the full subcategory of \( \operatorname{Fun}(\mathrm{N}(\operatorname{Ar}[n]), \mathcal{C}) \) spanned by the functors \( X_{\bullet,\bullet}:\mathrm{N}(\operatorname{Ar}[n]) \to \mathcal{C} \) such that, for each \( i \in I \), \( X_{i,i} \) is a zero object of \( \mathcal{C} \), and for each \( i < j < k \), the square
\[
\begin{tikzcd}
X_{i,j} \arrow[r] \arrow[d] & X_{i,k} \arrow[d] \\
X_{j,j} \arrow[r] & X_{j,k}
\end{tikzcd}
\]
is pushout.
\end{definition}
\begin{lemma}\cite[Lemma 7.3]{blumberg_gepner_tabuada_2013}
    Let $\mathcal{C}$ be a stable $\infty$-category. The forgetful functor
    \[
    \mathrm{Gap}([n],\mathcal{C})\to\mathrm{Fun}(\Delta^{n-1},\mathcal{C}),X_{\bullet,\bullet}\mapsto X_{0,\bullet}
    \]
    is an equivalence of $\infty$-category.
\end{lemma}
Let $F_{\bullet,\bullet}:X_{\bullet,\bullet}\to Y_{\bullet,\bullet}$ be a morphism in $\mathrm{Gap}([n],\mathcal{C})$, then $\mathrm{fib}(F_{\bullet,\bullet})_{i,j}\simeq\mathrm{fib}(F_{i,j})$ and $\mathrm{cofib}(F_{\bullet,\bullet})_{i,j}\simeq\mathrm{cofib}(F_{i,j})$. Consider the functor \( [n+1] \to [n] \) which sends \( i \mapsto i \) (if \( i \leq n \)) and \( n+1 \mapsto n \). This defines a functor:
\[
\operatorname{Ar}[n+1] \to \operatorname{Ar}[n]
\]
which induces an exact inclusion functor of stable $\infty$-categories
\[
i_n : \mathrm{Gap}([n],\mathcal{C}) \hookrightarrow \mathrm{Gap}([n+1],\mathcal{C}).
\]
Define 
\[\mathrm{Gap}(\mathcal{C}):= \operatorname{colim}\left\{ \mathcal{C} \stackrel{i_1}{\to} \mathrm{Gap}([2],\mathcal{C}) \stackrel{i_2}{\to} \cdots \stackrel{i_n}{\to} \mathrm{Gap}([n+1],\mathcal{C}) \stackrel{i_{n+1}}{\to} \cdots \right\}.
\]
Define by $\mathrm{Gap}^w([n],\mathcal{C})$ the full subcategory of $\mathrm{Gap}([n],\mathcal{C})$ whose objects are $X_{\bullet,\bullet}$ with $X_{0,n}= 0$ an zero object in $\mathcal{C}$. Thus, the exact inclusion $i_{n}$ restricts to an exact inclusion
\[
i^w_n: \mathrm{Gap}^w([n],\mathcal{C}) \hookrightarrow \mathrm{Gap}^w([n+1],\mathcal{C}).
\]
Define 
\[
\mathrm{Gap}^w(\mathcal{C}):= \operatorname{colim}\left\{ 0 \stackrel{i^w_{1}}{\to} \mathrm{Gap}^w([2],\mathcal{C}) \stackrel{i^w_{2}}{\to} \cdots \stackrel{i^w_{n}}{\to} \mathrm{Gap}^w([n+1],\mathcal{C}) \stackrel{i^w_{n+1}}{\to} \cdots \right\}.
\]
  \begin{definition}
  Assume $\mathcal{F}:\mathcal{A}\to\mathcal{C}$ is an exact functor of small stable $\infty$-categories. Define $\mathrm{Gap}([n],\mathcal{F})$ and $\mathrm{Gap}^w([n],\mathcal{F})$ by the pullbacks
  \[
  \begin{tikzcd}
      \mathrm{Gap}([n],\mathcal{F})\ar[r]\ar[d]&\mathrm{Gap}([n],\mathcal{C})\ar[d]\\
      \mathcal{A}^n\ar[r,"\mathcal{F}^n"]&\mathcal{C}^n,
  \end{tikzcd},\begin{tikzcd}
      \mathrm{Gap}^w([n],\mathcal{F})\ar[r]\ar[d]&\mathrm{Gap}^w([n],\mathcal{C})\ar[d]\\
      \mathcal{A}^n\ar[r,"F^n"]&\mathcal{C}^n,
  \end{tikzcd}
  \]
  where $\mathrm{Gap}([n],\mathcal{C})\to\mathcal{C}^n$ is defined by $X_{\bullet,\bullet}\mapsto (X_{0,1},\cdots,X_{n-1,n})$ and $\mathrm{Gap}^w([n],\mathcal{C})\to\mathcal{C}^n$ is its restriction.
\end{definition}
$\mathrm{Gap}([n],\mathcal{F})$ consists of
\begin{itemize}
    \item Objects are pair $(A_\bullet,X_{\bullet,\bullet})$, where $A_\bullet\in \mathcal{A}^n, X_{\bullet,\bullet}\in\mathrm{Gap}([n],\mathcal{C})$ are objects such that $\mathcal{F}(A_i)=X_{i-1,i}$.
    \item Morphims are pair $(f_\bullet,F_{\bullet,\bullet})$, where $f_\bullet\in \mathcal{A}^n, F_{\bullet,\bullet}\in\mathrm{Gap}([n],\mathcal{C})$ are morphisms such that $\mathcal{F}(f_i)=F_{i-1,i}$.
\end{itemize}
Thus, the exact inclusion $i_n$ and $\mathcal{A}^n\to\mathcal{A}^{n+1},(A_1,\cdots,A_n)\mapsto (A_1,\cdots, A_n,0)$ induces an exact inclusion
\[
i_{n,\mathcal{F}}: \mathrm{Gap}([n],\mathcal{F}) \hookrightarrow \mathrm{Gap}([n+1],\mathcal{F})
\]
Define 
\[
\mathrm{Gap}(\mathcal{F}):= \operatorname{colim}\left\{ \mathcal{A} \stackrel{i_{1,\mathcal{F}}}{\to} \mathrm{Gap}([2],\mathcal{F}) \stackrel{i_{2,\mathcal{F}}}{\to} \cdots \stackrel{i^w_{n,\mathcal{A}}}{\to} \mathrm{Gap}^w([n+1],\mathcal{F}) \stackrel{i_{n+1,\mathcal{F}}}{\to} \cdots \right\}.
\]
$\mathrm{Gap}^w([n],\mathcal{F})$ is naturally a full subcategory of $\mathrm{Gap}([n],\mathcal{F})$. Then, the exact inclusion $i_{n,\mathcal{A}}$ restricts to an exact inclusion
\[
i^w_{n,\mathcal{F}}: \mathrm{Gap}^w([n],\mathcal{F}) \hookrightarrow \mathrm{Gap}^w([n+1],\mathcal{F}).
\]
Define 
\[
\mathrm{Gap}^w(\mathcal{F}):= \operatorname{colim}\left\{ 0 \stackrel{i^w_{1,\mathcal{F}}}{\to} \mathrm{Gap}^w([2],\mathcal{F}) \stackrel{i^w_{2,\mathcal{F}}}{\to} \cdots \stackrel{i^w_{n,\mathcal{F}}}{\to} \mathrm{Gap}^w([n+1],\mathcal{F}) \stackrel{i^w_{n+1,\mathcal{F}}}{\to} \cdots \right\}.
\]
$\mathrm{Gap}([n],\mathcal{F}),\mathrm{Gap}^w([n],\mathcal{F})$ and $\mathrm{Gap}(\mathcal{F}),\mathrm{Gap}^w(\mathcal{F})$ are stable. Furthermore, if $\mathcal{C}$ and $\mathcal{A}$ are idempotent-complete, so do them. $\mathrm{Gap},\mathrm{Gap}^w:\mathcal{F}\mapsto\mathrm{Gap}(\mathcal{F}),\mathrm{Gap}^w(\mathcal{F})$ can be seen as functors $\mathrm{Fun}(\Delta^1,\mathrm{Cat}^\mathrm{ex})\to \mathrm{Cat}^\mathrm{ex}$.
\subsection{Definition of d\'evissage condition}
\begin{prop}\label{prop:devissagedefine}
    Let $\mathcal{F}:\mathcal{A}\to\mathcal{C}\in\mathrm{Cat}^{\mathrm{ex}}$ be an exact functor between small stable $\infty$-categories. Then the following are equivalent:
    \begin{enumerate}[label=(\arabic*)]
        \item For any morphism $f:X\to Y$ in $\mathcal{C}$, there is a factorization
    \[
    f:X=X_0\to X_1\to\cdots\to X_n=Y
    \]
    such that $\mathrm{cofib}(X_i\to X_{i+1})\in\mathcal{F}(\mathcal{A})$ up to equivalences.
        \item for any morphism $f:0\to Y$ in $\mathcal{C}$, there is a factorization
    \[
    f:0=X_0\to X_1\to\cdots\to X_n=Y
    \]
    such that $\mathrm{cofib}(X_i\to X_{i+1})\in\mathcal{F}(\mathcal{A})$ up to equivalences.
    \item For any morphism $f:X\to 0$ in $\mathcal{C}$, there is a factorization
    \[
    f:X=X_0\to X_1\to\cdots\to X_n=0
    \]
    such that $\mathrm{cofib}(X_i\to X_{i+1})\in\mathcal{F}(\mathcal{A})$ up to equivalences.
    \end{enumerate}
\end{prop}
\begin{proof}
    $(1) \Rightarrow (2),(3)$ is obvious. Now we prove $(2)\Rightarrow (1)$ and $(3)\Rightarrow (1)$.

    $(2)\Rightarrow (1)$. For a morphism $X\to Y$, assume $Z$ is its cofiber. From $(2)$, there is a sequence
    \[
    0=Z_0\to\cdots\to Z_n=Z
    \]
    such that $\operatorname{cofib}(Z_i\to Z_{i+1})\in\mathcal{F}(\mathcal{A})$ up to equivalences. Define $X_i$ by pullback
    \[
    \begin{tikzcd}
        X_i\ar[r]\ar[d]&Y\ar[d]\\
        Z_i\ar[r]&Z
    \end{tikzcd}
    \]
    Then $X_n=Y$ and $X_0=\operatorname{fib}(Y\to Z)=X$. There are also pushout squares in a stable $\infty$-category. Hence, we get pushout squares
    \[
    \begin{tikzcd}
        X_i\ar[r]\ar[d]&X_{i+1}\ar[d]\\
        Z_i\ar[r]&Z_{i+1}
    \end{tikzcd},
    \]
    by Lemma \ref{lem:diagram3}. By Lemma \ref{lem:cofiberofpullback}, we get $\operatorname{cofib}(X_{i}\to X_{i+1})\simeq\operatorname{cofib}(Z_i\to Z_{i+1})\in\mathcal{F}(\mathcal{A})$ up to equivalences. Then the sequence
    \[
    X=X_0\to X_1\to\cdots\to X_n=Y
    \]
    is what we desired.

    $(3)\Rightarrow (1)$. For a morphism $X\to Y$, assume $W$ is its fiber. From $(3)$, there is a sequence
    \[
    W=W_0\to\cdots\to W_n=0
    \]
    such that $\operatorname{cofib}(W_i\to W_{i+1})\in\mathcal{F}(\mathcal{A})$. Define $X_i$ by pushout
    \[
    \begin{tikzcd}
        W\ar[r]\ar[d]&X\ar[d]\\
        W_i\ar[r]&X_i
    \end{tikzcd}
    \]
    Then $X_0=X$ and $X_n=\operatorname{cofib}(W\to X)=Y$. Then we get pushout squares
    \[
    \begin{tikzcd}
        W_i\ar[r]\ar[d]&X_{i}\ar[d]\\
        W_{i+1}\ar[r]&X_{i+1}
    \end{tikzcd}.
    \]
    by Lemma \ref{lem:diagram3}. By Lemma \ref{lem:cofiberofpullback}, we get $\operatorname{cofib}(X_{i}\to X_{i+1})\simeq\operatorname{cofib}(W_i\to W_{i+1})\in\mathcal{F}(\mathcal{A})$ up to equivalences. Then the sequence
    \[
    X=X_0\to X_1\to\cdots\to X_n=Y
    \]
    is what we desired.
\end{proof}
\begin{remark}
    We will omit "up to equivalence" in the following.
\end{remark}
\begin{definition}\label{def:devissage}
    We say the exact functor $\mathcal{F}$ satisfies the \textit{d\'evissage condition} if the equivalent descriptions in the Proposition \ref{prop:devissagedefine} hold true. 
\end{definition}
\begin{example}
    Let $\mathcal{C}$ be a stable $\infty$-category. Then $\mathrm{id}_\mathcal{C}:\mathcal{C}\to\mathcal{C}$ and $0_\mathcal{C}:\mathcal{C}\to *$ satisfy the d\'evissage condition.
\end{example}
\begin{example}
    A fully faithful functor satisfies the d\'evissage condition if and only if it is an equivalence.
\end{example}
Denote by $\mathrm{E}(\mathcal{F})$ the full subcategory of $\mathcal{C}$ consisting of objects $X$ such that $0\to X$ admits a factorization as above,i.e. the extension closure of objects in $\mathcal{F}(\mathcal{A})$. By \cite[Lemma 1.1.3.3]{lurie_2017}, $\mathrm{E}(\mathcal{F})$ is stable full subcategory of $\mathcal{C}$. We can naturally get the following results.
\begin{prop}\label{prop:devissage}
    $\mathcal{F}$ satisfies the d\'evissage condition if and only if $\mathrm{E}(\mathcal{F})=\mathcal{C}$.
\end{prop}
\begin{remark}
    Ofcourse, $\mathcal{A}\to\mathrm{E}(\mathcal{F})$ satisfies the d\'evissage condition.
\end{remark}
When we assume $\mathcal{C}$ and $\mathcal{A}$ are idempotent-complete, we only need the weaker condition in some cases.
\begin{definition}\label{def:weakdevissage}
     Assume $\mathcal{C}$ and $\mathcal{A}$ are idempotent-complete. We say the $\mathcal{F}$ satisfies the \textit{weak d\'evissage condition} if $\mathrm{Idem}(\mathrm{E}(\mathcal{F}))=\mathcal{C}$.
\end{definition}
\begin{remark}
    $\mathrm{Idem}(\mathrm{E}(\mathcal{F}))$ is just the thick closure of $\mathrm{Ob}(\mathcal{F}(\mathcal{A}))$ in $\mathcal{C}$, i.e. the smallest stable full subcategory of $\mathcal{C}$ containing $\mathcal{F}(\mathcal{A})$ which is closed under extensions and direct factors.
\end{remark}
There are exact “evaluation” functors:
 \[
 \mathrm{ev}_{n,\mathcal{A}}:\mathrm{Gap}([n],\mathcal{F})\to\mathcal{C},(A_\bullet,X_{\bullet,\bullet})\mapsto X_{0,n}.
 \]
These are compatible with stabilization along the inclusion functors $i_{n,\mathcal{F}}$, so there is also an induced exact functor:
\[
  \mathrm{ev}_\mathcal{A}:\mathrm{Gap}(\mathcal{F})\to\mathcal{C}.
\]
\begin{prop}\label{prop:fullyfaithfulofev}
    The functor $\mathrm{ev}_\mathcal{F}:\mathrm{Gap}(\mathcal{F})\to\mathcal{C}$ induces a well-defined functor
    \[
    \widehat{\mathrm{ev}}_\mathcal{F}:\mathrm{Gap}(\mathcal{F})/\mathrm{Gap}^w(\mathcal{F})\to\mathrm{E}(\mathcal{F}),
    \]
    which is an equivalence.
\end{prop}
\begin{proof}
    By \cite[Corollary 5.11, Proposition 5.15]{blumberg_gepner_tabuada_2013}, we only need to prove
    \[
    \mathrm{Ho}(\mathrm{Gap}(\mathcal{F}))/\mathrm{Ho}(\mathrm{Gap}^w(\mathcal{F}))\to\mathrm{Ho}(\mathrm{E}(\mathcal{F}))
    \]
    is an equivalence of homotopy categories. We prove it by steps.
    \begin{enumerate}[label=Step \arabic*:]
        \item \emph{For every morphism $\mathrm{ev}_\mathcal{F}(A_\bullet,X_{\bullet,\bullet})\to Y$ in $\mathrm{E}(\mathcal{F})$, there is a morphism $(A_\bullet,X_{\bullet,\bullet})\to (A'_\bullet,X'_{\bullet,\bullet})$ such that $\mathrm{ev}_\mathcal{F}((A_\bullet,X_{\bullet,\bullet})\to (A'_\bullet,X'_{\bullet,\bullet}))\simeq\mathrm{ev}(A'_\bullet,X_{\bullet,\bullet})\to Y$}.

        Assume $(A_\bullet,X_{\bullet,\bullet})\in\mathrm{Gap}([n],\mathcal{F})$. Since $\mathcal{A}\to \mathrm{E}(\mathcal{F})$ satisfies the d\'evissage condition, there is a factorization
        \[
        \mathrm{ev}_\mathcal{F}(A_\bullet,X_{\bullet,\bullet})=Y_0\to Y_1\to\cdots\to Y_m=Y
        \]
        such that $\operatorname{cofib}(Y_i\to Y_{i+1})\in\mathcal{F}(\mathcal{A})$. Assume $B_1,\cdots, B_m\in\mathcal{A}$ such that $\mathcal{F}(B_i)=\operatorname{cofib}(Y_{i-1}\to Y_{i})$. Then $(A_{\bullet},X_{\bullet,\bullet})\to (A'_{\bullet},X'_{\bullet,\bullet})$ is defined as follows
        \[
        \begin{tikzcd}
            X_{\bullet,\bullet}:&0\ar[r]\ar[d]& X_{0,1}\ar[r]\ar[d]&\cdots\ar[r]&X_{0,n}\ar[r]\ar[d]&X_{0,n}\ar[r]\ar[d]&\cdots\ar[r]&X_{0,n}\ar[d]\\
            X'_{\bullet,\bullet}:&0\ar[r]&X_{0,1}\ar[r]&\cdots\ar[r]&X_{0,n}\ar[r]&Y_1\ar[r]&\cdots\ar[r]&Y_n\\
            A_\bullet:&A_1\ar[d,"\mathrm{i d}"]&\cdots&A_n\ar[d,"\mathrm{id}"]&0\ar[d,"0"]&\cdots&0\ar[d,"0"]\\
            A'_{\bullet}:&A_1&\cdots&A_n&B_1&\cdots&B_m\\
        \end{tikzcd}
        \]
        It is well-defined since $X_{0,n}=Y_0$.
        \item \emph{If $\mathrm{ev}_\mathcal{F}(A_\bullet,X_{\bullet,\bullet})\simeq\mathrm{ev}_\mathcal{F}(B_\bullet,Y_{\bullet,\bullet})$, $(A_\bullet,X_{\bullet,\bullet})\simeq (B_\bullet,Y_{\bullet,\bullet})$ in $\mathrm{Gap}(\mathcal{F})/\mathrm{Gap}^w(\mathcal{F})$.}

        Assume $(A_\bullet,X_{\bullet,\bullet}),(B_\bullet,Y_{\bullet,\bullet})\in \mathrm{Gap}([n],\mathcal{F})$. Then define an object $(C_\bullet,Z_{\bullet,\bullet})$ and two morphisms $(C_\bullet,Z_{\bullet,\bullet})\to (A_\bullet,X_{\bullet,\bullet})$ and $(C_\bullet,Z_{\bullet,\bullet})\to (B_\bullet,Y_{\bullet,\bullet})$ as follows.
        \[
        \begin{tikzcd}
            X_{\bullet,\bullet}:&0\ar[r]&\cdots\ar[r]&X_{0,n}\ar[r]&X_{0,n}\ar[r]&\cdots\ar[r]&X_{0,n}\\
            Z_{\bullet,\bullet}:&0\ar[r]\ar[d]\ar[u]&\cdots\ar[r]&0\ar[r]\ar[d]\ar[u]&X_{0,1}\ar[r]\ar[d]\ar[u]&\cdots\ar[r]&X_{0,n}\ar[d]\ar[u]\\
            Y_{\bullet,\bullet}:&0\ar[r]&\cdots\ar[r]&Y_{0,n}\ar[r]&Y_{0,n}\ar[r]&\cdots\ar[r]&Y_{0,n}\\
            A_\bullet:&A_1&\cdots&A_n&0&\cdots&0\\
            C_\bullet:&0\ar[d]\ar[u]&\cdots&0\ar[u]\ar[d]&A_1\ar[u]\ar[d]&\cdots&A_n\ar[u]\ar[d]\\
            B_\bullet:&B_1&\cdots&B_n&0&\cdots&0
        \end{tikzcd}
        \]
        There are well-defined since $X_{0,n}\simeq Y_{0,n}$. One can checks that there are equivalences in $\mathrm{Gap}(\mathcal{F})/\mathrm{Gap}^w(\mathcal{F})$.
        \item \emph{For two morphisms $(f_\bullet,F_{\bullet,\bullet}),(g_\bullet,G_{\bullet,\bullet}):(A_\bullet,X_{\bullet,\bullet})\to (B_\bullet,Y_{\bullet,\bullet})$, if $\mathrm{ev}_\mathcal{F}(f_\bullet,F_{\bullet,\bullet})\simeq\mathrm{ev}_\mathcal{F}(g_\bullet,G_{\bullet,\bullet})$, then $(f_\bullet,F_{\bullet,\bullet})\simeq (g_\bullet,G_{\bullet,\bullet})$ in $\mathrm{Gap}(\mathcal{F})/\mathrm{Gap}^w(\mathcal{F})$.}

        Assume $(f_\bullet,F_{\bullet,\bullet}), (g_\bullet,G_{\bullet,\bullet})\in\mathrm{Gap}([n],\mathcal{F})$. Consider the diagram
        \[
        \begin{tikzcd}[row sep=tiny,column sep=tiny]
             X'_{\bullet,\bullet}:&&0\ar[rr]\ar[dd]\ar[dl]&&\cdots\ar[rr]&&0\ar[rr]\ar[dd]\ar[dl]&&X_{0,1}\ar[rr]\ar[dd]\ar[dl,"F_{0,1}"']&&\cdots\ar[rr]&&X_{0,n}\ar[dd]\ar[dl,"F_{0,n}"']\\
             Y'_{\bullet,\bullet}:&0\ar[rr,crossing over]&&\cdots\ar[rr,crossing over]&&0\ar[rr,crossing over]&&Y_{0,1}\ar[rr,crossing over]&&\cdots\ar[rr,crossing over]&&Y_{0,n}&\\
             X_{\bullet,\bullet}:&&0\ar[rr]\ar[dl]&&\cdots\ar[rr]&&X_{0,n}\ar[dl]\ar[rr]&&X_{0,n}\ar[dl]\ar[rr]&&\cdots\ar[rr]&&X_{0,n}\ar[dl]\\
            Y_{\bullet,\bullet}:&0\ar[rr,crossing over]\ar[from=uu,crossing over]&&\cdots\ar[rr,crossing over]&&Y_{0,n}\ar[rr,crossing over]\ar[from=uu,crossing over]&&Y_{0,n}\ar[rr,crossing over]\ar[from=uu,crossing over]&&\cdots\ar[rr,crossing over]&&Y_{0,n}\ar[from=uu,crossing over]&\\
            A'_\bullet:&&0\ar[dl]\ar[dd]&&\cdots&&0\ar[dl]\ar[dd]&&A_1\ar[dl,"f_1"']\ar[dd]&&\cdots&&A_n\ar[dl,"f_n"']\ar[dd]\\
            B'_\bullet:&0\ar[dd]&&\cdots&&0\ar[dd]&&B_1\ar[dd]&&\cdots&&B_n\ar[dd]&\\
            A_\bullet:&&A_1\ar[dl]&&\cdots&&A_n\ar[dl]&&0\ar[dl]&&\cdots&&0\ar[dl]\\
            B_\bullet:&B_1&&\cdots&&B_n&&0&&\cdots&&0&
        \end{tikzcd}
        \]
        Here $(A'_\bullet,X'_{\bullet,\bullet})\to (A_\bullet,X_{\bullet,\bullet})$ and $(B'_\bullet,Y'_{\bullet,\bullet})\to (B_\bullet,Y_{\bullet,\bullet})$ are just identity in $\mathrm{Gap}(\mathcal{F})/\mathrm{Gap}^w(\mathcal{F})$ and $(A_\bullet,X_{\bullet,\bullet})\to (B_\bullet,Y_{\bullet,\bullet})$ can be either $(f_\bullet,F_{\bullet,\bullet})$ or $(g_\bullet,G_{\bullet,\bullet})$. Hence, $(f_\bullet,F_{\bullet,\bullet}),(g_\bullet,G_{\bullet,\bullet})$ are both equivalent to $(A'_\bullet,X'_{\bullet,\bullet})\to (B'_\bullet,Y'_{\bullet,\bullet})$ in $\mathrm{Gap}(\mathcal{F})/\mathrm{Gap}^w(\mathcal{F})$.
        \item \emph{$\mathrm{Ho}(\widehat{\mathrm{ev}}_\mathcal{F}):\mathrm{Ho}(\mathrm{Gap}(\mathcal{F}))/\mathrm{Ho}(\mathrm{Gap}^w(\mathcal{F}))\to\mathrm{Ho}(\mathrm{E}(\mathcal{F}))$ is an equivalence}.

        Step 1 tells us it is essentially surjective. Step 1 and Step 2 tells us it is surjective on morphisms. Step 3 tells us it is injective on morphisms. Hence, it is both essentially surjective and fully faithful.
    \end{enumerate}
\end{proof}
\begin{cor}\label{cor:fullyfaithfulofev}
    $\mathcal{F}$ satisfies the d\'evissage condition if and only if $\mathrm{ev}_\mathcal{F}$ is essentially surjective.
\end{cor}
For a small stable $\infty$-categories $\mathcal{C}$, its connective $K^\mathrm{cn}_0$-group is the Grothendieck group of $\mathrm{Ho}(\mathcal{C})$:
\[
K^\mathrm{cn}_0(\mathcal{C})=\frac{\text{Free abelian group on isomorphism classes in } \mathrm{Ho}(\mathcal{C})}{\langle [X]-[Y]+[Z] : \text{cofiber sequence }X\to Y\to Z \rangle}
\]
Its non-connective $K_0$-group is $K_0(\mathcal{C})=K^\mathrm{cn}(\mathrm{Idem}(\mathcal{C}))$.
\begin{lemma}\label{lem:devissageK_0}
    Let $\mathcal{F}:\mathcal{A}\to\mathcal{C}\in\mathrm{Cat}^{\mathrm{perf}}$ be an exact functor between small stable idempotent-complete $\infty$-categories, which satisfies the weak d\'evissage condition. Then the following statements are equivalent:
    \begin{enumerate}[label=(\arabic*)]
        \item $\mathcal{F}$ satisfies the d\'evissage condition.
        \item $K_0(\mathcal{F}):K_0(\mathcal{A})\to K_0(\mathcal{C})$ is an epimorphism.
    \end{enumerate}
\end{lemma}
\begin{proof}
    $(1)\Rightarrow (2)$ isobvious. Now assume $\mathcal{F}$ satisfies the weak d\'evissage condition and $K_0(\mathcal{F}):K_0(\mathcal{A})\to K_0(\mathcal{C})$ is an epimorphism. Notice that $K_0(\mathcal{F})$ is just $K_0^\mathrm{cn}(\mathcal{F})$, which can be factored as
    \[
    K_0^\mathrm{cn}(\mathcal{A})\to K_0^\mathrm{cn}(\mathrm{E}(\mathcal{F}))\to K_0^\mathrm{cn}(\mathcal{C})=K_0^\mathrm{cn}(\mathrm{Idem}(\mathrm{E}(\mathcal{F}))).
    \]
    Hence, $K_0^\mathrm{cn}(\mathrm{E}(\mathcal{F}))\to K_0^\mathrm{cn}(\mathcal{C})$ must be an epimorphism. Since $\mathrm{E}(\mathcal{F})$ is a dense subcategory of $\mathcal{C}$, by \cite[Theorem 2.1]{THOMASON_1997}, $K_0^\mathrm{cn}(\mathrm{E}(\mathcal{F}))\to K_0^\mathrm{cn}(\mathcal{C})$ is always a monomorphism and is an epimorphism only if $\mathrm{E}(\mathcal{F})=\mathcal{C}$. Hence, by Proposition \ref{prop:devissage}, $\mathcal{F}$ satisfies the d\'evissage condition.
\end{proof}
\begin{definition}
    Let $\mathcal{F}:\mathcal{A}\to\mathcal{C}$ be an $\mathcal{E}$-linear functor. We say $\mathcal{F}$ satisfies the \emph{(weak) d\'evissage condition} if its image under the forgetful functor $\mathrm{Cat}^\mathrm{perf}(\mathcal{E}) \to \mathrm{Cat}^\mathrm{perf}$ satisfies the (weak) d\'evissage condition.
\end{definition}
\section{Theorems for invariants}
In this section, we always assume $\mathcal{E}\in\mathrm{CAlg}^\mathrm{rig}(\mathrm{Cat}^\mathrm{perf})$.
\subsection{Additivity theorem}
Assume $\mathcal{C}\in\mathrm{Cat}^{\mathrm{perf}}(\mathcal{E})$ and $\mathcal{F}_\mathcal{A}:\mathcal{A}\to\mathcal{C}, \mathcal{F}_\mathcal{B}:\mathcal{B}\to\mathcal{C}$ are $\mathcal{E}$-linear functors. Define a category $\mathcal{S}(\mathcal{F}_\mathcal{A},\mathcal{F}_\mathcal{B})$ by pullback
\[
\begin{tikzcd}
    \mathcal{S}(\mathcal{F}_\mathcal{A},\mathcal{F}_\mathcal{B})\ar[r]\ar[d]&\mathrm{Gap}([2],\mathcal{C})\ar[d]\\
    \mathcal{A}\times\mathcal{B}\ar[r,"{(\mathcal{F}_\mathcal{A},\mathcal{F}_\mathcal{B})}"]&\mathcal{C}^2
\end{tikzcd}
\]
Where $\mathrm{Gap}([2],\mathcal{C}) \longrightarrow \mathcal{C}$ is defined by $X_{\bullet,\bullet} \mapsto \bigl(X_{0,1}, X_{1,2}\bigr)$. Then $\mathcal{S}(\mathcal{F}_\mathcal{A},\mathcal{F}_\mathcal{B})$ is also in $\mathrm{Cat}^\mathrm{perf}(\mathcal{E})$.
\begin{thm}[$\infty$-categorical Additivity Theorem]\label{thm:additivity}
\quad\\
    Assume $\mathcal{C}\in\mathrm{Cat}^{\mathrm{perf}}(\mathcal{E})$ and $\mathcal{F}_\mathcal{A}:\mathcal{A}\to\mathcal{C}, \mathcal{F}_\mathcal{B}:\mathcal{B}\to\mathcal{C}$ are $\mathcal{E}$-linear functors. The natural projection $ \mathcal{S}(\mathcal{F}_\mathcal{A},\mathcal{F}_\mathcal{B})\to\mathcal{A}\times\mathcal{B}$
    induces a homotopy equivalence $\mathcal{U}_{\mathrm{add}}(\mathcal{S}(\mathcal{F}_\mathcal{A},\mathcal{F}_\mathcal{B}))\xrightarrow{\simeq} \mathcal{U}_{\mathrm{add}}(\mathcal{A})\times\mathcal{U}_{\mathrm{add}}(\mathcal{B})$.
\end{thm}
\begin{proof}
There is a split exact sequence
\[
\mathcal{A}\overset{\xrightarrow{f}}{\xleftarrow[i]{}}\mathcal{S}(\mathcal{A},\mathcal{C},\mathcal{B})\overset{\xrightarrow{g}}{\xleftarrow[j]{}}\mathcal{B},
\]
where $i,g$ are projections, $f,j$ are defined by
\[
\begin{aligned}
f:\mathcal{A}&\to\mathcal{A}\times \mathcal{B},X\mapsto(X,0)\\
\mathcal{A}&\to\mathrm{Gap}([2],\mathcal{C}),X\mapsto (0\to \mathcal{F}_\mathcal{A}(X)\to\mathcal{F}_\mathcal{A}(X))\\
g:\mathcal{B}&\to\mathcal{A}\times\mathcal{B}, Y\mapsto(0,Y)\\
\mathcal{B}&\to\mathrm{Gap}([2],\mathcal{C}), Y\mapsto (0\to 0\to\mathcal{F}_\mathcal{B}(Y)).
\end{aligned}
\]
Thus, it induces a homotopy equivalence
\[
\mathcal{U}_{\mathrm{add}}( \mathcal{S}(\mathcal{F}_\mathcal{A},\mathcal{F}_\mathcal{B}))\xrightarrow{\simeq} \mathcal{U}_{\mathrm{add}}(\mathcal{A})\times \mathcal{U}_{\mathrm{add}}(\mathcal{B}).
\]
\end{proof}
Let $\mathcal{F}:\mathcal{A}\to\mathcal{C}$ be an $\mathcal{E}$-linear functor in $\mathrm{Cat}^\mathrm{perf}(\mathcal{E})$. Denote by $q_{n,\mathcal{F}}:\mathrm{Gap}([n],\mathcal{F})\to\mathcal{A}^n$ the natural projections.
\begin{prop}\label{prop:additive}
     For each $n \geq 1$, the exact functor $q_{n,\mathcal{F}}$ induces a homotopy equivalence:
     \[
     \mathcal{U}_{\mathrm{add}}(\mathrm{Gap}([n],\mathcal{F}))\xrightarrow{\simeq}\mathcal{U}_{\mathrm{add}}(\mathcal{A})^n.
     \]
      Furthermore, this also induces a homotopy equivalence $\mathcal{U}_{\mathrm{add}}(\mathrm{Gap}(\mathcal{F})) \simeq \mathrm{colim}_n \mathcal{U}_{\mathrm{add}}(\mathcal{A}^n)$.
\end{prop}
\begin{proof}
Define $\mathcal{F}_1:\mathcal{A} \to \mathrm{Gap}([n],\mathcal{F})$ by
\[
\begin{aligned}
    \mathcal{A}&\to \mathrm{Gap}([n],\mathcal{C}), X\mapsto (0\to\mathcal{F}(X)\to\mathcal{F}(X)\to\cdots\to\mathcal{F}(X))\\
    \mathcal{A}&\to\mathcal{A}^n, X\mapsto (X,0,\cdots,0).
\end{aligned}
\]
Define $\mathcal{F}_2:\mathrm{Gap}([n-1],\mathcal{F})\to \mathrm{Gap}([n],\mathcal{F})$ by
\[
\begin{aligned}
    \mathrm{Gap}([n-1],\mathcal{C})&\to\mathrm{Gap}([n],\mathcal{C}),(0\to X_{0,1}\to\cdots\to X_{0,n-1})\mapsto (0\to 0\to X_{0,1}\to\cdots\to X_{0,n-1})\\
    \mathcal{A}^{n-1}&\to\mathcal{A}^n, (A_1,\cdots,A_{n-1})\mapsto (0,A_1,\cdots,A_{n-1})\\
    \mathcal{C}^{n-1}&\to\mathcal{C}^n, (X_1,\cdots,X_{n-1})\mapsto (0,X_1,\cdots,X_{n-1})
\end{aligned}
\]
These functors admit projections given by the exact functors
\[
\begin{aligned}
s_1: \mathrm{Gap}([n],\mathcal{F}) &\to \mathcal{A}, \quad (A_\bullet,X_{\bullet,\bullet}) \mapsto A_1, \\
s_2: \mathrm{Gap}([n],\mathcal{F}) &\to \mathrm{Gap}([n-1],\mathcal{F}), \quad (A_\bullet,X_{\bullet,\bullet}) \mapsto \bigl(0,A_2,\cdots,A_n,0 \to 0 \to X_{1,2} \to \cdots \to X_{1,n}\bigr),
\end{aligned}
\]
which assemble to an exact functor $\mathrm{Gap}([n],\mathcal{F}) \to \mathcal{S}\bigl(\mathcal{F}_1,\mathcal{F}_2\bigr)$. By the Additivity Theorem \ref{thm:additivity}, the pair of exact functors
\[
(s_1,s_2): \mathrm{Gap}([n],\mathcal{F}) \overset{\rightarrow}{\leftarrow} \mathrm{Gap}([n-1],\mathcal{F}) \times \mathcal{A} : \mathcal{F}_2\lor\mathcal{F}_2
\]
induces inverse homotopy equivalences for $\mathcal{U}_{\mathrm{add}}$. Inductively, this yields homotopy equivalences
\[
\mathcal{U}_{\mathrm{add}}\bigl(\mathrm{Gap}([n],\mathcal{F})\bigr) \simeq \mathcal{U}_{\mathrm{add}}\bigl(\mathrm{Gap}([n-1],\mathcal{F})\bigr) \times \mathcal{U}_{\mathrm{add}}(\mathcal{A}) \simeq \cdots \simeq \mathcal{U}_{\mathrm{add}}(\mathcal{A})^n\simeq \mathcal{U}_{\mathrm{add}}(\mathcal{A}^n).
\]
Direct inspection shows the composite homotopy equivalence is defined by the exact functor \(q_{n,\mathcal{F}}\). The map \(\mathcal{U}_{\mathrm{add}}(q_{n,\mathcal{F}})\) admits a homotopy inverse given by the exact functor
\[
\mathcal{A}^n \to \mathrm{Gap}([n],\mathcal{F}), \quad (A_1,\cdots,A_n) \mapsto \bigl(A_1,\cdots,A_n,0 \to \mathcal{F}(A_1) \to \mathcal{F}(A_1) \lor \mathcal{F}(A_2) \to \cdots \to \lor_{i=1}^n \mathcal{F}(A_i)\bigr).
\]

Moreover, we have commutative diagrams of exact functors:
\[
\begin{tikzcd}
\mathrm{Gap}([n],\mathcal{F}) \ar[r, "i_{n,\mathcal{F}}"] \ar[d, bend left=20] & \mathrm{Gap}([n+1],\mathcal{F}) \ar[d, bend left=20] \\
\mathcal{A}^n \ar[r] \ar[u, bend left=20] & \mathcal{A}^{n+1} \ar[u, bend left=20]
\end{tikzcd}
\]
where the bottom functor is the canonical inclusion, sending \((A_1, \cdots, A_{n-1}) \mapsto (A_1, \cdots, A_{n-1}, 0)\). The case \(n = \infty\) follows immediately, as \(\operatorname{colim}_n \mathcal{U}_{\mathrm{add}}\bigl(\mathrm{Gap}([n],\mathcal{F})\bigr) \stackrel{\simeq}{\to} \mathcal{U}_{\mathrm{add}}\bigl(\mathrm{Gap}(\mathcal{F})\bigr)\).
\end{proof}
\subsection{Categorization of fiber with d\'evissage condition}
In this section, we still assume $\mathcal{F}:\mathcal{A}\to\mathcal{C}$ is an $\mathcal{E}$-linear functor in $\mathrm{Cat}^\mathrm{perf}(\mathcal{E})$. For each $n \geq 1$, we have an $\mathcal{E}$-linear functor:
\[
q_{n,\mathcal{F}}^w:\mathrm{Gap}^w([n],\mathcal{F})\to\mathcal{A}^{n-1},(A_\bullet,X_{\bullet,\bullet})\mapsto (A_1,\cdots,A_{n-1}),
\]
It has an $\mathcal{E}$-linear section
\[
\begin{aligned}
    j_{n,\mathcal{F}}:&\mathcal{A}^{n-1}\to\mathrm{Gap}^w([n],\mathcal{F}),\\
    (A_1,\cdots,A_{n-1})&\mapsto(A_1,\cdots,A_{n-1},\lor_{i=1}^{n-1}\Sigma A_i,0\to \mathcal{F}(A_1)\to \mathcal{F}(A_1)\lor \mathcal{F}(A_2)\to\cdots\to \lor_{i=1}^{n-1}\mathcal{F}(A_i)\to0).
\end{aligned}
\]
Define
\[
\begin{aligned}
    \tau_{n,\mathcal{F}}:&\mathcal{A}^{n-1}\to\mathrm{Gap}^w([n],\mathcal{F}),\\
    (A_1,\cdots,A_{n-1})&\mapsto(A_1,A_2\lor\Sigma A_1,\cdots, A_{n-1}\lor\Sigma A_{n-2},\Sigma A_{n-1},0\to \mathcal{F}(A_1)\to\cdots\to \mathcal{F}(A_{n-1})\to0),
\end{aligned}
\]
where the morphism $\mathcal{F}(A_i)\to \mathcal{F}(A_{i+1})$ is the zero morphism for each $i$. The composite functor
\[
q_{n,\mathcal{F}}^w\circ\tau_{n,\mathcal{F}}:(A_1,\cdots,A_{n-1})\mapsto(A_1,A_2\lor\Sigma A_1,\cdots, A_{n-1}\lor\Sigma A_{n-2}),
\]
induces a homotopy equivalence for any additive invariant.

The diagram
\[
\begin{tikzcd}
    \mathrm{Gap}^w([n],\mathcal{F})\ar[r,"i_{n,\mathcal{F}}^w"]&\mathrm{Gap}^w([n+1],\mathcal{F})\\
    \mathcal{A}^{n-1}\ar[r]\ar[u,"\tau_{n,\mathcal{F}}"]&\mathcal{A}^{n}\ar[u,"\tau_{n+1,\mathcal{F}}"]
\end{tikzcd}
\]
commutes, where the bottom functor is the inclusion $(A_1,\cdots,A_{n-1})\mapsto(A_1,\cdots,A_{n-1},0)$. Thus, we get a functor
\[
\tau_\mathcal{F}:\mathrm{colim}_n\mathcal{A}^{n-1}\to \mathrm{Gap}^w(\mathcal{F}).
\]
We introduce a result about categorification of cofiber and its proof.
\begin{prop}\cite[Proposition 3.7]{Maxime_Vladimir_Christoph_2024}\label{prop:categorificationofcofiber}
    Let $\mathcal{F}:\mathcal{C}\to\mathcal{D}$ be an $\mathcal{E}$-linear functor. There exists an $\mathcal{E}$-linear category $\mathrm{Cone}(F)$ with an $\mathcal{E}$-linear functor $\mathcal{D}\to\mathrm{Cone}(\mathcal{F})$ that induces an equivalence
    \[
    \operatorname{cofib}(\mathcal{U}_\mathrm{loc}(\mathcal{C})\xrightarrow{\mathcal{U}_\mathrm{loc}(F)}\mathcal{U}_\mathrm{loc}(\mathcal{D}))\simeq \mathcal{U}_\mathrm {loc}(\mathrm{Cone}(\mathcal{F}))
    \]
\end{prop}
\begin{proof}
    Define the \emph{lex pullback} $\mathcal{D}\overset{\to}{\times}_{\mathrm{Ind}(\mathcal{D})}\mathrm{Ind}(\mathcal{C})$ by the pullback
    \[
    \begin{tikzcd}
        \mathcal{D}\overset{\to}{\times}_{\mathrm{Ind}(\mathcal{D})}\mathrm{Ind}(\mathcal{C})\ar[r]\ar[d]&\ar[d]\mathrm{Fun}(\Delta^1,\mathrm{Ind}(\mathcal{D}))\\
        \mathcal{D}\times\mathrm{Ind}(\mathcal{C})\ar[r]& \mathrm{Ind}(\mathcal{D})\times \mathrm{Ind}(\mathcal{D}).
    \end{tikzcd}
    \]
    Then there is a fully faithful functor $\mathcal{C}\to \mathcal{D}\overset{\to}{\times}_{\mathrm{Ind}(\mathcal{D})}\mathrm{Ind}(\mathcal{C}), X\mapsto(\mathcal{F}(X),X,\mathcal{F}(X)\xrightarrow{\mathrm{id}}\mathcal{F}(X))$. Define $\mathrm{Cone}(\mathcal{F}):=\mathrm{Idem}(\mathcal{D}\overset{\to}{\times}_{\mathrm{Ind}(\mathcal{D})}\mathrm{Ind}(\mathcal{C})/\mathcal{C})$. From \cite[Proposition 10]{Tamme_2018}, $\mathcal{U}_\mathrm{loc}(\mathcal{D}\overset{\to}{\times}_{\mathrm{Ind}(\mathcal{D})}\mathrm{Ind}(\mathcal{C}))\simeq \mathcal{U}_\mathrm{loc}(\mathcal{D})\oplus\mathcal{U}_\mathrm{loc}(\mathrm{Ind}(\mathcal{C}))$
    \[
    \operatorname{cofib}(\mathcal{U}_\mathrm{loc}(\mathcal{C})\xrightarrow{\mathcal{U}_\mathrm{loc}(\mathcal{F})}\mathcal{U}_\mathrm{loc}(\mathcal{D})\oplus\mathcal{U}_\mathrm{loc}(\mathrm{Ind}(\mathcal{C})))\simeq \mathcal{U}_\mathrm{loc}(\mathrm{Cone}(\mathcal{F})).
    \]
    Notice that $\mathcal{U}_\mathrm{loc}(\mathrm{Ind}(\mathcal{C}))\simeq 0$ since $\mathrm{Ind}(\mathcal{C})$ admits an Eilenberg swindle.
\end{proof}
\begin{remark}
    We replace $\mathrm{Ind}(-)^{\omega_1}$ by $\mathrm{Ind}(-)$. usually, $\mathrm{Ind}(\mathcal{C})$ is not small. To define localizing invariants, we need to use bigger universe. For the details of this technology, readers can refer to \cite{hennion_2017}. Simply put, we choose two universes $\mathbb{U}\in \mathbb{V}$. Then we use $\mathrm{Ind}^\mathbb{U}(\mathcal{C})$, which is $\mathbb{V}$-small. Hence, we could define localizing invariants of $\mathrm{Ind}^\mathbb{U}(\mathcal{C})$. Or, we could use $\mathrm{Ind}(-)^{\kappa}$ for some regular cardinal $\kappa>\omega$ like the construction in the original text \cite{Maxime_Vladimir_Christoph_2024}.
\end{remark}
Denote by $\mathrm{Q}(\mathcal{F}):=\mathrm{Cone}(\tau_\mathcal{F})$. Since $\mathcal{U}_\mathrm{loc}(\tau_{\mathcal{F}})$ is the section of $\mathcal{U}_\mathrm{loc}(q^w_{\mathcal{F}})$, the cofiber sequence
\[
\mathcal{U}_\mathrm{loc}(\operatorname{colim}_n\mathcal{A}^{n-1})\xrightarrow{\mathcal{U}_\mathrm{loc}(\tau_\mathcal{F})}\mathcal{U}_\mathrm{loc}(\mathrm{Gap}^w(\mathcal{F}))\to\mathcal{U}_\mathrm{loc}(\mathrm{Q}(\mathcal{F}))
\]
split.
\begin{thm}\label{thm:quotient}
    If $\mathcal{F}$ satisfies the weak d\'evissage condition, there is a homotopy equivalence
    \[
    \mathcal{U}_{\mathrm{loc}}(\mathrm{Q}(\mathcal{F}))\simeq \mathrm{fib}(\mathcal{U}_{\mathrm{loc}}(\mathcal{F}))
    \]
\end{thm}
\begin{proof}
    By Proposition \ref{prop:fullyfaithfulofev} and the definition of weak d\'evissage condition, there is a homotopy fiber sequence
\[
\mathcal{U}_{\mathrm{loc}}(\mathrm{Gap}^w(\mathcal{F}))\to \mathcal{U}_{\mathrm{loc}}(\mathrm{Gap}(\mathcal{F}))\to \mathcal{U}_{\mathrm{loc}}(\mathcal{C}).
\]
There is a homotopy fiber sequence
\[
\mathcal{U}_{\mathrm{loc}}(\mathcal{A}^{n-1})\to \mathcal{U}_{\mathrm{loc}}(\mathcal{A}^{n})\to \mathcal{U}_{\mathrm{loc}}(\mathcal{A})
\]
where the left map is induced from $(A_1,\cdots,A_{n-1})\mapsto(A_1,A_2\lor \Sigma A_1,\cdots, A_{n-1}\lor\Sigma A_{n-2},\Sigma A_{n-1})$ and the right map is induced by $\lor$. It induces a homotopy fiber sequence
\[
\mathcal{U}_{\mathrm{loc}}(\mathrm{colim}_n\mathcal{A}^{n-1})\to \mathcal{U}_{\mathrm{loc}}(\mathrm{colim}_n\mathcal{A}^{n})\to \mathcal{U}_{\mathrm{loc}}(\mathcal{A})
\]
Consider the diagram
\[
\begin{tikzcd}
    \mathcal{U}_{\mathrm{loc}}(\mathrm{Gap}^w(\mathcal{F}))\ar[r]& \mathcal{U}_{\mathrm{loc}}(\mathrm{Gap}(\mathcal{F}))\ar[r] & \mathcal{U}_{\mathrm{loc}}(\mathcal{C})\\
    \mathcal{U}_{\mathrm{loc}}(\mathrm{colim}_n\mathcal{A}^{n-1})\ar[r]\ar[u,"\mathcal{U}_{\mathrm{loc}}(\tau_\mathcal{F})"] & \mathcal{U}_{\mathrm{loc}}(\mathrm{colim}_n\mathcal{A}^{n})\ar[r]\ar[u,"\mathcal{U}_{\mathrm{loc}}(\mathrm{colim}_nq_{n,\mathcal{F}}^{-1})"] & \mathcal{U}_{\mathrm{loc}}(\mathcal{A})\ar[u,"\mathcal{U}_\mathrm{loc}(\mathcal{F})"]
\end{tikzcd}
\]
Since $\mathcal{U}_{\mathrm{loc}}(\mathrm{colim}_nq_{n,\mathcal{F}}^{-1})$ is a homotopy equivalence by Proposition \ref{prop:additive}, we have a homotopy equivalence by Lemma \ref{lem:diagram1}
\[
\mathcal{U}_{\mathrm{loc}}(\mathrm{Q}(\mathcal{F}))\simeq\mathrm{cofib}(\mathcal{U}_{\mathrm{loc}}(\tau_\mathcal{F}))\simeq\Omega \mathrm{cofib}(\mathcal{U}_{\mathrm{loc}}(\mathcal{F})).
\]
Since $\mathcal{U}_{\mathrm{loc}}$ is into a stable $\infty$-category, we get $\mathcal{U}_{\mathrm{loc}}(\mathrm{Q}(\mathcal{F}))\simeq \mathrm{fib}(\mathcal{U}_{\mathrm{loc}}(\mathcal{F}))$.
\end{proof}
\begin{remark}
    Part of ideas of this Theorem from \cite{Raptis_2018}.
\end{remark}
\begin{cor}\label{cor:quotient}
     Given a localizing invariant $L$, if $\mathcal{F}$ satisfies the d\'evissage condition, $\mathcal{L}(\mathcal{F})$ is a homotopy equivalence if and only if $L(\mathrm{Q}(\mathcal{F}))$ is homotopy trivial.
\end{cor}
\begin{cor}
   If $\mathcal{F}$ satisfies the d\'evissage condition and $K_0(\mathcal{F})$ is an isomorphism, $\mathcal{F}$ induces isomorphisms
    \[
    K_n(\mathcal{F}):K_n(\mathcal{A})\xrightarrow{\simeq} K_n(\mathcal{C})
    \]
    for all $n\geq 1$ if and only if $K_n(\mathrm{Q}(\mathcal{F}))\simeq 0$ for all $n\geq 0$.
\end{cor}
\subsection{Categorization of loop}
Consider the diagonal map $\delta_\mathcal{C}:\mathcal{C}\to\mathcal{C}^2,X\mapsto 
(X,X)$. It is injective on both objects and morphisms.
\begin{lemma}\label{lem:diagdevissage}
    $\delta_\mathbf{C}$ satisfies the weak d\'evissage condition.
\end{lemma}
\begin{proof}
    For $(X,Y)\in \mathcal{C}^2$, it is a direct factor of
    \[
    (X,Y)\oplus (Y,X)\simeq (X\oplus Y,X\oplus Y)\in\delta_\mathcal{C}(\mathcal{C})\subseteq\mathrm{E}(\delta_\mathcal{C}).
    \]
    Hence, $(X,Y)\in\mathrm{Idem}(\mathrm{E}(\delta_\mathcal{C}))$.
\end{proof}
Denote $\theta(\mathcal{C})=\mathrm{Q}(\delta_\mathcal{C})$. Then from Theorem \ref{thm:quotient}, we get the loop theorem.
\begin{thm}\label{thm:loop}
    Let $\mathcal{C}\in\mathrm{Cat}^{\mathrm{perf}}(\mathcal{E})$. Then we have a homotopy equivalence
    \[
    \mathcal{U}_\mathrm{loc}(\theta(\mathcal{C}))\simeq\operatorname{fib}(\mathcal{U}_\mathrm{loc}(\delta_\mathcal{C}))\simeq \Omega \mathcal{U}_\mathrm{loc}(\mathcal{C}).
    \]
    Besides, $\theta$ is functorial.
\end{thm}
\begin{cor}
    Let $\mathcal{F}:\mathcal{C}\to\mathcal{D}$ be an $\mathcal{E}$-linear functor. Then there is a homotopy equivalence
    \[
    \mathcal{U}_\mathrm{loc}(\mathrm{Cone}(\theta(\mathcal{F}))\simeq\operatorname{fib}(\mathcal{U}_\mathrm{loc}(\mathcal{F})).
    \]
\end{cor}
Let $\mathcal{C}\in\mathrm{Cat}^{\mathrm{perf}}(\mathcal{E})$. Define
\[
\mathrm{gr}_n:\mathrm{Gap}([n],\mathcal{C})\to\mathcal{C}^n,X_{\bullet,\bullet}\mapsto(X_{0,1},\cdots,X_{n-1,n})
\]
and $\mathrm{gr}:\mathrm{Gap}(\mathcal{C})\to\operatorname{colim}_n\mathcal{C}^n$ as the colimit $\operatorname{colim}_n\mathrm{gr}_n$. Besides, define some categories of binary sequences by pullbacks
\[
\begin{tikzcd}
    \mathrm{BGap}([n],\mathcal{C})\ar[r]\ar[d]&\mathrm{Gap}([n],\mathcal{C})\ar[d,"\mathrm{gr}_n"]\\
    \mathrm{Gap}([n],\mathcal{C})\ar[r,"\mathrm{gr}_n"]&\mathcal{C}^n
\end{tikzcd},\begin{tikzcd}
    \mathrm{BGap}^w([n],\mathcal{C})\ar[r]\ar[d]&\mathrm{Gap}^w([n],\mathcal{C})\ar[d,"\mathrm{gr}_n"]\\
    \mathrm{Gap}^w([n],\mathcal{C})\ar[r,"\mathrm{gr}_n"]&\mathcal{C}^n
\end{tikzcd}
\]
and
\[
\quad\begin{tikzcd}
    \mathrm{BGap}(\mathcal{C})\ar[r]\ar[d]&\mathrm{Gap}([n],\mathcal{C})\ar[d,"\mathrm{gr}"]\\
    \mathrm{Gap}(\mathcal{C})\ar[r,"\mathrm{gr}"]&\operatorname{colim}_n\mathcal{C}^n
\end{tikzcd},\begin{tikzcd}
    \mathrm{BGap}^w(\mathcal{C})\ar[r]\ar[d]&\mathrm{Gap}(\mathcal{C})\ar[d,"\mathrm{gr}"]\\
    \mathrm{Gap}(\mathcal{C})\ar[r,"\mathrm{gr}"]&\operatorname{colim}_n\mathcal{C}^n
\end{tikzcd}.
\]
Then the identity define diagonal functors $\Delta$.
\begin{thm}\cite[Theorem 1.1]{kasprowski_christoph_2019}\label{thm:loopcofiber}
    Let $\mathcal{C}\in\mathrm{Cat}^{\mathrm{perf}}(\mathcal{E})$. There is a homotopy equivalence
    \[
    \operatorname{cofib}(\mathcal{U}_\mathrm{loc}(\mathrm{Gap}^w(\mathcal{C}))\xrightarrow{\mathcal{U}_\mathrm{loc}(\Delta)} \mathcal{U}_\mathrm{loc}(\mathrm{BGap}^w(\mathcal{C})))\simeq\Omega \mathcal{U}_\mathrm{loc}(\mathcal{C}).
    \]
\end{thm}
$\mathrm{BGap}^w(\mathcal{C})$ consisting of
\begin{enumerate}[label=(\arabic*)]
    \item Objects are $2$-tuples $(X_{\bullet,\bullet},Y_{\bullet,\bullet})$, where $X_{\bullet,\bullet},Y_{\bullet,\bullet}\in\mathrm{Gap}^w(\mathcal{C})$ and $:X_{i-1,i}=Y_{i-1,i}$.
    \item Morphisms between $(X_{\bullet,\bullet},Y_{\bullet,\bullet})$ and $(X'_{\bullet,\bullet},Y'_{\bullet,\bullet})$ are $2$-tuples $(f_{\bullet,\bullet}:X_{\bullet,\bullet}\to X'_{\bullet,\bullet},g_{\bullet,\bullet}:Y_{\bullet,\bullet}\to Y'_{\bullet,\bullet})$ of morphisms in $\mathrm{Gap}^w(\mathcal{C})$ such that $g_{i-1,i}\simeq f_{i-1,i}$ for all $i\geq 1$.
\end{enumerate}
And, $\Delta(X_{\bullet,\bullet})=(X_{\bullet,\bullet},X_{\bullet,\bullet})$.
\begin{prop}\label{prop:binarydevissage}
    $\Delta$ satisfies the weak d\'evissage condition.
\end{prop}
\begin{proof}
For $(X_{\bullet,\bullet},Y_{\bullet,\bullet})\in \mathrm{BGap}^w(\mathcal{C})$, we have
\[
(X_{\bullet,\bullet}\oplus Y_{\bullet,\bullet}, Y_{\bullet,\bullet}\oplus X_{\bullet,\bullet})\simeq (X_{\bullet,\bullet}\oplus Y_{\bullet,\bullet},X_{\bullet,\bullet}\oplus Y_{\bullet,\bullet})\in\mathrm{E}(\Delta),
\]
Then $(X_{\bullet,\bullet},Y_{\bullet,\bullet})$ is a direct factor of $(X_{\bullet,\bullet}\oplus Y_{\bullet,\bullet}, Y_{\bullet,\bullet}\oplus X_{\bullet,\bullet})$, which means it is in $\mathrm{Idem}(\mathrm{E}(\Delta))$. Hence, $\Delta$ satisfies the weak d\'evissage condition.
\end{proof}

Denote $\Theta(\mathcal{C}):=\mathrm{Q}(\Delta)$. This define a functor $\Theta:\mathrm{Cat}^{\mathrm{perf}}\to \mathrm{Cat}^{\mathrm{perf}}$. Then we get the square loop theorem.
\begin{thm}\label{thm:squareloop}
    Let $\mathcal{C}\in\mathrm{Cat}^{\mathrm{perf}}$ be a small stable idempotent-complete $\infty$-category. Then we have a homotopy equivalence
    \[
    \mathcal{U}_\mathrm{loc}(\Theta(\mathcal{C}))\simeq\Omega^2 \mathcal{U}_\mathrm{loc}(\mathcal{C}).
    \]
    Besides, $\Theta$ is functorial.
\end{thm}
\begin{proof}
    By Theorem \ref{thm:quotient} and Proposition \ref{prop:binarydevissage}, we have a homotopy equivalence
    \[
    \mathcal{U}_\mathrm{loc}(\Theta(\mathcal{C}))\simeq\operatorname{fib}(\mathcal{U}_\mathrm{loc}(\mathrm{Gap}^w(\mathcal{C}))\xrightarrow{\mathcal{U}_\mathrm{loc}(\Delta)} \mathcal{U}_\mathrm{loc}(\mathrm{BGap}^w(\mathcal{C}))).
    \]
    By Theorem \ref{thm:loopcofiber}, we have
    \[
    \operatorname{fib}(\mathcal{U}_\mathrm{loc}(\mathrm{Gap}^w(\mathcal{C}))\xrightarrow{\mathcal{U}_\mathrm{loc}(\Delta)} \mathcal{U}_\mathrm{loc}(\mathrm{BGap}^w(\mathcal{C})))\simeq \Omega\operatorname{cofib}(\mathcal{U}_\mathrm{loc}(\mathrm{Gap}^w(\mathcal{C}))\xrightarrow{\mathcal{U}_\mathrm{loc}(\Delta)} \mathcal{U}_\mathrm{loc}(\mathrm{BGap}^w(\mathcal{C})))\simeq\Omega^2\mathcal{U}_\mathrm{loc}(\mathcal{C}).
    \]
    Hence, we have a homotopy equivalence
    \[
    \mathcal{U}_\mathrm{loc}(\Theta(\mathcal{C}))\simeq\Omega^2 \mathcal{U}_\mathrm{loc}(\mathcal{C}).
    \]
\end{proof}
\begin{remark}
    \cite{Maxime_Vladimir_Christoph_2024} gives a functor $\Gamma:\mathrm{Cat}^{\mathrm{perf}}\to \mathrm{Cat}^{\mathrm{perf}}$ with $\mathcal{U}_\mathrm{loc}(\Gamma(\mathcal{C}))\simeq \Omega \mathcal{U}_\mathrm{loc}(\mathcal{C})$: $\Gamma(\mathcal{C}):=\mathrm{Cone}(\Delta)$. It also use Proposition \ref{prop:categorificationofcofiber} and Theorem \ref{thm:loopcofiber}.
\end{remark}
Denote by $\mathrm{Calk}(\mathcal{C}):=\mathrm{Idem}(\mathrm{Ind}(\mathcal{C})^{\omega_1}/\mathcal{C})$ the $\omega_1$-small Calkin category of $\mathcal{C}$. Then $\mathcal{U}_\mathrm{loc}(\mathrm{Calk}(\mathcal{C}))\simeq \Sigma \mathcal{U}_\mathrm{loc}(\mathcal{C})$ since $\mathrm{Ind}(\mathcal{C})^{\omega_1}$ admits an Eilenberg swindle.
\begin{cor}
    Let $\mathcal{C}\in\mathrm{Cat}^{\mathrm{perf}}(\mathcal{E})$. Then
    \[
    \mathcal{U}_\mathrm{loc}(\Theta\circ \mathrm{Calk}(\mathcal{C}))\simeq\Omega \mathcal{U}_\mathrm{loc}(\mathcal{C})\simeq \mathcal{U}_\mathrm{loc}(\mathrm{Calk}\circ \Theta(\mathcal{C}))
    \]
\end{cor}
\begin{remark}
    \cite{hennion_2017} defines $\mathrm{Tate}(\mathcal{C})$ with $K(\mathrm{Tate}(\mathcal{C}))\simeq \Sigma K(\mathcal{C})$. So, we can also use $\mathrm{Tate}$ here.
\end{remark}
\begin{remark}
In fact, we can give a simple proof of Theorem \ref{thm:loopcofiber} now.
\begin{proof}[A proof of Theorem \ref{thm:loopcofiber}]
    Firstly, observation shows that $\mathrm{BGap^w(\mathcal{C})}\simeq \mathrm{Gap}^w(\delta_\mathcal{C})$. Then we have a diagram of cofiber sequences
    \[
    \begin{tikzcd}
        \mathcal{U}_\mathrm{loc}(\mathrm{Gap}^w(\mathcal{C}))\ar[r]\ar[d,"\mathcal{U}_\mathrm{loc}(\Delta)"]&\mathcal{U}_\mathrm{loc}(\mathrm{Gap}(\mathcal{C}))\ar[r]\ar[d,"\mathcal{U}_\mathrm{loc}(\Delta)"]&\mathcal{U}_\mathrm{loc}(\mathcal{C})\ar[d,"\mathcal{U}_\mathrm{loc}(\Delta)"]\\
\mathcal{U}_\mathrm{loc}(\mathrm{Gap}^w(\delta_\mathcal{C}))\ar[r]&\mathcal{U}_\mathrm{loc}(\mathrm{Gap}(\delta_\mathcal{C}))\ar[r]&\mathcal{U}_\mathrm{loc}(\mathcal{C}^2)
    \end{tikzcd}
    \]
    by Proposition \ref{prop:fullyfaithfulofev}. The middle $\mathcal{U}_\mathrm{loc}(\Delta)$ is a homotopy equivalence because the second map and entire map in
    \[
    \mathcal{U}_\mathrm{loc}(\mathrm{Gap}(\mathcal{C}))\xrightarrow{\mathcal{U}_\mathrm{loc}(\Delta)}\mathcal{U}_\mathrm{loc}(\mathrm{Gap}(\mathcal{C}^2,\mathcal{C}))\to \mathcal{U}_\mathrm{loc}(\operatorname{colim}_n\mathcal{C}^n)
    \]
    are homotopy equivalences from Proposition \ref{prop:additive}. Then by Lemma \ref{lem:diagram1}, there is a cofiber sequence
    \[
    \operatorname{cofib}(\mathcal{U}_\mathrm{loc}(\mathrm{Gap}^w(\mathcal{C}))\xrightarrow{\mathcal{U}_\mathrm{loc}(\Delta)} \mathcal{U}_\mathrm{loc}(\mathrm{BGap}^w(\mathcal{C})))\to 0\to \operatorname{cofib}(\mathcal{U}_\mathrm{loc}(\mathcal{C})\xrightarrow{\mathcal{U}_\mathrm{loc}(\Delta)} \mathcal{U}_\mathrm{loc}(\mathcal{C}^2))\simeq \mathcal{U}_\mathrm{loc}(\mathcal{C}).
    \]
    Hence,
    \[
    \operatorname{cofib}(\mathcal{U}_\mathrm{loc}(\mathrm{Gap}^w(\mathcal{C}))\xrightarrow{\mathcal{U}_\mathrm{loc}(\Delta)} \mathcal{U}_\mathrm{loc}(\mathrm{BGap}^w(\mathcal{C})))\simeq\Omega \mathcal{U}_\mathrm{loc}(\mathcal{C}).
    \]
\end{proof}
\end{remark}
\subsection{Remarks on Raptis' d\'evissage theorem}
Ideas of Section 2 and Section 3 are from \cite{Raptis_2018}, which gives a d\'evissage-type theorem for Waldhausen's K-theory \cite{Waldhausen_1985}. Let's recall some notation from this.
\begin{definition}
    Let $\mathcal{C}$ be a Waldhausen category and $\mathcal{A}\subseteq\mathcal{C}$ be a Waldhausen subcategory. We say $(\mathcal{C},\mathcal{A})$ satisfies the \emph{d\'evissage condition} if for every morphism $f:X\to Y$ in $\mathcal{C}$, there is a weak equivalence $g:Y\xrightarrow{\sim}Y'$, $gf:X\to Y'$ admits a factorization
    \[
    X=X_0\rightarrowtail X_1\rightarrowtail\cdots\rightarrowtail X_m\xrightarrow{\sim} Y',
    \]
    where $X_i/X_{i-1}\in\mathscr{A}$ for all $i\geq 1$.
\end{definition}
\begin{definition}
    Let $\mathscr{C}$ be a Waldhausen category. Define $S_n\mathscr{C}$ be the full subcategory of $\mathrm{Fun}([n],\mathscr{C})$ spanned by filtrations
    \[
    X:0=X_0\rightarrowtail\cdots\rightarrowtail X_n.
    \]
    It is also a Waldhausen category. Let $\mathcal{A}\subseteq\mathcal{C}$ be a Waldhausen subcategory. Define $S_n\mathcal{C}_\mathcal{A}$ by the subcategory of $S_n\mathcal{C}$:
    \begin{itemize}
        \item Objects are $X$ such that $X_i/X_{i-1}$ are objects in $\mathcal{A}$ for all $i\geq 1$.
        \item Morphisms are $X\to Y$ such that $X_i/X_{i-1}\to Y_i/Y_{i-1}$ are morphisms in $\mathcal{A}$ for all $i\geq 1$.
    \end{itemize}
\end{definition}
The map $[n+1]\to [n],n+1\mapsto n$ induces $S_n\mathcal{C}\to S_{n+1}\mathcal{C}$ and $S_n\mathcal{C}_\mathcal{A}\to S_{n+1}\mathcal{C}_\mathcal{A}$. Thus, we can define
\[
S_\infty\mathcal{C}:=\operatorname{colim}_nS_n\mathcal{C},S_\infty\mathcal{C}_\mathcal{A}:=\operatorname{colim}_nS_n\mathcal{C}_\mathcal{A}.
\]
\begin{definition}
    For $n\in\mathbb{N}$, define $S_n^w\mathcal{C}_\mathcal{A}$ be the full subcategory of $S_n\mathcal{C}_\mathcal{A}$ spanned by objects $X$ such that $0\to X_n$ is weak equivalence. Define $S^w_\infty\mathcal{C}_\mathcal{A}:=\operatorname{colim}_nS_n^w\mathcal{C}_\mathcal{A}$.
\end{definition}
Then we have a table of correspondence between section 2,3 in this article and \cite{Raptis_2018}.
\begin{table}[h!]
    \centering\small
    \begin{tabular}{ |l|l|l|l|  }
 \hline
 \multicolumn{2}{|p{7cm}|}{Exact functor $\mathcal{F}:\mathcal{A}\to\mathcal{C}$ between idempotent-complete stable $\infty$-categories satisfties the d\'evissage condition}& \multicolumn{2}{|p{7cm}|}{A pair of a good Waldhausen category $\mathcal{C}$ and good Waldhausen subcategory $\mathcal{A}$ satisfties the d\'evissage condition}\\
 \hline
 Result & Label & Result & label\\
 \hline
 $K(\mathrm{Gap}(\mathcal{F}))\simeq\operatorname{colim}_nK(\mathcal{A})^n$&Proposition \ref{prop:additive}&$K(S_\infty\mathcal{C}_\mathcal{A})\simeq \operatorname{colim}_nK(\mathcal{A})^n$&\cite[Proposition 3.2]{Raptis_2018}\\
 \hline
 $K(\mathrm{Gap}^w(\mathcal{F}))\to K(\mathrm{Gap}(\mathcal{F}))\to K(\mathcal{C})$&Proposition \ref{prop:fullyfaithfulofev}&$K(S^w_\infty\mathcal{C}_\mathcal{A})\to K(S\mathcal{C}_\mathcal{A})\to K(\mathcal{C})$&\cite[Proposition 3.7,5.5]{Raptis_2018}\\
 \hline
 $K(\mathrm{Q}(\mathcal{F}))\simeq \operatorname{fib}K(\mathcal{F})$&Theorem \ref{thm:quotient}&Raptis' d\'evissage theorem&\cite[Theorem 5.7]{Raptis_2018}\\
 \hline
\end{tabular}
\end{table}
\section{Fillability}
In this section, we assume $\mathcal{F}:\mathcal{A}\to\mathcal{C}$ is an $\mathcal{E}$-linear functor in $\mathrm{Cat}^\mathrm{perf}(\mathcal{E})$ for $\mathcal{E}\in\mathrm{CAlg}^\mathrm{rig}(\mathrm{Cat}^\mathrm{perf})$.
\subsection{$1$-fillable functors}
Denote by $\mathrm{Gap}^\lor(\mathcal{F})$ the pullback
\[
\begin{tikzcd}
    \mathrm{Gap}^\lor(\mathcal{F})\ar[r]\ar[d]&\mathrm{Gap}(\mathcal{F})\ar[d,"\mathrm{ev}_\mathcal{F}"]\\
    \mathrm{Gap}(\mathcal{F})\ar[r,"\mathrm{ev}_\mathcal{F}"]&\mathcal{C}
\end{tikzcd}
\]
Hence, it is also $\mathcal{E}$-linear. Denote  $\mathrm{Gap}^\lor(\mathcal{C}):=\mathrm{Gap}^\lor(\mathrm{id}_\mathcal{C})$.

Define an $\mathcal{E}$-linear functor $\mathrm{Gap}_\mathcal{F}:\mathrm{Gap}(\mathcal{A})\to \mathrm{Gap}(\mathcal{F})$ by $\mathrm{Gap}(\mathcal{A})\to\mathrm{Gap}(\mathcal{C}),X_{\bullet,\bullet}\mapsto\mathcal{F}(X_{\bullet,\bullet})$ and $\mathrm{gr}:\mathrm{Gap}(\mathcal{A})\to\operatorname{colim}_n\mathcal{A}^n$ in last section. Then it induces an $\mathcal{E}$-linear functor $\mathrm{Gap}^\lor_\mathcal{F}:\mathrm{Gap}^\lor(\mathcal{A})\to \mathrm{Gap}^\lor(\mathcal{F})$.

\begin{definition}
We call \( \mathcal{F} \) \emph{$1$-fillable} if $\mathrm{Gap}^\lor_\mathcal{F}$ satisfies the d\'evissage condition.
\end{definition}
Define $\mathcal{E}$-linear functors
\[
\tau_\perp:\mathrm{Gap}^w(\mathcal{F})\to\mathrm{Gap}^\lor(\mathcal{F}),(A_\bullet,X_{\bullet,\bullet})\mapsto ((A_\bullet,X_{\bullet,\bullet}),(0_\bullet,0_{\bullet,\bullet}))
\]
and
\[
p_\dashv:\mathrm{Gap}^\lor(\mathcal{F})\to\mathrm{Gap}(\mathcal{F}),((A_\bullet,X_{\bullet,\bullet}),(B_\bullet,Y_{\bullet,\bullet}))\mapsto (B_\bullet,Y_{\bullet,\bullet}).
\]
\begin{prop}\label{prop:1fill}
    \[
    \mathrm{Gap}^w(\mathcal{F})\xrightarrow{\tau_\perp}\mathrm{Gap}^\lor(\mathcal{F})\xrightarrow{p_\dashv}\mathrm{Gap}(\mathcal{F})
    \]
    is an exact sequence. Besides, the cofiber sequence
    \[
     K(\mathrm{Gap}^w(\mathcal{F}))\xrightarrow{K(\tau_\perp)}K(\mathrm{Gap}^\lor(\mathcal{F}))\xrightarrow{K(p_\dashv)}K(\mathrm{Gap}(\mathcal{F}))
    \]
    split.
\end{prop}
\begin{proof}
    By \cite[Proposition 5.15]{blumberg_gepner_tabuada_2013}, we only need to prove the exact sequence in the sense of homotopy categories.
    
    For two objects $(A_\bullet,X_{\bullet,\bullet}),(B_\bullet,Y_{\bullet,\bullet})\in \mathrm{Gap}^w(\mathcal{F})$, we have
    \[
    \begin{aligned}
        \mathrm{Map}_{\mathrm{Gap}^\lor(\mathcal{F})}(\tau_\perp(A_\bullet,X_{\bullet,\bullet}),\tau_\perp(B_\bullet,Y_{\bullet,\bullet}))&\simeq\mathrm{Map}_{\mathrm{Gap}(\mathcal{F})}((A_\bullet,X_{\bullet,\bullet}),(B_\bullet,Y_{\bullet,\bullet}))\prod_{\mathrm{Map}_\mathcal{C}(0,0)}\mathrm{Map}_{\mathrm{Gap}(\mathcal{F})}(0,0)\\
        &\simeq\mathrm{Map}_{\mathrm{Gap}(\mathcal{F})}((A_\bullet,X_{\bullet,\bullet}),(B_\bullet,Y_{\bullet,\bullet}))\\
        &\simeq \mathrm{Map}_{\mathrm{Gap}^w(\mathcal{F})}((A_\bullet,X_{\bullet,\bullet}),(B_\bullet,Y_{\bullet,\bullet}))
    \end{aligned}
    \]
    Hence, $\tau_\perp$ is fully faithful. Since $p_\dashv$ admits a section
    \[
    \Delta:\mathrm{Gap}(\mathcal{F})\to \mathrm{Gap}^\lor(\mathcal{F}),(A_\bullet,X_{\bullet,\bullet})\mapsto ((A_\bullet,X_{\bullet,\bullet}),(A_\bullet,X_{\bullet,\bullet})),
    \]
    $p_\dashv$ is essentially surjective on both objects and morphisms. Let $((A_\bullet,X_{\bullet,\bullet}),(B_\bullet,Y_{\bullet,\bullet}))\in\mathrm{Gap}^\lor(\mathcal{F})$, there are two morphisms $(A_\bullet,X_{\bullet,\bullet})\to (C_\bullet,Z_{\bullet,\bullet}),(B_\bullet,Y_{\bullet,\bullet})\to (C_\bullet,Z_{\bullet,\bullet})$ which are stable at identities from Proposition \ref{prop:fullyfaithfulofev}. Hence, there are two morphisms
    \[
    ((A_\bullet,X_{\bullet,\bullet}),(B_\bullet,Y_{\bullet,\bullet}))\to((C_\bullet,Z_{\bullet,\bullet}),(B_\bullet,Y_{\bullet,\bullet})),((B_\bullet,Y_{\bullet,\bullet}),(B_\bullet,Y_{\bullet,\bullet}))\to((C_\bullet,Z_{\bullet,\bullet}),(B_\bullet,Y_{\bullet,\bullet}))
    \]
    whose cofibers are in $\tau_\perp(\mathrm{Gap}^w(\mathcal{F}))$. Hence,
    \[
    \mathrm{Gap}^\lor(\mathcal{F})/\mathrm{Gap}^w(\mathcal{F})\to \mathrm{Gap}(\mathcal{F})
    \]
    is injective on objects. Similarly, from Proposition \ref{prop:fullyfaithfulofev}, it is also injective on morphisms. Hence, it is an equivalence in the sense of homotopy categories.

    Then we have a cofiber sequence
    \[
     K(\mathrm{Gap}^w(\mathcal{F}))\xrightarrow{K(\tau_\perp)}K(\mathrm{Gap}^\lor(\mathcal{F}))\xrightarrow{K(p_\dashv)}K(\mathrm{Gap}(\mathcal{F})).
    \]
    Since $K(p_\dashv)$ admits a section $K(\Delta)$, it is split.
\end{proof}
\begin{lemma}\label{lem:fillable}\quad
    \begin{enumerate}[label=(\arabic*)]
        \item $\mathrm{Gap}_\mathcal{F}$ satisfies the d\'evissage condition.
        \item $\mathrm{Gap}^\lor_\mathcal{F}$ satisfies the weak d\'evissage condition.
    \end{enumerate}
\end{lemma}
\begin{proof}\quad
\begin{enumerate}[label=(\arabic*)]
    \item Assume $(A_\bullet,X_{\bullet,\bullet})\in\mathrm{Gap}([n],\mathcal{F})$. Then
    \[
    \begin{tikzcd}
        X_{\bullet,\bullet}^n:&0\ar[r]& X_{0,1}\ar[r]&\cdots\ar[r]&X_{0,n-1}\ar[r]&X_{0,n}\\
        X_{\bullet,\bullet}^{n-1}:&0\ar[r]\ar[u]& X_{0,1}\ar[r]\ar[u]&\cdots\ar[r]&X_{0,n-1}\ar[r]\ar[u]&X_{0,n-1}\ar[u]\\
        &\vdots&\vdots&\ddots&\vdots&\vdots\\
        X_{\bullet,\bullet}^1:&0\ar[r]\ar[u]& X_{0,1}\ar[r]\ar[u]&\cdots\ar[r]&X_{0,1}\ar[r]\ar[u]&X_{0,1}\ar[u]\\
        X_{\bullet,\bullet}^0:&0\ar[r]\ar[u]& 0\ar[r]\ar[u]&\cdots\ar[r]&0\ar[r]&0\ar[u]\\
    \end{tikzcd}
    \]
    and
    \[
    A_\bullet^k=(A_1,\cdots,A_k,0,\cdots,0),\quad 0\leq k\leq n
    \]
    gives the desired filtration.
    \item Assume $(A_\bullet,X_{\bullet,\bullet},B_\bullet,Y_{\bullet,\bullet})\in\mathrm{Gap}^\lor(\mathcal{F})$. By $(1)$,
    \[
    (A_\bullet\oplus B_\bullet,X_{\bullet,\bullet}\oplus Y_{\bullet,\bullet},A_\bullet\oplus B_\bullet,X_{\bullet,\bullet}\oplus Y_{\bullet,\bullet})\in\mathrm{E}(\mathrm{Gap}^\lor_\mathcal{F}).
    \]
    Since $(A_\bullet,X_{\bullet,\bullet},B_\bullet,Y_{\bullet,\bullet})$ is a direct summand of $(A_\bullet\oplus B_\bullet,X_{\bullet,\bullet}\oplus Y_{\bullet,\bullet},A_\bullet\oplus B_\bullet,X_{\bullet,\bullet}\oplus Y_{\bullet,\bullet})$, we have
    \[
    (A_\bullet,X_{\bullet,\bullet},B_\bullet,Y_{\bullet,\bullet})\in \mathrm{Idem}(\mathrm{E}(\mathrm{Gap}^\lor_\mathcal{F})).
    \]
    Thus, $\mathrm{Gap}^\lor_\mathcal{F}$ satisfies the weak d\'evissage condition.
\end{enumerate}
\end{proof}
\begin{thm}\label{thm:1fillableK0}
     $\mathcal{F}$ is $1$-fillable if and only if $K_0(\mathrm{Q}(\mathcal{F}))=0$.
\end{thm}
\begin{proof}
    $\mathrm{Gap}_\mathcal{F}$ restricts to a functor $\mathrm{Gap}^w_\mathcal{F}:\mathrm{Gap}^w(\mathcal{A})\to\mathrm{Gap}^w(\mathcal{F})$. Consider the diagram
    \[
    \begin{tikzcd}
        \mathrm{Gap}^w(\mathcal{A})\ar[r]\ar[d,"\mathrm{Gap}^w_\mathcal{F}"]&\mathrm{Gap}^\lor(\mathcal{A})\ar[r]\ar[d,"\mathrm{Gap}^\lor_\mathcal{F}"]&\mathrm{Gap}(\mathcal{A})\ar[d,"\mathrm{Gap}_\mathcal{F}"]\\
        \mathrm{Gap}^w(\mathcal{F})\ar[r]&\mathrm{Gap}^\lor(\mathcal{F})\ar[r]&\mathrm{Gap}(\mathcal{F})
    \end{tikzcd}
    \]
    By Proposition \ref{prop:1fill}, $K_0(\mathrm{Gap}^\lor_\mathcal{F})$ is just the direct sum of $K_0(\mathrm{Gap}^w_\mathcal{F})$ and $K_0(\mathrm{Gap}_\mathcal{F})$. Since $K(\mathrm{Gap}_\mathcal{F})$ is an equivalence by Proposition \ref{prop:additive}, By Lemma \ref{lem:devissageK_0}, $K_0(\mathrm{Gap}^w_\mathcal{F})$ is an epimorphism if and only if $\mathcal{F}$ is $1$-fillable. Consider the diagram of two split cofiber sequence
    \[
    \begin{tikzcd}
        K(\operatorname{colim}_n\mathcal{A}^n)\ar[r,"K(\tau_{\mathrm{id}_\mathcal{A}})"]\ar[d,"\mathrm{id}"]&K(\mathrm{Gap}^w(\mathcal{A}))\ar[r]\ar[d,"{K(\mathrm{Gap}^w_\mathcal{F})}"]&K(\mathrm{Q}(\mathrm{id}_\mathcal{A}))
        \ar[d]\\
        K(\operatorname{colim}_n\mathcal{A}^n)\ar[r,"{K(\tau_\mathcal{F})}"]&K(\mathrm{Gap}^w(\mathcal{F}))\ar[r]&K(\mathrm{Q}(\mathcal{F}))
    \end{tikzcd}
    \]
    Then $K_0(\mathrm{Q}(\mathrm{id}_\mathcal{A}))\to K_0(\mathrm{Q}(\mathcal{F}))$ is an epimorphism if and only if $\mathcal{F}$ is $1$-fillable. Since $K(\mathrm{Q}(\mathrm{id}_\mathcal{A}))\simeq\operatorname{fib}(K(\mathrm{id}_\mathcal{A}))=0$ by Theorem \ref{thm:quotient}, $K_0(\mathrm{Q}(\mathcal{F}))=0$ if and only if $\mathcal{F}$ is $1$-fillable.
\end{proof}
\begin{example}
    $\mathrm{id}_\mathcal{C}$ is always $1$-fillable: $\mathrm{Gap}^\lor(\mathrm{id}_\mathcal{C})\simeq\mathrm{Gap}^\lor(\mathcal{C})$ and $\mathrm{Gap}^\lor_{\mathrm{id}_\mathcal{C}}\simeq\mathrm{id}_{\mathrm{Gap}^\lor(\mathcal{C})}$ is also the identity map.
\end{example}
$\mathrm{Gap}([n],\mathrm{Gap}([n],\mathcal{C}))$ is a full subcategory of $\mathrm{Fun}(\mathrm{N}(\mathrm{Ar}[n]\times\mathrm{Ar}[n]),\mathcal{C})$, whose objects are donoted by $X_{\bullet,\bullet}^{\bullet,\bullet}$, where for every $i,j$, $X_{i,j}^{\bullet,\bullet},X_{\bullet,\bullet}^{i,j}$ are objects in $\mathrm{Gap}(\mathcal{C})$. By \cite[Lemma 7.3]{blumberg_gepner_tabuada_2013}, an object $X_{\bullet,\bullet}^{\bullet,\bullet}$ is decided by $X^{0,\bullet}_{0,\bullet}$, the $0$-the grid. Hence, an object $X_{\bullet,\bullet}^{\bullet,\bullet}$ in $\mathrm{Gap}(\mathrm{Gap}(\mathcal{C}))$ has two sides and one vertex at infinity, which is an object in $\mathrm{Gap}^\lor(\mathcal{C})$. This actually means that, whether we look at it from rows or columns, it is an object within $\mathrm{Gap}(\mathrm{Gap}(\mathcal{C}))$. From this, we stipulate the evaluation functor is in the sense of columns:
\[
\mathrm{ev}_{\mathrm{id}_{\mathrm{Gap}(\mathcal{C})}}:\mathrm{Gap}(\mathrm{Gap}(\mathcal{C}))\to\mathrm{Gap}(\mathcal{C}), X_{\bullet,\bullet}^{\bullet,\bullet}\mapsto \left(0\to\operatorname{colim}_nX_{0,n}^{0,1}\to\cdots\to \operatorname{colim}_nX_{0,n}^{0,m}\to\cdots\right)
\]
And, we have a "transpose" functor, which swaps rows and columns:
\[
T_\mathcal{C}:\mathrm{Gap}(\mathrm{Gap}(\mathcal{C}))\to \mathrm{Gap}(\mathrm{Gap}(\mathcal{C})), T_\mathcal{C}(X_{\bullet,\bullet}^{\bullet,\bullet})_{i,j}^{k,l}=X_{k,l}^{i,j},
\]
which induces a transpose evaluation: $\mathrm{ev}_{\mathrm{id}_{\mathrm{Gap}(\mathcal{C})}}^T:=\mathrm{ev}_{\mathrm{id}_{\mathrm{Gap}(\mathcal{C})}}\circ T_\mathcal{C}$.

Objects $\mathrm{Gap}(\mathrm{Gap}_\mathcal{F})$ are three tuples $(A_{\bullet,\bullet}^\bullet,X_{\bullet,\bullet}^{\bullet,\bullet},A^{\bullet,\bullet}_{\bullet})$, where $A_{\bullet,\bullet}^\bullet,A^{\bullet,\bullet}_\bullet\in\operatorname{colim}_n(\mathrm{Gap}(\mathcal{A}))^n$ and $X_{\bullet,\bullet}^{\bullet,\bullet}\in\mathrm{Gap}(\mathrm{Gap}(\mathcal{C}))$ such that for each $i,j$, $(A^\bullet_{i,j},X_{i,j}^{\bullet,\bullet}),(A_\bullet^{i,j},X_{\bullet,\bullet}^{i,j})\in\mathrm{Gap}(\mathcal{F})$. We also stipulate the evaluation, transpose and the transpose evaluation functors as follows:
\[
\begin{aligned}
    \mathrm{ev}_{\mathrm{Gap}_{\mathcal{F}}}:\mathrm{Gap}(\mathrm{Gap}_\mathcal{F})&\to\mathrm{Gap}(\mathcal{F}),(A_{\bullet,\bullet}^\bullet,X_{\bullet,\bullet}^{\bullet,\bullet},A^{\bullet,\bullet}_{\bullet})\mapsto \left(\operatorname{colim}_nA_{0,n}^\bullet,\mathrm{ev}_{\mathrm{id}_{\mathrm{Gap}(\mathcal{C})}}(X_{\bullet,\bullet}^{\bullet,\bullet})\right)\\
    T_\mathcal{F}:\mathrm{Gap}(\mathrm{Gap}_\mathcal{F})&\to \mathrm{Gap}(\mathrm{Gap}_\mathcal{F}),(A_{\bullet,\bullet}^\bullet,X_{\bullet,\bullet}^{\bullet,\bullet},A^{\bullet,\bullet}_{\bullet})\mapsto (A_\bullet^{\bullet,\bullet},T_\mathcal{C}(X_{\bullet,\bullet}^{\bullet,\bullet}),A_{\bullet,\bullet}^\bullet)\\
    \mathrm{ev}_{\mathrm{Gap}_{\mathcal{F}}}^T&:=\mathrm{ev}_{\mathrm{Gap}_{\mathcal{F}}}\circ T_\mathcal{F}
\end{aligned}
\]
Then we can define a "forgetful" functor $U_\mathcal{F}:\mathrm{Gap}(\mathrm{Gap}_\mathcal{F})\to\mathrm{Gap}^\lor(\mathcal{F})$ by $(\mathrm{ev}_{\mathrm{Gap}_{\mathcal{F}}}^T,\mathrm{ev}_{\mathrm{Gap}_{\mathcal{F}}})$.

the composition
\[
\mathrm{Gap}(\mathrm{Gap}_\mathcal{F})\to\mathrm{Gap}(\mathrm{Gap}^\lor_\mathcal{F})\xrightarrow{\mathrm{ev}_{\mathrm{Gap}^\lor_\mathcal{F}}}\mathrm{Gap}^\lor(\mathcal{F}).
\]
\begin{prop}\label{prop:1fillable}
    The following are equivalent:
    \begin{enumerate}[label=(\arabic*)]
        \item $\mathcal{F}$ is $1$-fillable.
        \item $U_\mathcal{F}$ is essentially surjective.
    \end{enumerate}
\end{prop}
\begin{proof}
    By Corollary \ref{cor:fullyfaithfulofev}, $\mathcal{F}$ is $1$-fillable if and only if  $\mathrm{ev}_{\mathrm{Gap}^\lor_\mathcal{F}}$ is essentially surjective. 
    
    $U_\mathcal{F}$ is the composition
    \[
    \mathrm{Gap}(\mathrm{Gap}_\mathcal{F})\to\mathrm{Gap}(\mathrm{Gap}^\lor_\mathcal{F})\simeq \mathrm{Gap}^\lor(\mathrm{Gap}_\mathcal{F})\xrightarrow{\mathrm{ev}_{\mathrm{Gap}^\lor_\mathcal{F}}}\mathrm{Gap}^\lor(\mathcal{F}),
    \]
    where the first map $\mathrm{Gap}(\mathrm{Gap}_\mathcal{F})\to \mathrm{Gap}^\lor(\mathrm{Gap}_\mathcal{F})$ is defined by $(T_\mathcal{F},\mathrm{id})$.
    
    Hence, $(2)\Rightarrow (1)$ is obvious. Now consider $(1)\Rightarrow (2)$.

    For $((A_\bullet,X_{\bullet,\bullet}),(B_{\bullet},Y_{\bullet,\bullet}))\in\mathrm{Gap}^\lor(\mathcal{F})$, assume there is an object
    \[
    ((A^\bullet_{\bullet,\bullet},X_{\bullet,\bullet}^{\bullet,\bullet},A_\bullet^{\bullet,\bullet}),(B^\bullet_{\bullet,\bullet},Y_{\bullet,\bullet}^{\bullet,\bullet},B_\bullet^{\bullet,\bullet}))\in\mathrm{Gap}([n],\mathrm{Gap}^\lor_\mathcal{F})
    \]
    such that $(A_{0,n}^\bullet,X_{0,n}^{\bullet,\bullet})=(A_\bullet,X_{\bullet,\bullet}),(B^\bullet_{0,n},Y_{0,n}^{\bullet,\bullet})=(B_\bullet,Y_{\bullet,\bullet})$. Also, we can assume $X_{0,i}^{\bullet,\bullet},Y_{0,i}^{\bullet,\bullet}\in\mathrm{Gap}([m],\mathcal{F})$ for all $1\leq i\leq n$. It means $X_{0,i}^{0,m}=Y_{0,i}^{0,m}$ for all $1\leq i\leq n$. Then define $Z_{\bullet,\bullet}^{\bullet,\bullet}\in\mathrm{Gap}(\mathrm{Gap}(\mathcal{C}))$ by
    \[
        \begin{tikzcd}
            0\ar[r]\ar[d]&0\ar[r]\ar[d]&\cdots\ar[r]&0\ar[r]\ar[d]&0\ar[d]\ar[r]&\cdots\ar[r]&0\ar[d]\\
            0\ar[r]\ar[d]&\tilde{X}_{0,1}^{0,1}\ar[r]\ar[d]&\cdots\ar[r]&\tilde{X}^{0,m}_{0,1}=Y_{0,1}^{0,1}\ar[r]\ar[d]&Y^{0,1}_{0,2}\ar[d]\ar[r]&\cdots\ar[r]&Y^{0,1}_{0,n}\ar[d]\\
            \vdots\ar[d]&\vdots\ar[d]&\ddots&\vdots\ar[d]&\vdots\ar[d]&\ddots&\vdots\ar[d]\\
            0\ar[r]\ar[d]&\tilde{X}_{0,m}^{0,1}=X^{0,1}_{0,1}\ar[r]\ar[d]&\cdots\ar[r]&\tilde{X}_{0,m}^{0,m}=X^{0,m}_{0,1}\ar[r]\ar[d]&X^{0,m}_{0,2}\ar[d]\ar[r]&\cdots\ar[r]&X^{0,m}_{0,n}\ar[d]\\
            0\ar[r]\ar[d]&X^{0,1}_{0,2}\ar[r]\ar[d]&\cdots\ar[r]&X^{0,m}_{0,2}\ar[r]\ar[d]&X^{0,m}_{0,2}\ar[d]\ar[r]&\cdots\ar[r]&X^{0,m}_{0,n}\ar[d]\\
            \vdots\ar[d]&\vdots\ar[d]&\ddots&\vdots\ar[d]&\vdots\ar[d]&\ddots&\vdots\ar[d]\\
            0\ar[r]&X^{0,1}_{0,n}\ar[r]&\cdots\ar[r]&X^{0,m}_{0,n}\ar[r]&X^{0,m}_{0,n}\ar[r]&\cdots\ar[r]&X^{0,m}_{0,n}
        \end{tikzcd}
        \]
        where for $1\leq i,j\leq m, \tilde{X}^{0,i}_{0,j}:=X^{0,i}_{0,1}\prod_{X^{0,m}_{0,1}}Y^{0,j}_{0,1}$. Besides, define $C_{\bullet,\bullet}^\bullet,C_\bullet^{\bullet,\bullet}$ by:
        \[
        C_{0,j}^i=\begin{cases}
            A_{0,1}^i,\quad 1\leq i,j\leq m\\
            B_{0,j-m+1}^i,\quad 1\leq i\leq m,m<j\leq m+n-1\\
            A^{0,i-m+1}_j,\quad 1\leq j\leq m, m<i\leq m+n-1\\
            A^{j-m+1}_{0,i-m+1},\quad m<i,j\leq m+n-1\\
            0,\quad\text{otherwise}
        \end{cases}
        \]
        and
        \[
        C^{0,j}_i=\begin{cases}
            B^{0,1}_i,\quad 1\leq i,j\leq m\\
            A^{0,j-m+1}_i,\quad 1\leq i\leq m,m<j\leq m+n-1\\
            B_{0,i-m+1}^j,\quad 1\leq j\leq m, m<i\leq m+n-1\\
            B^{j-m+1}_{0,i-m+1},\quad m<i,j\leq m+n-1\\
            0,\quad\text{otherwise}
        \end{cases}
        \]
        
        The image of $(C_{\bullet,\bullet}^\bullet,X_{\bullet,\bullet}^{\bullet,\bullet},C_\bullet^{\bullet,\bullet})$ under $U_\mathcal{F}$ is just $((A_\bullet,X_{\bullet,\bullet}),(B_{\bullet},Y_{\bullet,\bullet}))$.
\end{proof}
\begin{remark}
    The essentially surjection of $U_\mathcal{F}$ just means that given two infinity sides with the same infinity vertex in $\mathrm{Gap}^\lor(\mathcal{F})$, we can "fill" it into a whole grid in $\mathrm{Gap}(\mathrm{Gap}_\mathcal{F})$. This is the origin of the name "fillability".
\end{remark}
\subsection{D\'evissage Theorem for Algebraic K-theory}
 Define $\mathcal{E}$-linear functors $\mathcal{F}_n:\mathcal{A}_n\to\mathcal{C}_n$ by
\begin{itemize}
    \item $\mathcal{C}_1:=\mathcal{C},\mathcal{A}_1:=\mathcal{A},\mathcal{F}_1:=\mathcal{F}$.
    \item For $n\geq 2$, $\mathcal{C}_n:=\mathrm{Gap}^\lor(\mathcal{F}_{n-1}),\mathcal{A}_n:=\mathrm{Gap}^\lor(\mathcal{A}_{n-1})$ and $\mathcal{F}_n:=\mathrm{Gap}^\lor_{\mathcal{F}_{n-1}}$.
\end{itemize}
\begin{definition}
    $\mathcal{F}$ is called \emph{$n$-fillable} if for every $1\leq i\leq n$, $\mathcal{F}_i$ is $1$-fillable. $\mathcal{F}$ is called \emph{fillable} for every $i\geq 1$, $\mathcal{F}_i$ is $1$-fillable.
\end{definition}
\begin{lemma}\label{lem:inductionofquotient}
   Let $\mathcal{F}:\mathcal{A}\to\mathcal{C}\in\mathrm{Cat}^{\mathrm{perf}}$ be an exact functor between small stable idempotent-complete $\infty$-categories.   Then $\mathcal{U}_\mathbf{loc}(\mathrm{Q}(\mathcal{F}_{k+1}))\simeq\Omega \mathcal{U}_\mathbf{loc}(\mathrm{Q}(\mathcal{F}_k))$ for all $k\geq 1$. As a consequence, $\mathcal{U}_\mathbf{loc}(\mathrm{Q}(\mathcal{F}_{k+1}))\simeq\Omega^k \mathcal{U}_\mathbf{loc}(\mathrm{Q}(\mathcal{F}))$.
\end{lemma}
\begin{proof}
    By Theorem \ref{thm:quotient} and Lemma \ref{lem:fillable}, there is a cofiber sequence
    \[
    \mathcal{U}_\mathbf{loc}(\mathrm{Q}(\mathcal{F}_{k+1}))\simeq\operatorname{fib}(\mathcal{U}_\mathbf{loc}(\mathcal{F}_{k+1}))).
    \]
    By Proposition \ref{prop:1fill},
    \[
    \mathcal{U}_\mathbf{loc}(\mathcal{A}_{k+1})\simeq \mathcal{U}_\mathbf{loc}(\mathrm{Gap}^w(\mathcal{A}_k))\oplus \mathcal{U}_\mathbf{loc}(\mathrm{Gap}(\mathcal{A}_k)),\mathcal{U}_\mathbf{loc}(\mathcal{C}_{k+1})\simeq \mathcal{U}_\mathbf{loc}(\mathrm{Gap}^w(\mathcal{F}_k))\oplus \mathcal{U}_\mathbf{loc}(\mathrm{Gap}(\mathcal{F}_k)).
    \]
    Hence,
    \[
    \begin{aligned}
        &\operatorname{fib}(\mathcal{U}_\mathbf{loc}(\mathcal{F}_{k+1}))\\
        \simeq&\operatorname{fib}(\mathcal{U}_\mathbf{loc}(\mathrm{Gap}^w_{\mathcal{F}_k}))\oplus\operatorname{fib}(\mathcal{U}_\mathbf{loc}(\mathrm{Gap}_{\mathcal{F}_k})).
    \end{aligned}
    \]
    The first part $\operatorname{fib}(\mathcal{U}_\mathbf{loc}(\mathrm{Gap}^w_{\mathcal{F}_k}))$ is just $\Omega \mathcal{U}_\mathbf{loc}(\mathrm{Q}(\mathcal{F}_k))$. And since $\mathcal{U}_\mathbf{loc}(\mathrm{Gap}(\mathcal{F}^k))$ is an equivalence from Proposition \ref{prop:additive}, $\operatorname{fib}(\mathcal{U}_\mathbf{loc}(\mathrm{Gap}_{\mathcal{F}_k}))\simeq0$. Thus,
    \[
    \mathcal{U}_\mathbf{loc}(\mathrm{Q}(\mathcal{F}_{k+1}))\simeq\Omega \mathcal{U}_\mathbf{loc}(\mathrm{Q}(\mathcal{F}_k)).
    \]
\end{proof}
\begin{thm}\label{thm:nfillableKn}
    $\mathcal{F}$ is $n$-fillable if and only if $K_i(\mathrm{Q}(\mathcal{F}))=0$ for all $0\leq i\leq n-1$.
\end{thm}
\begin{proof}
    By Theorem \ref{thm:1fillableK0}, $\mathcal{F}$ is $n$-fillable if and only if $K_0(\mathrm{Q}(\mathcal{F}_i))=0$ for all $1\leq i\leq n$. By Lemma \ref{lem:inductionofquotient}, $K_i(\mathrm{Q}(\mathcal{F}))=0$ for all $0\leq i\leq n-1$.
\end{proof}
\begin{thm}\label{mainresult1}(Theorem \ref{thmA})
    Let $\mathcal{F}:\mathcal{A}\to\mathcal{C}\in\mathrm{Cat}^{\mathrm{perf}}(\mathcal{E})$ be an $\mathcal{E}$-linear functor. If $\mathcal{F}$ satisfies the weak d\'evissage condition(resp. d\'evissage condition), then the following statements are equivalent:
    \begin{enumerate}[label=(\arabic*)]
        \item $\mathcal{F}$ is $n$-fillable, and 
        \item $\mathcal{F}$ induces isomorphisms $K_i(\mathcal{A})\xrightarrow{\simeq}K_i(\mathcal{C})$ for all $1\leq i\leq n-1$, an epimorphism $K_n(\mathcal{A})\twoheadrightarrow K_n(\mathcal{C})$ and a monomorphism (resp. an isomorphism) $K_0(\mathcal{A})\hookrightarrow K_0(\mathcal{C})$.
    \end{enumerate}
\end{thm}
\begin{proof}
    Just from Theorem \ref{thm:nfillableKn} and Theorem \ref{thm:quotient}.
\end{proof}
\begin{cor}\label{cor:mainresult1}
    If $\mathcal{F}$ satisfies the weak d\'evissage condition(resp. d\'evissage condition), then the following statements are equivalent:
    \begin{enumerate}[label=(\arabic*)]
        \item $\mathcal{F}$ is fillable, and 
        \item $\mathcal{F}$ induces isomorphisms $K_n(\mathcal{A})\xrightarrow{\simeq}K_n(\mathcal{C})$ for all $n\geq 1$, and a monomorphism (resp. an isomorphism) $K_0(\mathcal{A})\hookrightarrow K_0(\mathcal{C})$.
    \end{enumerate}
\end{cor}
\begin{definition}
    Let $\mathcal{C}\in\mathrm{Cat}^\mathrm{perf}(\mathcal{E})$. $\mathcal{C}$ is called \emph{$n$-fillable} if the delta map $\delta_\mathcal{C}:\mathcal{C}\to\mathcal{C}^2$ is $n$-fillable. $\mathcal{C}$ is called \emph{fillable} if $\delta_\mathcal{C}$ is fillable.
\end{definition}
\begin{thm}\label{mainresult2}(Theorem \ref{thmB})
    Let $\mathcal{C}$ be a small idempotent-complete $\infty$-category. Then $K_i(\mathcal{C})=0$ for all $1\leq i\leq n$ if and only if $\mathcal{C}$ is $n$-fillable.
\end{thm}
\begin{proof}
    By Theorem \ref{mainresult1} and Lemma \ref{lem:diagdevissage}, $K_i(\mathcal{C})\to K_i(\mathcal{C}^2)$ are surjective for all $1\leq i\leq n$, which forces them are trivial.
\end{proof}
\begin{cor}\label{cor:mainresult2}
        Let $\mathcal{C}$ be a small idempotent-complete $\infty$-category. Then $K_n(\mathcal{C})=0$ for all $n\geq 1$ if and only if $\mathcal{C}$ is fillable.
\end{cor}
\subsection{Generic D\'evissage Theorem for Localizing Invariants}
\begin{definition}
    Denote by
    \begin{itemize}
        \item $\mathfrak{S}_\mathrm{devissage}(\mathcal{E})$, the class of all $\mathcal{E}$-linear functors which satisfy the d\'evissage condition.
        \item  $\mathfrak{S}_\mathrm{devissage}^w(\mathcal{E})$, the class of all $\mathcal{E}$-linear functors which satisfy the weak d\'evissage condition.
        \item $\mathfrak{S}_\mathrm{n-fillability}(\mathcal{E})$, the class of all $n$-fillable $\mathcal{E}$-linear functors for $n\geq 1$.
        \item $\mathfrak{S}_\mathrm{fillability}(\mathcal{E})$, the class of all fillable $\mathcal{E}$-linear functors.
    \end{itemize}
    In particular, denote by $\mathfrak{S}_\mathrm{devissage}:=\mathfrak{S}_\mathrm{devissage}(\mathrm{Sp}^\omega),\mathfrak{S}^w_\mathrm{devissage}:=\mathfrak{S}^w_\mathrm{devissage}(\mathrm{Sp}^\omega),\mathfrak{S}_\mathrm{n-fillability}:=\mathfrak{S}_\mathrm{n-fillability}(\mathrm{Sp}^\omega)$ and $\mathfrak{S}_\mathrm{fillability}:=\mathfrak{S}_\mathrm{fillability}(\mathrm{Sp}^\omega)$.
\end{definition}
\begin{definition}
    Let $\mathfrak{S}$ be a class of some $\mathcal{E}$-linear functors and $f:\mathrm{Fun}(\Delta^1,\mathrm{Cat}^\mathrm{perf}(\mathcal{E}))\to \mathrm{Fun}(\Delta^1,\mathrm{Cat}^\mathrm{perf}(\mathcal{E}))$ be a functor. For $n\geq 1$, denote by
    \begin{itemize}
        \item $f^n(\mathfrak{S})$ the class of all $\mathcal{E}$-linear functors with form $f^n(\mathcal{F})$ for $\mathcal{F}\in\mathfrak{S}$.
        \item $f^{-n}(\mathfrak{S})$ the class of all $\mathcal{E}$-linear functors $\mathcal{F}$ such that $f^n(\mathcal{F})\in\mathfrak{S}$.
    \end{itemize}
\end{definition}
\begin{definition}
    A class $\mathfrak{S}$ of some $\mathcal{E}$-linear functors is called \emph{fillable} if the following equivalent statements hold:
    \begin{enumerate}[label=(\arabic*)]
        \item $\mathrm{Gap}^\lor_{(-)}(\mathfrak{S}\cap\mathfrak{S}_\mathrm{1-fillability}(\mathcal{E}))\subseteq\mathfrak{S}$.
        \item $(\mathrm{Gap}^\lor_{(-)})^k(\mathfrak{S}\cap\mathfrak{S}_\mathrm{n-fillability}(\mathcal{E}))\subseteq\mathfrak{S}$ for every $1\leq k\leq n$ and every $n\geq 1$.
    \end{enumerate}
\end{definition}
\begin{thm}[Generic D\'evissage Theorem]\label{mainresult3}
    Suppose $\mathfrak{S}$ is a fillable class of some $\mathcal{E}$-linear functors and $L:\mathrm{Cat}^\mathrm{perf}(\mathcal{E})\to\mathrm{Sp}$ is a localizing invariant such that $L_m$ is $\mathfrak{S}$-epimorphic for some $m\in\mathbb{Z}$. Then
      \begin{enumerate}[label=(\arabic*)]
          \item $L_m$ is $\mathfrak{S}_\mathrm{devissage}^w(\mathcal{E})\cap\mathfrak{S}_\mathrm{1-fillability}(\mathcal{E})\cap(\mathrm{Gap}^\lor_{(-)})^{-1}\mathfrak{S}$-monomorphic.
          \item $L_m$ is $\mathfrak{S}_\mathrm{devissage}(\mathcal{E})\cap\mathfrak{S}_\mathrm{1-fillability}(\mathcal{E})\cap\mathfrak{S}$-invariant.
          \item $L_{m+n}$ is $\mathfrak{S}_\mathrm{devissage}^w(\mathcal{E})\cap\mathfrak{S}_\mathrm{n-fillability}(\mathcal{E})\cap(\mathrm{Gap}^\lor_{(-)})^{-1}\mathfrak{S}$-epimorphic for every $n\geq 1$.
          \item $L_{m+k}$ is $\mathfrak{S}_\mathrm{devissage}^w(\mathcal{E})\cap\mathfrak{S}_\mathrm{n-fillability}(\mathcal{E})\cap(\mathrm{Gap}^\lor_{(-)})^{-1}\mathfrak{S}$-invariant for every $0<k<n$.
      \end{enumerate}
\end{thm}
\begin{proof}\quad
    \begin{enumerate}[label=(\arabic*)]
    \item Let $(\mathcal{F}:\mathcal{A}\to\mathcal{C})\in \mathfrak{S}_\mathrm{devissage}^w(\mathcal{E})\cap\mathfrak{S}_\mathrm{1-fillability}(\mathcal{E})\cap(\mathrm{Gap}^\lor_{(-)})^{-1}\mathfrak{S}$. By Proposition \ref{prop:1fill} and Theorem \ref{thm:quotient}, $L_m(\mathrm{Gap}^\lor(\mathcal{F}))\simeq L_m(\mathrm{Gap}^\lor(\mathcal{A}))\oplus\pi_m\operatorname{fib}(L(\mathcal{F}))$. Since $\mathrm{Gap}^\lor(\mathcal{A})\to\mathrm{Gap}^\lor(\mathcal{F})$ satisfies the d\'evissage condition and $L_m$ is $\mathfrak{S}_\mathrm{devissage}(\mathcal{E})$-epimorphic,
    \[
    L_m(\mathrm{Gap}^\lor(\mathcal{A}))\to L_m(\mathrm{Gap}^\lor(\mathcal{F}))\simeq L_m(\mathrm{Gap}^\lor(\mathcal{A}))\oplus\pi_m\operatorname{fib}(L(\mathcal{F}))
    \]
    is an epimorphism. Then $\pi_m\operatorname{fib}(L(\mathcal{F}))=0$. Hence, $L_m(\mathcal{F})$ is a monomorphism.
    \item On the one hand, since $\mathfrak{S}_\mathrm{devissage}(\mathcal{E})\cap\mathfrak{S}_\mathrm{1-fillability}(\mathcal{E})\cap\mathfrak{S}\subseteq \mathfrak{S}^w_\mathrm{devissage}(\mathcal{E})\cap\mathfrak{S}_\mathrm{1-fillability}(\mathcal{E})\cap\mathfrak{S}$, by $(1)$ and Lemma \ref{lem:exactproperty1}, we know $L_m$ is $\mathfrak{S}_\mathrm{devissage}(\mathcal{E})\cap\mathfrak{S}_\mathrm{1-fillability}(\mathcal{E})\cap\mathfrak{S}$-monomorphic. 
    
    On the other hand, since $\mathfrak{S}_\mathrm{devissage}(\mathcal{E})\cap\mathfrak{S}_\mathrm{1-fillability}(\mathcal{E})\cap\mathfrak{S}\subseteq \mathfrak{S}$ and $L_m$ is $\mathfrak{S}$-epimorphic, $L_m$ is $\mathfrak{S}_\mathrm{devissage}(\mathcal{E})\cap\mathfrak{S}_\mathrm{1-fillability}(\mathcal{E})\cap\mathfrak{S}$-epimorphic. 
    
    Hence, $L_m$ is $\mathfrak{S}_\mathrm{devissage}(\mathcal{E})\cap\mathfrak{S}_\mathrm{1-fillability}(\mathcal{E})\cap\mathfrak{S}$-invariant.
    \item Let $(\mathcal{F}:\mathcal{A}\to\mathcal{C})\in \mathfrak{S}_\mathrm{devissage}^w(\mathcal{E})\cap\mathfrak{S}_\mathrm{n-fillability}(\mathcal{E})\cap\mathfrak{S}$. Then $(\mathrm{Gap}^\lor_{(-)})^n\mathcal{F}\in\mathfrak{S}_\mathrm{devissage}(\mathcal{E})$. Since $L_m$ is $\mathfrak{S}$-epimorphic, by Lemma \ref{lem:inductionofquotient} and Theorem \ref{thm:quotient},
    \[
    L_m((\mathrm{Gap}^\lor)^n\mathcal{A})\to L_m((\mathrm{Gap}^\lor)^n\mathcal{A})\oplus\pi_m\Omega^{n-1}\operatorname{fib}(L(\mathcal{F}))
    \]
    is an epimorphism. Hence, $\pi_{m+n-1}\operatorname{fib}(L(\mathcal{F}))=0$, which means $L_{m+n}(\mathcal{F})$ is an epimorphism.
    \item Let $\mathcal{F}\in \mathfrak{S}_\mathrm{devissage}^w(\mathcal{E})\cap\mathfrak{S}_\mathrm{n-fillability}(\mathcal{E})\cap\mathfrak{S}\subseteq \mathfrak{S}_\mathrm{devissage}^w(\mathcal{E})\cap\mathfrak{S}_\mathrm{fillability}^k(\mathcal{E})\cap\mathfrak{S}$. Similar to $(3)$, we have $\pi_{m+k-1}\operatorname{fib}(L(\mathcal{F}))=0$ for all $1\leq k\leq n$. Hence, $L_{m+k}(F)$ are isomorphisms for all $1\leq k<n$.
\end{enumerate}
\end{proof}
\begin{cor}\label{cor:mainresult3.1}
   Suppose $\mathfrak{S}$ is a fillable class of some $\mathcal{E}$-linear functors and $L:\mathrm{Cat}^\mathrm{perf}(\mathcal{E})\to\mathrm{Sp}$ is a localizing invariant such that $L_m$ is $\mathfrak{S}$-epimorphic for some $m\in\mathbb{Z}$. Then $L_{m+n}$ is $\mathfrak{S}_\mathrm{devissage}^w(\mathcal{E})\cap\mathfrak{S}_\mathrm{fillability}(\mathcal{E})\cap (\mathrm{Gap}^\lor_{(-)})^{-1}\mathfrak{S}$-invariant for every $n\geq 1$.
\end{cor}
\begin{cor}\label{cor:mainresult3.2}
    Suppose $\mathfrak{S}$ is a fillable class of some $\mathcal{E}$-linear functors and $L:\mathrm{Cat}^\mathrm{perf}(\mathcal{E})\to\mathrm{Sp}$ is a localizing invariant such that $L_m$ is $\mathfrak{S}$-epimorphic for some $m\in\mathbb{Z}$. Then $L_{m+n}$ is $\mathfrak{S}_\mathrm{devissage}(\mathcal{E})\cap\mathfrak{S}_\mathrm{fillability}(\mathcal{E})\cap \mathfrak{S}$-invariant for every $n\geq 0$.
\end{cor}
\begin{cor}\label{cor:mainresult3.3}
    Suppose $\mathfrak{S}$ is a fillable class of some $\mathcal{E}$-linear functors and $L:\mathrm{Cat}^\mathrm{perf}(\mathcal{E})\to\mathrm{Sp}$ is a localizing invariant such that $L_m$ is $\mathfrak{S}$-epimorphic for all $m\in\mathbb{Z}$. Then $L_m$ is $\mathfrak{S}_\mathrm{devissage}(\mathcal{E})\cap\mathfrak{S}^1_\mathrm{fillability}(\mathcal{E})\cap \mathfrak{S}$-invariant for all $m\in\mathbb{Z}$.
\end{cor}
\begin{cor}\label{cor:mainresult3.4}
    Suppose
    \begin{itemize}
        \item $\mathfrak{S}$ is a fillable class of $\mathcal{E}'$-linear functors.
        \item $L':\mathrm{Cat}^\mathrm{perf}(\mathcal{E}')\to\mathrm{Sp}$ is a localizing invariant such that $L'_m$ is $\mathfrak{S}$-epimorphic.
        \item there is a commutative diagram
        \[
        \begin{tikzcd}
            \mathrm{Cat}^\mathrm{perf}(\mathcal{E})\ar[d,"f"]\ar[dr,"L"]\\
            \mathrm{Cat}^\mathrm{perf}(\mathcal{E}')\ar[r,"L'"']&\mathrm{Sp}
        \end{tikzcd}
    \]
    \end{itemize}
Then
\begin{enumerate}[label=(\arabic*)]
          \item $L_m$ is $f^{-1}(\mathfrak{S}_\mathrm{devissage}^w(\mathcal{E})\cap\mathfrak{S}_\mathrm{1-fillability}(\mathcal{E})\cap(\mathrm{Gap}^\lor_{(-)})^{-1}\mathfrak{S})$-monomorphic.
          \item $L_m$ is $f^{-1}(\mathfrak{S}_\mathrm{devissage}(\mathcal{E})\cap\mathfrak{S}_\mathrm{1-fillability}(\mathcal{E})\cap\mathfrak{S})$-invariant.
          \item $L_{m+n}$ is $f^{-1}(\mathfrak{S}_\mathrm{devissage}^w(\mathcal{E})\cap\mathfrak{S}_\mathrm{n-fillability}(\mathcal{E})\cap(\mathrm{Gap}^\lor_{(-)})^{-1}\mathfrak{S})$-epimorphic for every $n\geq 1$.
          \item $L_{m+k}$ is $f^{-1}(\mathfrak{S}_\mathrm{devissage}^w(\mathcal{E})\cap\mathfrak{S}_\mathrm{n-fillability}(\mathcal{E})\cap(\mathrm{Gap}^\lor_{(-)})^{-1}\mathfrak{S})$-invariant for every $0<k<n$.
      \end{enumerate}
\end{cor}
\begin{example}
     When $\mathcal{E}=\mathcal{E}'$, we can set $L=\Sigma\circ L'$ and $f:\mathcal{C}\mapsto\mathrm{Ind}(\mathcal{C})^{\omega_1}/\mathcal{C}$ or $L=\Omega\circ L'$ and $f$ is the functor constructed by \cite[Theorem 3.5]{Maxime_Vladimir_Christoph_2024}.
\end{example}
\begin{example}
    \cite[Theorem A]{Maxime_Vladimir_Christoph_2024} constructs a functor $\mathrm{Sp}\to\mathrm{Cat}^\mathrm{perf},E\mapsto \mathcal{C}_E$ such that $K(\mathcal{C}_E)\simeq E$, where $K:\mathrm{Cat}^\mathrm{perf}\to\mathrm{Sp}$ is the non-connective algebraic K-theory constructed by \cite{blumberg_gepner_tabuada_2013}. Hence, when $\mathcal{E}'=\mathrm{Sp}^\omega$ and $L'=K$, for any $L,\mathcal{E}$, we can find suitable $f$ and apply this theorem.
\end{example}
\begin{definition}
    Suppose $\mathfrak{S}$ is a class of exact functors in $\mathrm{Cat}^\mathrm{perf}$. Denote by $\mathfrak{S}(\mathcal{E})$ the class of all $\mathcal{E}$-linear functors in $\mathfrak{S}$.
\end{definition}
\begin{lemma}\label{lem:fillablesubclass}
    If $\mathfrak{S}$ is fillable, so is $\mathfrak{S}(\mathcal{E})$.
\end{lemma}
\begin{proof}
    Omitted.
\end{proof}
There are some simple examples of fillable classes.
\begin{example}
    The class of all equivalences in $\mathrm{Cat}^\mathrm{perf}(\mathcal{E})$ is fillable.
\end{example}
\begin{example}
    Let $\mathcal{F}$ be an $\mathcal{E}$-linear functor. Then $\{\mathcal{F},\mathrm{Gap}^\lor_\mathcal{F},\cdots,(\mathrm{Gap}^\lor_{(-)})^n\mathcal{F},\cdots\}$ is fillable.
\end{example}
\begin{example}
For trivial localizing invariant $L^\mathrm{trivial}:\mathrm{Cat}^\mathrm{perf}(\mathcal{E})\to\mathrm{Sp}$, which sends all $\mathcal{E}$-linear $\infty$-categories to zero objects in $\mathrm{Sp}$, $L_m^\mathrm{trivial}$ is $\mathfrak{S}_\mathrm{deivssage}(\mathcal{E})$-invariant for all $m\in\mathbb{Z}$.
\end{example}
\section{On the theorem of heart}
\subsection{t-structure}
t-structure is firstly introduced by \cite[Definition 1.3.1]{Deligne_1982}. \cite[Definition 1.2.1.1]{lurie_2017} gives the definition of t-structure on stable $\infty$-categories with homological indexing.

A $t$-structure on a stable category $\mathcal{C}$ is a pair of full subcategories $(\mathcal{C}_{\ge 0}, \mathcal{C}_{\le 0})$ such that $(h\mathcal{C}_{\ge 0}, h\mathcal{C}_{\le 0})$ is a $t$-structure on $h\mathcal{C}$. 

If $\mathcal{D}$ is another stable category with a $t$-structure $(\mathcal{D}_{\ge 0}, \mathcal{D}_{\le 0})$, then an exact functor $F : \mathcal{C} \to \mathcal{D}$ is called left $t$-exact if $F(\mathcal{C}_{\le 0}) \subset \mathcal{D}_{\le 0}$, and right $t$-exact if $F(\mathcal{C}_{\ge 0}) \subset \mathcal{D}_{\ge 0}$. Further, $F$ is called $t$-exact if it is both left and right $t$-exact.

Given a $t$-structure $(\mathcal{C}_{\ge 0}, \mathcal{C}_{\le 0})$ on a stable category $\mathcal{C}$, for any $a \in \mathbb{Z}$ we put
\[
\mathcal{C}_{\ge a} = \mathcal{C}_{\ge 0}[a], \quad \mathcal{C}_{\le a} = \mathcal{C}_{\le 0}[a].
\]
We denote by $\tau_{\ge a} : \mathcal{C} \to \mathcal{C}_{\ge a}$ the right adjoint to the inclusion, and by $\tau_{\le a} : \mathcal{C} \to \mathcal{C}_{\le a}$ the left adjoint to the inclusion. We also denote by the same symbols the compositions
\[
\mathcal{C} \xrightarrow{\tau_{\ge a}} \mathcal{C}_{\ge a} \hookrightarrow \mathcal{C}, \quad \mathcal{C} \xrightarrow{\tau_{\le a}} \mathcal{C}_{\le a} \hookrightarrow \mathcal{C}.
\]
For integers $a \leq b$ we put $\mathcal{C}_{[a,b]} = \mathcal{C}_{\ge a} \cap \mathcal{C}_{\le b}$. In particular, $\mathcal{C}_{[0,0]} = \mathcal{C}^\heartsuit$ is the heart of the $t$-structure, which is an abelian category. We denote by $\tau_{[a,b]} : \mathcal{C} \to \mathcal{C}_{[a,b]}$ the functor $\tau_{\ge a} \circ \tau_{\le b} \cong \tau_{\le b} \circ \tau_{\ge a}$. Denote
\[
\pi_a(x) = \tau_{[a,a]}(x)[-a] \in \mathcal{C}^\heartsuit, \quad a \in \mathbb{Z},\ x \in \mathcal{C}.
\]
We recall the notation
\[
\mathcal{C}^+ = \bigcup_{a \in \mathbb{Z}} \mathcal{C}_{\le a}, \quad \mathcal{C}^- = \bigcup_{a \in \mathbb{Z}} \mathcal{C}_{\ge a}, \quad \mathcal{C}^b = \mathcal{C}^+ \cap \mathcal{C}^- = \bigcup_{n \ge 0} \mathcal{C}_{[-n,n]}.
\]
The $t$-structure $(\mathcal{C}_{\ge 0}, \mathcal{C}_{\le 0})$ is called left bounded resp. right bounded resp. bounded if $\mathcal{C} = \mathcal{C}^+$ resp. $\mathcal{C} = \mathcal{C}^-$ resp. $\mathcal{C} = \mathcal{C}^b$. We use the notation from \cite{efimov2026}:
\begin{definition}
    A small \emph{$t$-category} is a small stable $\infty$-category $\mathcal{C}$ with a bounded $t$-structure $(\mathcal{C}_{\ge 0}, \mathcal{C}_{\le 0})$.
\end{definition}
\begin{definition}
    An $t$-exact functor $F:\mathcal{A}\to\mathcal{C}$ between small $t$-categories is called \emph{coconnective} if
        \begin{itemize}
            \item $F(\mathcal{A})$ generates $\mathcal{C}$ as a stable idempotent-complete $\infty$-categories(it satisfies the weak d\'evissage condition), and
            \item $F|_{\mathcal{A}^\heartsuit}:\mathcal{A}^\heartsuit\to\mathcal{C}^\heartsuit$ is fully faithful.
        \end{itemize}
\end{definition}
\begin{lemma}\label{lem:derived1}\cite[Proposition 1.51]{efimov2026}
    Let $\mathcal{F}:\mathcal{A}\to\mathcal{C}$ be a coconnective $t$-exact functor. Then $(\mathcal{C}^\heartsuit,\mathcal{A}^\heartsuit)$ satisfies Quillen's d\'evissage condition, i.e. $\mathcal{A}^\heartsuit$ is closed under subobjects and subquotients in $\mathcal{C}^\heartsuit$. As a consequence, $\mathcal{F}$ satisfies the d\'evissage condition.
\end{lemma}
\subsection{Some remarks on Barwick's theorem}
Recall the theorem of heart.
\begin{thm}\cite{Barwick_2015}
    Let $\mathcal{C}$ be a stable $\infty$-category equipped with a bounded t-structure. Then the natural inclusion $\mathcal{C}^\heartsuit\to\mathcal{C}$ induces isomoprhisms $K_n(\mathcal{C}^\heartsuit)\to K_n(\mathcal{C})$ for all $n\geq 0$.
\end{thm}
\begin{lemma}\label{lem:heart}
    Let $\mathcal{C}$ be a stable $\infty$-category equipped with a bounded t-structure. Then $\mathrm{Gap}^\lor(\mathcal{C})$ admits a bounded t-structure.
\end{lemma}
\begin{proof}
    Define full subcategories:
    \[
    \begin{aligned}
        \mathrm{Gap}^\lor(\mathcal{C})_{\geq 0}:=\{(X,Y)\in\mathrm{Gap}^\lor(\mathcal{C})|X_{0,n},Y_{0,n}\in\mathcal{C}_{\geq 0},\forall n\geq 0\}\\
        \mathrm{Gap}^\lor(\mathcal{C})_{\leq 0}:=\{(X,Y)\in\mathrm{Gap}^\lor(\mathcal{C})|X_{0,n},Y_{0,n}\in\mathcal{C}_{\leq 0},\forall n\geq 0\}.
    \end{aligned}
    \]
    \begin{itemize}
        \item $\mathrm{Gap}^\lor(\mathcal{C})_{\geq 1}\subseteq \mathrm{Gap}^\lor(\mathcal{C})_{\geq 0}$ and $\mathrm{Gap}^\lor(\mathcal{C})_{\leq 0}\subseteq \mathrm{Gap}^\lor(\mathcal{C})_{\leq 1}$ are obvious.
        \item let $(X_{\bullet,\bullet},Y_{\bullet,\bullet})\in \mathrm{Gap}^\lor(\mathcal{C})_{\geq 0}$ and $(X'_{\bullet,\bullet},Y'_{\bullet,\bullet})\in \mathrm{Gap}^\lor(\mathcal{C})_{\leq -1}$. To prove $\mathrm{Map}((X_{\bullet,\bullet},Y_{\bullet,\bullet}),(X'_{\bullet,\bullet},Y'_{\bullet,\bullet}))\simeq0$, we just need to prove $\mathrm{Map}(X_{\bullet,\bullet},X'_{\bullet,\bullet})\simeq0$. Since $\mathrm{Map}_\mathcal{C}(X_{0,n},X'_{0,n})\simeq0$ and $\mathrm{Map}_\mathcal{C}(X_{0,n}[1],X'_{0,n+1})\simeq 0$ for all $n\geq 0$, we know $\mathrm{Map}(X_{\bullet,\bullet},X'_{\bullet,\bullet})=0$.
        \item let $(X_{\bullet,\bullet},Y_{\bullet,\bullet})\in \mathrm{Gap}^\lor(\mathcal{C})$, define
        \[
        \begin{aligned}
            \tau_{\geq n}(X_{\bullet,\bullet},Y_{\bullet,\bullet})=(0\to\tau_{\geq n}X_{0,1}\to\tau_{\geq n}X_{0,2}\to\cdots,0\to\tau_{\geq n}Y_{0,1}\to\tau_{\geq n}Y_{0,2}\to\cdots)\\
            \tau_{\leq n}(X_{\bullet,\bullet},Y_{\bullet,\bullet})=(0\to\tau_{\leq n}X_{0,1}\to\tau_{\leq n}X_{0,2}\to\cdots,0\to\tau_{\leq n}Y_{0,1}\to\tau_{\leq n}Y_{0,2}\to\cdots).
        \end{aligned}
        \]
    \end{itemize}
\end{proof}
\begin{lemma}\label{lem:derived3}
    Let $\mathcal{F}:\mathcal{A}\to\mathcal{C}$ be a coconnective $t$-exact functor.
    \begin{enumerate}[label=(\arabic*)]
        \item For $X\in\mathcal{A}$, if $F(X)\in\mathcal{C}_{[a,b]}$, then $X\in\mathcal{A}_{[a,b]}$.
        \item For $X\in\mathcal{A}_{\geq 0}, Y\in\mathcal{A}_{\leq 0}$,
        \[
        \mathrm{Map}_{\mathcal{A}}(X,Y)\simeq\mathrm{Map}_{\mathcal{A}^\heartsuit}(\pi_0X,\pi_0Y)\simeq\mathrm{Map}_{\mathcal{C}^\heartsuit}(\pi_0\mathcal{F}(X),\pi_0\mathcal{F}(Y))\simeq\mathrm{Map}_\mathcal{C}(\mathcal{F}(X),\mathcal{F}(Y)).
        \]
        \item Let $X\to Y\to Z$ be a cofiber sequence in $\mathcal{C}$. Then there is a cofiber sequence
        \[
            \operatorname{cofib}(\tau_{\geq n}X\to\tau_{\geq n}Y)\to\tau_{\geq n}Z\to\pi_nZ/\operatorname{coker}(\pi_nX\to\pi_nY)[n].
        \]
    \end{enumerate}
\end{lemma}
\begin{proof}\quad
    \begin{enumerate}[label=(\arabic*)]
        \item It is because for every $n$, $\mathcal{F}(\pi_nX)\to \pi_n \mathcal{F}(X)$ is equivalence.
        \item Apply the (co)fiber sequence $\pi_0Y\to Y\to \tau_{\leq -1}Y$, we have fiber sequence
        \[
        \mathrm{Map}_{\mathcal{A}}(X,\pi_0Y)\to \mathrm{Map}_{\mathcal{A}}(X,Y)\to \mathrm{Map}_{\mathcal{A}}(X,\tau_{\leq -1}Y).
        \]
        Since $\mathrm{Map}_{\mathcal{A}}(X,\tau_{\leq -1}Y)$ is contractible, we have equivalence
        \[
        \mathrm{Map}_{\mathcal{A}}(X,\pi_0Y)\simeq \mathrm{Map}_{\mathcal{A}}(X,Y).
        \]
        Apply the (co)fiber sequence $\tau_{\geq 1}X\to X\to \pi_0X$, we have fiber sequence
        \[
        \mathrm{Map}_{\mathcal{A}}(\tau_{\geq 1}X,\pi_0Y)\to \mathrm{Map}_{\mathcal{A}}(X,\pi_0Y)\to \mathrm{Map}_{\mathcal{A}}(\pi_0X,\pi_0Y)
        \]
        Since $\mathrm{Map}_{\mathcal{A}}(\tau_{\geq 1}X,\pi_0Y)$ is contractible, we have equivalence
        \[
        \mathrm{Map}_{\mathcal{A}}(X,\pi_0Y)\simeq \mathrm{Map}_{\mathcal{A}^\heartsuit}(\pi_0X,\pi_0Y).
        \]
        Hence, $\mathrm{Map}_{\mathcal{A}}(X,Y)\simeq\mathrm{Map}_{\mathcal{A}^\heartsuit}(\pi_0X,\pi_0Y)$. Similarly,
        \[
        \mathrm{Map}_\mathcal{C}(\mathcal{F}(X),\mathcal{F}(Y))\simeq\mathrm{Map}_{\mathcal{C}^\heartsuit}(\pi_0\mathcal{F}(X),\pi_0\mathcal{F}(Y))
        \]
        Since $\mathcal{F}$ is coconnective, $\mathrm{Map}_{\mathcal{A}^\heartsuit}(\pi_0X,\pi_0Y)\simeq\mathrm{Map}_{\mathcal{C}^\heartsuit}(\pi_0\mathcal{F}(X),\pi_0\mathcal{F}(Y))$.
        \item Compare two long exact sequences
        \[
        \begin{tikzcd}
            \pi_i\tau_{\geq n}X\ar[r]\ar[d]&\pi_i\tau_{\geq n} Y\ar[r]\ar[d]&\pi_i\operatorname{cofib}(\tau_{\geq n}X\to \tau_{\geq n}Y)\ar[r]\ar[d]&\pi_{i-1}\tau_{\geq n}X\ar[r]\ar[d]&\pi_{i-1}\tau_{\geq n} Y\ar[d]\\
            \pi_iX\ar[r]&\pi_iY\ar[r]&\pi_iZ\ar[r]&\pi_{i-1}X\ar[r]&\pi_{i-1}Y
        \end{tikzcd}
        \]
        If $i>n$, by the five lemma, we know $\pi_i\operatorname{cofib}(\tau_{\geq n}X\to \tau_{\geq n}Y)\to\pi_iZ$ is equivalence.

        Since $\pi_i\tau_{\geq n}X=\pi_i\tau_{\geq n}Y=0$ when $i<n$, we know $\pi_i\operatorname{cofib}(\tau_{\geq n}X\to \tau_{\geq n}Y)=0$ for all $i<n$. Besides, $\pi_n\operatorname{cofib}(\tau_{\geq n}X\to \tau_{\geq n}Y)=\operatorname{coker}(\pi_nX\to\pi_nY)$ and $\pi_n\operatorname{cofib}(\tau_{\geq n}X\to \tau_{\geq n}Y)\to\pi_nZ$ is monomorphism.

        Hence,
        \[
        \operatorname{cofib}(\operatorname{cofib}(\tau_{\geq n}X\to \tau_{\geq n}Y)\to\tau_{\geq n}Z)\simeq\pi_nZ/\operatorname{coker}(\pi_nX\to\pi_nY)[n].
        \]
    \end{enumerate}
\end{proof}
\begin{prop}\label{prop:heart}
    Let $\mathcal{F}:\mathcal{A}\to\mathcal{C}$ be a coconnective $t$-exact functor. Then $\mathrm{Gap}^\lor(\mathcal{F})$ admits a bounded t-structure and $\mathrm{Gap}^\lor_\mathcal{F}$ is coconnective.
\end{prop}
\begin{proof}\quad
Define full subcategories
    \[
    \begin{aligned}
        \mathrm{Gap}^\lor(\mathcal{F})_{\geq 0}=\{(A_\bullet,X_{\bullet,\bullet},B_\bullet,Y_{\bullet,\bullet})\in\mathrm{Gap}^\lor(\mathcal{F})|X_{0,n},Y_{0,n}\in\mathcal{C}_{\geq 0},\forall n\geq 0\}\\
        \mathrm{Gap}^\lor(\mathcal{F})_{\leq 0}=\{(A_\bullet,X_{\bullet,\bullet},B_\bullet,Y_{\bullet,\bullet})\in\mathrm{Gap}^\lor(\mathcal{F})|X_{0,n},Y_{0,n}\in\mathcal{C}_{\leq 0},\forall n\geq 0\}
    \end{aligned}
    \]
    \begin{itemize}
        \item $\mathrm{Gap}^\lor(\mathcal{F})_{\geq 1}\subseteq \mathrm{Gap}^\lor(\mathcal{F})_{\geq 0}$ and $\mathrm{Gap}^\lor(\mathcal{F})_{\leq  0}\subseteq \mathrm{Gap}^\lor(\mathcal{F})_{\leq 1}$ are obvious.
        \item Let $(A_\bullet,X_{\bullet,\bullet},B_\bullet,Y_{\bullet,\bullet})\in\mathrm{Gap}^\lor(\mathcal{F})_{\geq 0}$ and $(A'_\bullet,X'_{\bullet,\bullet},B'_\bullet,Y'_{\bullet,\bullet})\in\mathrm{Gap}^\lor(\mathcal{F})_{\leq -1}$. To prove
        \[
        \mathrm{Map}((A_\bullet,X_{\bullet,\bullet},B_\bullet,Y_{\bullet,\bullet}),(A'_\bullet,X'_{\bullet,\bullet},B'_\bullet,Y'_{\bullet,\bullet}))=0,
        \]
        we just need to prove $\mathrm{Map}((A_\bullet,X_{\bullet,\bullet}),(A'_\bullet,X'_{\bullet,\bullet}))=0$. It is
        \[
        \mathrm{Map}((A_\bullet,X_{\bullet,\bullet}),(A'_\bullet,X'_{\bullet,\bullet}))\simeq\mathrm{Map}(X_{\bullet,\bullet},X'_{\bullet,\bullet})\prod_{\operatorname{colim}_n\Pi_{i=1}^n\mathrm{Map}_{\mathcal{C}}(\mathcal{F}(A_i),\mathcal{F}(A'_i))}\operatorname{colim}_n\Pi_{i=1}^n\mathrm{Map}_{\mathcal{A}}(A_i,A'_i)
        \]
        Since $\mathrm{Map}(X_{\bullet,\bullet},X'_{\bullet,\bullet})\simeq 0$, we only need to prove
        \[
        \mathrm{Map}_{\mathcal{A}}(A_n,A'_n)\to\mathrm{Map}_{\mathcal{C}}(\mathcal{F}(A_n),\mathcal{F}(A'_n))
        \]
        is equivalence for each $n\geq 1$.

        Since $\mathcal{F}(A_n)=\operatorname{cofib}(X_{0,n-1}\to X_{0,n})$ and $\mathcal{F}(A'_n)=\operatorname{cofib}(X'_{0,n-1}\to X'_{0,n})$, we know $\mathcal{F}(A_n)\in\mathcal{C}_{\geq 0},\mathcal{F}(A'_n)\in\mathcal{C}_{\leq 0}$. By Lemma \ref{lem:derived3}, we know $A_n\in\mathcal{A}_{\geq 0},A'_n\in\mathcal{A}_{\leq 0}$ and
        \[
        \mathrm{Map}_{\mathcal{A}}(A_n,A'_n)\simeq\mathrm{Map}_{\mathcal{A}^\heartsuit}(\pi_0A_n,\pi_0A'_n)\to\mathrm{Map}_{\mathcal{C}}(\mathcal{F}(A_n),\mathcal{F}(A'_n))\simeq\mathrm{Map}_{\mathcal{C}^\heartsuit}(\pi_0\mathcal{F}(A_n),\pi_0\mathcal{F}(A'_n))
        \]
        is equivalence.
        \item Let $(A_\bullet,X_{\bullet,\bullet},B_\bullet,Y_{\bullet,\bullet})\in\mathrm{Gap}^\lor(\mathcal{F})$. By Lemma \ref{lem:derived3}, there is a cofiber sequence
        \[
        \operatorname{cofib}(\tau_{\geq 0}X_{0,i}\to\tau_{\geq 0}X_{0,i+1})\to\tau_{\geq 0}\mathcal{F}(A_{i+1})\to\pi_0\mathcal{F}(A_{i+1})/\operatorname{coker}(\pi_0X_{0,i}\to\pi_0X_{0,i+1})
        \]
        for each $i\geq 0$. By Lemma \ref{lem:derived1}, the image of $\mathcal{A}^\heartsuit$ in $\mathcal{C}^\heartsuit$ is closed under subquotients. Hence, $\pi_0\mathcal{F}(A_{i+1})/\operatorname{coker}(\pi_0X_{0,i}\to\pi_0X_{0,i+1})$ has preimage. Hence, by Lemma \ref{lem:derived3},
        \[
        \begin{aligned}
            &\mathrm{Map}_\mathcal{C}(\tau_{\geq 0}\mathcal{F}(A_{i+1}),\pi_0\mathcal{F}(A_{i+1})/\operatorname{coker}(\pi_0X_{0,i}\to\pi_0X_{0,i+1}))\\
            \simeq&\mathrm{Map}_{\mathcal{A}}(\tau_{\geq 0}A_{i+1},\mathcal{F}^{-1}(\pi_0\mathcal{F}(A_{i+1})/\operatorname{coker}(\pi_0X_{0,i}\to\pi_0X_{0,i+1})))
        \end{aligned}
        \]
        Then, there is a morphism $\tau_{\geq 0}A_{i+1}\to \mathcal{F}^{-1}(\pi_0\mathcal{F}(A_{i+1})/\operatorname{coker}(\pi_0X_{0,i}\to\pi_0X_{0,i+1}))$ whose image under $\mathcal{F}$ is $\tau_{\geq 0}\mathcal{F}(A_{i+1})\to\pi_0\mathcal{F}(A_{i+1})/\operatorname{coker}(\pi_0X_{0,i}\to\pi_0X_{0,i+1})$. Thus,
        \[
        \mathcal{F}(\operatorname{fib}(\tau_{\geq 0}A_{i+1}\to \mathcal{F}^{-1}(\pi_0\mathcal{F}(A_{i+1})/\operatorname{coker}(\pi_0X_{0,i}\to\pi_0X_{0,i+1})))\simeq \operatorname{cofib}(\tau_{\geq 0}X_{0,i}\to\tau_{\geq 0}X_{0,i+1}).
        \]
        Similarly, there is a morphism $\tau_{\geq 0}B_{i+1}\to \mathcal{F}^{-1}(\pi_0\mathcal{F}(B_{i+1})/\operatorname{coker}(\pi_0Y_{0,i}\to\pi_0Y_{0,i+1}))$ such that
        \[
        \mathcal{F}(\operatorname{fib}(\tau_{\geq 0}B_{i+1}\to \mathcal{F}^{-1}(\pi_0\mathcal{F}(B_{i+1})/\operatorname{coker}(\pi_0Y_{0,i}\to\pi_0Y_{0,i+1})))\simeq \operatorname{cofib}(\tau_{\geq 0}Y_{0,i}\to\tau_{\geq 0}Y_{0,i+1})
        \]
        Define
        \[
        \tau_{\geq 0}(A_\bullet,X_{\bullet,\bullet},B_\bullet,Y_{\bullet,\bullet}):=\begin{cases}
            \operatorname{fib}(\tau_{\geq 0}A_{i+1}\to \mathcal{F}^{-1}(\pi_0\mathcal{F}(A_{i+1})/\operatorname{coker}(\pi_0X_{0,i}\to\pi_0X_{0,i+1}))\\
            0\to\tau_{\geq 0}X_{0,1}\to\tau_{\geq 0}X_{0,2}\to\cdots\\
            \operatorname{fib}(\tau_{\geq 0}B_{i+1}\to \mathcal{F}^{-1}(\pi_0\mathcal{F}(B_{i+1})/\operatorname{coker}(\pi_0Y_{0,i}\to\pi_0Y_{0,i+1}))\\
            0\to\tau_{\geq 0}Y_{0,1}\to\tau_{\geq 0}Y_{0,2}\to\cdots
        \end{cases}
        \]
        and $\tau_{\leq -1}(A_\bullet,X_{\bullet,\bullet},B_\bullet,Y_{\bullet,\bullet}):=\operatorname{cofib}(\tau_{\geq 0}(A_\bullet,X_{\bullet,\bullet},B_\bullet,Y_{\bullet,\bullet})\to (A_\bullet,X_{\bullet,\bullet},B_\bullet,Y_{\bullet,\bullet}))$.
    \end{itemize}

    Since $\mathrm{Gap}^\lor_{\mathcal{F}}$ always satisfies the weak d\'evissage condition by Lemma \ref{lem:fillable}, we only need to prove $\mathrm{Gap}^\lor(\mathcal{A})^\heartsuit\to\mathrm{Gap}^\lor(\mathcal{F})^\heartsuit$ is fully faithful. For $(X_{\bullet,\bullet},Y_{\bullet,\bullet}),(X'_{\bullet,\bullet},Y'_{\bullet,\bullet})\in \mathrm{Gap}^\lor(\mathcal{A})^\heartsuit$, since $\mathcal{A}^\heartsuit\to\mathcal{C}^\heartsuit$ is fully faithful,
    \[
    \begin{aligned}
        &\mathrm{Map}(\mathrm{Gap}^\lor_{\mathcal{F}}(X_{\bullet,\bullet},Y_{\bullet,\bullet}),\mathrm{Gap}^\lor_{\mathcal{F}}(X'_{\bullet,\bullet},Y'_{\bullet,\bullet}))\\
        \simeq&\mathrm{Map}(\mathrm{Gap}_\mathcal{F}(X_{\bullet,\bullet}),\mathrm{Gap}_\mathcal{F}(X'_{\bullet,\bullet}))\prod_{\mathrm{Map}_\mathcal{C}(\mathrm{ev}_\mathcal{F}(\mathrm{Gap}_\mathcal{F}(X_{\bullet,\bullet}))),\mathrm{ev}_\mathcal{F}(\mathrm{Gap}_\mathcal{F}(Y_{\bullet,\bullet})))} \mathrm{Map}(\mathrm{Gap}_\mathcal{F}(Y_{\bullet,\bullet}),\mathrm{Gap}_\mathcal{F}(Y'_{\bullet,\bullet}))\\
        \simeq&\mathrm{Map}(\mathrm{Gap}_\mathcal{F}(X_{\bullet,\bullet}),\mathrm{Gap}_\mathcal{F}(X'_{\bullet,\bullet}))\prod_{\mathrm{Map}_\mathcal{C}(\mathrm{ev}_\mathcal{A}(X_{\bullet,\bullet})),\mathrm{ev}_\mathcal{A}(Y_{\bullet,\bullet}))} \mathrm{Map}(\mathrm{Gap}_\mathcal{F}(Y_{\bullet,\bullet}),\mathrm{Gap}_\mathcal{F}(Y'_{\bullet,\bullet}))
    \end{aligned}
    \]
    Hence, we only need to prove $\mathrm{Map}(X_{\bullet,\bullet},X'_{\bullet,\bullet})\simeq\mathrm{Map}(\mathrm{Gap}_\mathcal{F}(X_{\bullet,\bullet}),\mathrm{Gap}_\mathcal{F}(X'_{\bullet,\bullet}))$. By Lemma \ref{lem:derived3},
    \[
    \begin{aligned}
        &\mathrm{Map}(\mathrm{Gap}_\mathcal{F}(X_{\bullet,\bullet}),\mathrm{Gap}_\mathcal{F}(X'_{\bullet,\bullet}))\\
        \simeq&\mathrm{Map}(\mathcal{F}(X_{\bullet,\bullet}),\mathcal{F}(X'_{\bullet,\bullet}))\prod_{\operatorname{colim}_n\Pi_{i=1}^n\mathrm{Map}_\mathcal{C}(\mathcal{F}(X_{i-1,i}),\mathcal{F}(X'_{i-1,i}))}\operatorname{colim}_n\Pi_{i=1}^n\mathrm{Map}_\mathcal{A}(X_{i-1,i},X'_{i-1,i})\\
        \simeq&\mathrm{Map}(X_{\bullet,\bullet},X'_{\bullet,\bullet}).
    \end{aligned}
    \]
\end{proof}
\begin{definition}
    Denote by $\mathfrak{S}_\mathrm{coconnective}$ the subclass of all coconnective $t$-exact functors $\mathcal{F}:\mathcal{A}\to\mathcal{C}$ between small $t$-categories.
\end{definition}
Proposition \ref{prop:heart} and Lemma \ref{lem:derived1} tells us:
\begin{cor}\label{cor:heart}
    $\mathfrak{S}_\mathrm{coconnective}(\mathcal{E})$ is a fillable class and $\mathfrak{S}_\mathrm{coconnective}(\mathcal{E})\subset\mathfrak{S}_\mathrm{fillability}(\mathcal{E})$. As a consequence, given a localizing $L:\mathrm{Cat}^\mathrm{perf}(\mathcal{E})\to\mathrm{Sp}$ and an integral $m\in\mathbb{Z}$, the following statements are equivalent:
    \begin{enumerate}[label=(\arabic*)]
        \item $L_m$ is $\mathfrak{S}_\mathrm{coconnective}(\mathcal{E})$-epimorphic, and
        \item $L_{m+n}$ is $\mathfrak{S}_\mathrm{coconnective}(\mathcal{E})$-invariant for all $n\geq 0$.
    \end{enumerate}
\end{cor}
\subsection{Some remarks on Efimov's theorem}
\begin{definition}\cite[Definition 4.2]{efimov2026}
    An $t$-exact functor $F:\mathcal{A}\to\mathcal{C}$ between small $t$-categories is called \emph{$(-n)$-coconnective} if
        \begin{itemize}
            \item $F(\mathcal{A})$ generates $\mathcal{C}$ as a stable idempotent-complete $\infty$-categories(it satisfies the weak d\'evissage condition), and
            \item for $x,y\in\mathcal{A}^\heartsuit$, the map $\mathrm{Ext}^i_\mathcal{A}(x,y)\to\mathrm{Ext}^i_\mathcal{C}(\mathcal{F}(x),\mathcal{F}(y))$ is an isomorphism for $i\leq n-1$, and a monomorphism for $i=n$.
        \end{itemize}
\end{definition}
\cite{efimov2026} proves a stronger version of the theorem of heart:
\begin{thm}\cite[Theorem 0.3]{efimov2026}
    Let $F:\mathcal{A}\to\mathcal{C}$ be a $(-n)$-coconnective functor between small $t$-categories. Then the induces map $K_j(\mathcal{A})\to K_j(\mathcal{C})$ is an isomorphism for $j\geq -n$, and a monomorphism for $j=-n-1$.
\end{thm}
\begin{lemma}\label{lem:extgroup}
    Let $F:\mathcal{A}\to\mathcal{C}$ be a $(-n)$-coconnective functor between small $t$-categories. Then for $X,Y\in\mathrm{Gap}^\lor(\mathcal{A})^\heartsuit$,
    \[
    \mathrm{Ext}^i_{\mathrm{Gap}^\lor(\mathcal{A})}(X,Y)\to \mathrm{Ext}^i_{\mathrm{Gap}^\lor(\mathcal{C})}(X,Y)
    \]
    is an isomorphism for $i\leq n-1$  and a monomorphism for $i=n$
\end{lemma}
\begin{proof}
    Since $\mathrm{Gap}^\lor(\mathcal{A})^\heartsuit\simeq\operatorname{colim}_n\mathrm{Fun}([n]\sqcup_{[0]}[n],\mathcal{A}^\heartsuit)$, we can assume $X,Y\in \mathrm{Fun}([n]\sqcup_{[0]}[n],\mathcal{A}^\heartsuit)$ for some $n$. Denote $K:=[n]\sqcup_{[0]}[n]$. Then the mapping spectra are
    \[
    \begin{aligned}
        \mathrm{Map}^\mathrm{Sp}_{\mathrm{Gap}^\lor(\mathcal{A})}(X,Y)\simeq\operatorname{lim}_{(i,j)\in\mathrm{tw}(K)}\mathrm{Map}^\mathrm{Sp}_\mathcal{A}(X(i),Y(j))\\
        \mathrm{Map}^\mathrm{Sp}_{\mathrm{Gap}^\lor(\mathcal{C})}(X,Y)\simeq\operatorname{lim}_{(i,j)\in\mathrm{tw}(K)}\mathrm{Map}^\mathrm{Sp}_\mathcal{C}(X(i),Y(j)).
    \end{aligned}
    \]
    By assumption, we know
    \[
    \begin{cases}
        \pi_k\mathrm{Map}^\mathrm{Sp}_\mathcal{A}(X(i),Y(j))\simeq\pi_k\mathrm{Map}^\mathrm{Sp}_\mathcal{C}(X(i),Y(j))=0,\quad k>0\\
        \pi_k\mathrm{Map}^\mathrm{Sp}_\mathcal{A}(X(i),Y(j))\simeq\pi_k\mathrm{Map}^\mathrm{Sp}_\mathcal{C}(X(i),Y(j)),\quad -n+1\leq k\leq 0\\
        \pi_{-n}\mathrm{Map}^\mathrm{Sp}_\mathcal{A}(X(i),Y(j))\to\pi_{-n}\mathrm{Map}^\mathrm{Sp}_\mathcal{C}(X(i),Y(j)) \text{ is a monomorphism}
    \end{cases}
    \]
    for each $i,j$.

    Hence, $\pi_{-i}\mathrm{Map}^\mathrm{Sp}_\mathcal{A}(X(i),Y(j))\to\pi_{-i}\mathrm{Map}^\mathrm{Sp}_\mathcal{C}(X(i),Y(j))$ is an isomorphism for $i\leq n-1$  and a monomorphism for $i=n$.
\end{proof}
\begin{prop}\label{prop:extgroup}
    Let $F:\mathcal{A}\to\mathcal{C}$ be a $(-n)$-coconnective functor between small $t$-categories. Then for $X,Y\in\mathrm{Gap}^\lor(\mathcal{A})^\heartsuit$,
    \[
    \mathrm{Ext}^i_{\mathrm{Gap}^\lor(\mathcal{A})}(X,Y)\to \mathrm{Ext}^i_{\mathrm{Gap}^\lor(\mathcal{F})}(X,Y)
    \]
    is an isomorphism for $i\leq n-1$  and a monomorphism for $i=n$.
\end{prop}
\begin{proof}
    The functor $\mathrm{gr}:\mathrm{Gap}(\mathcal{C})\to\operatorname{colim}_n\mathcal{C}^n$ induces funtor $\mathrm{gr}^\lor:\mathrm{Gap}^\lor(\mathcal{C})\to(\operatorname{colim}_n\mathcal{C}^n)^2$. Then
    \[
    \begin{tikzcd}
        \mathrm{Gap}^\lor(\mathcal{F})\ar[r]\ar[d]&\mathrm{Gap}^\lor(\mathcal{C})\ar[d]\\
        (\operatorname{colim}_n\mathcal{A}^n)^2\ar[r]&(\operatorname{colim}_n\mathcal{C}^n)^2
    \end{tikzcd}
    \]
    is pullback of $\infty$-categories. Now consider the long exacct sequence
    \[
    \begin{aligned}
        \cdots&\to\pi_{k+1}\mathrm{Map}^\mathrm{Sp}_{\mathrm{Gap}^\lor(\mathcal{C})}(X,Y)\oplus\pi_{k+1}\mathrm{Map}^\mathrm{Sp}_{(\operatorname{colim}_n\mathcal{A}^n)^2}(\mathrm{gr}^\lor(X),\mathrm{gr}^\lor(Y))\\
        &\to \pi_{k+1}\mathrm{Map}^\mathrm{Sp}_{(\operatorname{colim}_n\mathcal{C}^n)^2}(\mathrm{gr}^\lor(X),\mathrm{gr}^\lor(Y))\to\pi_k\mathrm{Map}^\mathrm{Sp}_{\mathrm{Gap}^\lor(\mathcal{F})}(X,Y)\\
        &\to \pi_k\mathrm{Map}^\mathrm{Sp}_{\mathrm{Gap}^\lor(\mathcal{C})}(X,Y)\oplus\pi_k\mathrm{Map}^\mathrm{Sp}_{(\operatorname{colim}_n\mathcal{A}^n)^2}(\mathrm{gr}^\lor(X),\mathrm{gr}^\lor(Y))\\
        &\to \pi_k\mathrm{Map}^\mathrm{Sp}_{(\operatorname{colim}_n\mathcal{C}^n)^2}(\mathrm{gr}^\lor(X),\mathrm{gr}^\lor(Y))\to\cdots
    \end{aligned}
    \]
    Since $\pi_k\mathrm{Map}^\mathrm{Sp}_{(\operatorname{colim}_n\mathcal{A}^n)^2}(\mathrm{gr}^\lor(X),\mathrm{gr}^\lor(Y))\to \pi_k\mathrm{Map}^\mathrm{Sp}_{(\operatorname{colim}_n\mathcal{C}^n)^2}(\mathrm{gr}^\lor(X),\mathrm{gr}^\lor(Y))$ is isomorphism for $k\geq -n+1$ and is monomorphism fo $k=-n$. Hence,
    \[
    \pi_k\mathrm{Map}^\mathrm{Sp}_{\mathrm{Gap}^\lor(\mathcal{F})}(X,Y)\to \pi_k\mathrm{Map}^\mathrm{Sp}_{\mathrm{Gap}^\lor(\mathcal{C})}(X,Y)
    \]
     is isomorphism for $k\geq -n+1$ and is monomorphism fo $k=-n$.

     Now consider the commutative diagram
     \[
     \begin{tikzcd}
         \pi_{-i}\mathrm{Map}^\mathrm{Sp}_{\mathrm{Gap}^\lor(\mathcal{A})}(X,Y)\ar[dr,"h_i"]\ar[d,"g_i"]&\\
         \pi_{-i}\mathrm{Map}^\mathrm{Sp}_{\mathrm{Gap}^\lor(\mathcal{F})}(X,Y)\ar[r,"f_i"]&\pi_{-i}\mathrm{Map}^\mathrm{Sp}_{\mathrm{Gap}^\lor(\mathcal{C})}(X,Y)
     \end{tikzcd}
     \]
     Since $h_i,f_i$ are isomorphisms for $i\leq n-1$ and are monomorphisms $i=n$, so is $g_i$.
\end{proof}
\begin{definition}
    Denote by $\mathfrak{S}^{-n}_\mathrm{coconnective}$ the subclass of all $(-n)$-coconnective $t$-exact functors $\mathcal{F}:\mathcal{A}\to\mathcal{C}$ between small $t$-categories.
\end{definition}
\begin{cor}
    $\mathfrak{S}^{-n}_\mathrm{coconnective}(\mathcal{E})$ is a fillable class and $\mathfrak{S}^{-n}_\mathrm{coconnective}(\mathcal{E})\subset\mathfrak{S}_\mathrm{fillability}(\mathcal{E})$. As a consequence, given a localizing $L:\mathrm{Cat}^\mathrm{perf}(\mathcal{E})\to\mathrm{Sp}$ and an integral $m\in\mathbb{Z}$, the following statements are equivalent:
    \begin{enumerate}[label=(\arabic*)]
        \item $L_m$ is $\mathfrak{S}_\mathrm{coconnective}(\mathcal{E})$-epimorphic, and
        \item $L_{m+n}$ is $\mathfrak{S}_\mathrm{coconnective}(\mathcal{E})$-invariant for all $n\geq 0$.
    \end{enumerate}
\end{cor}
\section*{Acknowledgments}

I would like to express my sincere gratitude to everyone who supported me throughout the process of completing this article. First and foremost, I am deeply thankful to my supervisor, Professor Christian Haesemeyer, for his invaluable guidance, encouragement, and insightful feedback. His expertise and patience were instrumental in shaping this work. My appreciation extends to my colleagues and friends in the School of Mathematics and Statistics, the University of Melbourne, who provided a supportive and inspiring environment.

\bibliography{refbase}

\end{document}